\documentclass[reqno]{amsart}
\usepackage{amssymb,amsmath,epsfig,graphics,mathrsfs}

\usepackage{todonotes}

\usepackage{fancyhdr}
\usepackage{hyperref}
\hypersetup{
	colorlinks   = true,
	urlcolor     = blue,
	linkcolor    = blue,
	citecolor   = red ,
	bookmarksopen=true
}

\usepackage{dsfont}
\usepackage{cancel} 
\usepackage{esvect}
\usepackage{tikz}
\usepackage{tikz-3dplot}
\usepackage{pgfplots}
\usepackage{pgfmath}

\pgfplotsset{compat=1.18}

\usepackage{multicol}%%%%%%%%%%%%%%%%%%%%%%%%%%%%%%%%%%%%%%%%%%%%%%%%%%%%%%%%%%%%%%%%%%%%%%%%%%%%%%%%%%%%%%%%%%%%%%%%%%%%%%%%%%%%%%%%%%%%%%%%%%%%%%multicol

\usepackage{amssymb,amsmath,epsfig,graphics,mathrsfs}

\usepackage{graphicx}
\usepackage{caption}
\usepackage{bbm}
\usepackage[normalem]{ulem}

\usepackage[a4paper,
left=1in,
right=1in,
top=1in,
bottom=1in,
footskip=.25in]{geometry}
\usepackage{hyperref}
\hypersetup{
	colorlinks   = true,
	urlcolor     = blue,
	linkcolor    = blue,
	citecolor   = red ,
	bookmarksopen=true
}

\numberwithin{equation}{section}

\newcommand{\beq}{\begin{equation}}
	\newcommand{\bea}[1]{\begin{array}{#1} }
		\newcommand{\eeq}{ \end{equation}}
	\newcommand{\ea}{ \end{array}}

\newtheorem{theorem}{Theorem}[section]
\newtheorem{lemma}[theorem]{Lemma}
\newtheorem{proposition}[theorem]{Proposition}
\newtheorem{corollary}[theorem]{Corollary}
\newtheorem{remark}[theorem]{Remark}

\makeatletter
\def\@settitle{\begin{center}%
		\baselineskip14\p@\relax
		\bfseries
		\uppercasenonmath\@title
		\@title
		\ifx\@subtitle\@empty\else
		\\[1ex]\uppercasenonmath\@subtitle
		\footnotesize\mdseries\@subtitle
		\fi
	\end{center}%
}
\def\subtitle#1{\gdef\@subtitle{#1}}
\def\@subtitle{}
\makeatother

\usepackage[section]{placeins}

\usepackage{pgfplots}

\begin{document}
	\title[Finite-time blow-up for the three dimensional axially symmetric Keller-Segel system]{Finite-time blow-up for the three dimensional axially symmetric Keller-Segel system} 
	%\date{\today}%, file: \textsf{KS}}
\begin{abstract}
We construct axially symmetric finite-time blow-up solutions to the three-dimensional Keller–Segel system. By adapting gluing techniques, we derive a precise asymptotic expansion for Type II singularities that generalizes the recent work of Hou, Nguyen, and Song \cite{HNS}. In our construction the mass concentrates along multiple rings and we obtain a refined expansion for the blow-up rate.

\end{abstract}

\author[ F.~Buseghin]{Federico Buseghin}
\address{\noindent F.~Buseghin: Laboratoire AGM\\ CY Cergy Paris Universit\'e\\ 2 avenue Adolphe Chauvin, 95300 Pontoise, France}
\email{federico.buseghin@cyu.fr}

\author[J.~D\'avila]{Juan D\'avila}
\address{\noindent J.~D\'avila: Department of Mathematical Sciences University of Bath, Bath BA2 7AY, United Kingdom}
\email{jddb22@bath.ac.uk}

\author[M.~del Pino]{Manuel del Pino}
\address{\noindent M.~del Pino: Department of Mathematical Sciences University of Bath, Bath BA2 7AY, United Kingdom}
\email{mdp59@bath.ac.uk}

\author[M.~Musso]{Monica Musso}
\address{\noindent M.~Musso: Department of Mathematical Sciences University of Bath, Bath BA2 7AY, United Kingdom.}
\email{mm2683@bath.ac.uk}
\maketitle

\section{Introduction}
We study the three dimensional Keller-Segel system given by
\begin{align}\label{KS}
	\begin{cases}
		\partial_{t}u=\Delta u - \nabla \cdot(u\nabla v), \quad \text{in } \mathbb{R}^{3}\times (0,T),\\
		v=(-\Delta_{\mathbb{R}^{3}})u:=\frac{1}{4\pi}\int_{\mathbb{R}^{3}}\frac{1}{|x-y|}u(y,t)dy,\\
		u(\cdot,0)=u_{0}\ge0, \quad \text{in } \mathbb{R}^{3}.
	\end{cases}
\end{align}
The Keller-Segel system models \emph{chemotaxis}, a biological phenomenon describing the movement of organisms in response to chemical signals, for instance, the motion of slime mold \emph{Dictyostelium Discoideum} and the bacterium \emph{Escherichia Coli}. In \eqref{KS} $u(x,t)$ indicates the population density, while $v(x,t)$ represents the density of the external chemical stimulus usually called \emph{chemoattractant}. The model was originally introduced by Patlak \cite{Pat} and further developed by Keller and Segel \cite{KS}. For surveys and related mathematical developments, see \cite{H0}, \cite{Ch}. Following the pioneering contribution of W. J\"{a}ger and S. Luckhaus in two dimensions \cite{JL}, the topic has attracted considerable attention and has been extensively studied. In three dimensions it can be interpreted as a model of stellar dynamics influenced by friction and fluctuations.
Due to its divergence form, strong solutions to the Keller-Segel system conserve total mass:
\begin{align*}
	\int_{\mathbb{R}^{3}}u(x,t)\,dx=\int_{\mathbb{R}^{3}}u_{0}(x)\,dx=:M.
\end{align*}
Moreover, the system admits an important scaling symmetry. If $u(x,t)$ is a solution, then so is 
\[
u_{\lambda}(x,t)=\lambda^{2}u(\lambda x,\lambda^{2}t), \quad \text{for any } \lambda>0.
\]
We say that a solution blows up at finite time $T$ if 
\begin{align*}
	\limsup_{t\to T} \|u(t)\|_{L^{\infty}(\mathbb{R}^{3})}=+\infty.
\end{align*}
The blow-up is said to be of \emph{type I} if there exists a constant $C>0$ such that
\begin{align*}
	\limsup_{t\to T}(T-t)\|u(t)\|_{L^{\infty}(\mathbb{R}^{3})}\le C.
\end{align*}
Otherwise, it is classified as a \emph{type II} blow-up. The goal of this work is to construct a class of finite-time type II blow-up solutions to \eqref{KS}, in which mass concentrates along an arbitrary number of rings. This construction generalizes the recent work by Hou, Nguyen, and Song \cite{HNS}. \newline
There exists a rich literature on both the well-posedness and singularity formation in the Keller-Segel system. In two dimensions, the system is called $L^{1}$-critical, as the scaling $u\mapsto u_{\lambda}$ preserves the $L^{1}$-norm. In higher dimensions, the problem becomes $L^{1}$-supercritical, with scaling preserving the $L^{n/2}$-norm. \newline
In two dimensions, the chemoattractant $v$ in \eqref{KS} is determined via convolution with the Newtonian potential kernel $\Psi(x)=-\frac{1}{2\pi}\log |x|$ and a fundamental role is played by the stationary state
\begin{align*}
	U(y):=\frac{8}{(1+|y|^{2})^{2}}.
\end{align*}
It is known that $M=8\pi$ is the critical mass threshold between global existence and finite-time blow-up. When $M<8\pi$, global solutions exist and decay to zero \cite{BDP, BKLN}. At $M=8\pi$, both global regularity and infinite-time blow-up may occur \cite{BCM, DdPDMW, BCC}. For $M>8\pi$, various examples of finite-time blow-up solutions have been constructed. Type II blow-up profiles in the radial setting were obtained by Herrero and Vel\'azquez in \cite{HV3} (see also \cite{Ve, Ve1, Ve2} for stability analysis), and later by Rapha\"el and Schweyer \cite{RS}. 
More recently, Collot, Ghoul, Masmoudi, and Nguyen \cite{CGMN1, CGMN2} refined the results in \cite{RS} by constructing stable nonradial blow-up dynamics, where the solution satisfy
\begin{align}\label{finitetimeDyn}
	u(x,t)\approx \frac{1}{\lambda^{2}(t)}U\left(\frac{x}{\lambda(t)}\right), \quad \lambda(t)\approx 2e^{-\frac{\gamma+2}{2}}\sqrt{T-t}e^{-\sqrt{\frac{|\ln(T-t)|}{2}}},
\end{align}
with $\gamma=0.57721\ldots$ denoting the Euler-Mascheroni constant. In \cite{CGMN2}, they also constructed additional blow-up solutions with different, unstable blow-up rates. Mizoguchi \cite{Mi} later proved that \eqref{finitetimeDyn} is the only stable blow-up behavior among radial nonnegative solutions.

Given $k$ distinct points $\xi_{1}, \xi_{2},\ldots,\xi_{k}\in \mathbb{R}^{2}$, the authors in \cite{BDdPM} used gluing techniques to construct solutions satisfying the expansion
\[
u(x,t)\approx\sum_{j=1}^{k}\frac{1}{\lambda_{j}^{2}(t)}U\left(\frac{x-\xi_{j}}{\lambda_{j}(t)}\right), \quad \lambda_{j}(t) \approx 2e^{-\frac{\gamma+2}{2}}\sqrt{T-t}e^{-\sqrt{\frac{|\ln(T-t)|}{2}}}.
\]
In this paper, we aim to extend this construction to the axisymmetric three dimensional case. \newline 
Similar to the two dimensional case, in three dimensions there exists a threshold on $\|u_{0}\|_{L^{3/2}}$ that delineates global existence from finite-time blow-up. For small initial data, global solutions exist; see \cite{CPZ, CCE}. Unlike in two dimensions, higher-dimensional cases admit type I blow-up solutions \cite{HMV, S, NNZ, CZ, LiZh}. Type II radial blow-up via collapsing spheres has been rigorously established in \cite{CGMN3}. \newline 
\underline{The three-dimensional axially symmetric system:}  
Assume that the initial condition $u_0$ is axially symmetric, that is, $u_0(x) = u_0(r, z)$ with $x = (r e^{i\theta}, z)$. In this case, the system \eqref{KS} can be rewritten in the form
\begin{align}\label{KS3d}
	\begin{cases}
		\partial_{t} u = \Delta_{(r,z)} u - \nabla_{(r,z)} v \cdot \nabla_{(r,z)} u + u^2, & \text{in } \mathbb{R}^{+} \times \mathbb{R} \times (0,T), \\[1ex]
		- \Delta_{(r,z)} v - \dfrac{1}{r} v = u, & \\[1ex]
		u(\cdot, 0) = u_0, & \text{in } \mathbb{R}^{+} \times \mathbb{R},
	\end{cases}
\end{align}
where, in the second equation, the potential $v$ is understood as the solution obtained via the inverse Laplacian, which we always take to be the convolution with with the Newtonian potential kernel $\Psi(x)=\frac{1}{4\pi |x|}$.
In the axisymmetric case, Hou, Nguyen, and Song \cite{HNS} constructed solutions satisfying
\[
u(r,z,t)\approx\frac{1}{\lambda^{2}(t)} U\left(\frac{r-q}{\lambda(t)},\frac{z}{\lambda(t)}\right), \quad \lambda(t) \approx \sqrt{T-t}e^{-\sqrt{\frac{|\ln(T-t)|}{2}} + O(1)},
\]
for some $q>0$ and where $r=\sqrt{x^{2}+y^{2}}$. Their approach is inspired by the techniques developed in \cite{CGMN1, CGMN2}. The goal of the present work is to generalize the result in \cite{HNS} by employing the gluing techniques introduced in \cite{BDdPM}. \newline
The following is our main result. 
\begin{theorem}\label{Mthm}
	Given any $k$ distinct points $(r_{1},z_{1}), (r_{2},z_{2}),...,(r_{k},z_{k}) \in \mathbb{R}^{+}\times \mathbb{R}$, there is an initial data $u_{0}$ of \eqref{KS3d} for which the corresponding solution $u(x,t)$ has the form 
	\begin{align*}
		u(r,z,t)=\sum_{j=1}^{k}\frac{1}{\lambda_{j}^{2}(t)} U\big(\frac{(r,z)-(r_{j}(t),z_{j}(t))}{\lambda_{j}(t)} \big)(1+o(1)), \ \ \ U(y)=\frac{8}{(1+|y|^{2})^{2}},
	\end{align*} 
	uniformly on bounded sets of $\mathbb{R}^{+}\times \mathbb{R}$, with $\lambda_{j}(t)\approx 2e^{-\frac{\gamma +2}{2}}\sqrt{T-t}e^{-\sqrt{\frac{|\ln(T-t)|}{2}}}$ where $\gamma$ is the Euler-Mascheroni constant and $(r_{j}(t),z_{j}(t))\to (r_{j},z_{j})$.
\end{theorem}
We observe that our result extends the construction in \cite{HNS} to multiple rings. Moreover we are able to provide the refined asymptotics of the rate of the blow-up that is consistent with the two dimensional constructions in \cite{
	CGMN1, CGMN2, BDdPM}. \newline
The construction here differs from the one in \cite{BDdPM} in two main technical aspects: 
\begin{itemize}
	\item In Section \ref{ChemoExSec}, we derive a highly precise asymptotic expansion for the three-dimensional axially symmetric inverse Laplacian of $U$, relying on asymptotic expansions of certain elliptic integrals. This expansion is expected to be of relevance in other contexts where it plays a central role, such as the study of leapfrogging vortex rings in the setting of the Euler equations; see \cite{DdPMW1};
	\item In Section \ref{outerTheorySec}, where we develop what is commonly referred to in the gluing procedure as the \emph{outer theory}, we introduce an intermediate region in our estimates to handle the transition from an essentially two dimensional problem to the full three dimensional equation. In this context, it is crucial to have precise control of the three dimensional inverse Laplacian of $U$, as established in Section \ref{ChemoEx}.		
\end{itemize}
The structure of the paper naturally follows the framework developed in \cite{BDdPM}. In Section \ref{ChemoExSec}, we construct the first approximation, presenting classical expansions of elliptic integrals and deriving the precise asymptotic behavior of the three-dimensional axially symmetric inverse Laplacian of $U$. In Section \ref{3dphilambdaSec}, we introduce the second term of the expansion, following the approach in \cite{BDdPM}, with the distinction that here the approximation is considered only locally. Section \ref{innerouterSec} is devoted to the introduction of the \emph{inner-outer} gluing scheme. We briefly recall the inner theories from \cite{BDdPM}, which are essential to solving the inner problem in appropriately chosen normed spaces. In Section \ref{outerTheorySec}, we develop the revisited outer theory required for our analysis. Finally, in Section \ref{proofThm1Sec}, we implement a fixed point argument to complete the proof of the main theorem.

\subsection*{Acknowledgements.} The work of F.~Buseghin, which began during his PhD studies, is part of the ERC Starting Grant project FloWAS that has received funding from the European Research Council (ERC) under the European Union's Horizon 2020 research and innovation programme (Grant Agreement No. 101117820).  
J.~D\'avila has been supported  by  a Royal Society Wolfson Fellowship, UK, Grant RSWF/FT/191007. 
M.~del Pino has been supported by the Royal Society Research Professorship grant RP-R1-180114 and by the ERC/UKRI Horizon Europe grant ASYMEVOL, EP/Z000394/1.

\section{The first approximation $u_{0}$}\label{ChemoExSec}
Since our gluing scheme naturally extends to the multi-ring case with only minor technical adjustments, we choose to keep the exposition concise and defer the details to Section \ref{multiringsproof}.
We now focus exclusively on the scenario involving a single ring. In this section, we aim to construct an initial, basic approximation of a solution to the three dimensional, axially symmetric Keller-Segel system. Consider the parameter functions
\begin{align*}
	0<\lambda(t)\to 0, \quad \xi(t):=(q_{1}(t),0)\to \xi(T):= (r^{\star},0) \in \mathbb{R}^{+}\times \mathbb{R}, \quad \alpha(t)\to 1 \quad \text{as } t\to T
\end{align*}
which we will specify in more detail later. We anticipate that the parameter $T>0$ plays a crucial role in ensuring the success of our argument. Indeed, throughout the construction of various barriers and in the final step of proving the main theorem, that will be established via a fixed point argument, it is essential to choose $T$ sufficiently small.
We also define a smooth cut-off function $\chi$, where
\[
\chi(r,z,t)=\chi_{0}\left(\frac{|(r,z)-\xi(t)|}{\sqrt{\delta(T-t)}}\right),
\]
with $\chi_{0} : \mathbb{R} \to \mathbb{R} $ being a smooth function satisfying
\begin{align}\label{cutoff}
	\chi_{0}(s)=
	\begin{cases}
		1 & \text{if } s\le 1,\\
		0 & \text{if } s\ge 2.
	\end{cases}
\end{align}
Here, $\delta$ is a small positive constant which will be fixed later in Section~\ref{reviewInnerTheories}. Now, define the functions
\begin{align*}
	U(y)=\frac{8}{(1+|y|^{2})^{2}}, \quad \Gamma_{0}(y)=\log U(y),
\end{align*}
noting that this pair satisfies the steady-state system
\begin{align*}
	\nabla \cdot ( \nabla U - U \nabla \Gamma_{0})=0, \quad \Delta \Gamma_{0}=U.
\end{align*}
The approximate solution is given by
\begin{equation}\label{def-u0-v0}
\begin{aligned}
	u_{0}(r,z,t)=\frac{\alpha(t)}{\lambda^{2}}U\left(\frac{(r,z)-\xi(t)}{\lambda}\right)\chi_{0}\left(\frac{|(r,z)-\xi(t)|}{\sqrt{\delta(T-t)}}\right), \quad 
	v_{0}(r,z,t)=(-\Delta_{\mathbb{R}^{3}})^{-1}u_{0}(r,z,t).
\end{aligned}
\end{equation}
Next, we define the error operator
\begin{align}\label{erroroperator}
	S(u)=-\partial_{t}u+\Delta u- \nabla \cdot(u\nabla v), \quad \text{where } v=(-\Delta_{\mathbb{R}^{3}})^{-1}u.
\end{align}
In this context, we can decompose
\begin{align*}
	S(u_{0}) = S_{(2d)}(u_{0}) + \frac{1}{r}\partial_{r}u_{0},
\end{align*}
where the two dimensional component is given by
\begin{align}\label{S2dERR}
	S_{(2d)}(u_{0}) &= -\frac{\dot{\alpha}}{\lambda^{2}}U\chi + \alpha\frac{\dot{\lambda}}{\lambda^{3}}Z_{0}\chi + \frac{\alpha}{\lambda^{3}}\dot{\xi} \cdot \nabla_{y}U\chi \nonumber\\
	&\quad -\frac{\alpha}{\lambda^{2}}U\partial_{t}\chi +\frac{2\alpha}{\lambda^{3}}\nabla_{(r,z)}\chi_ \cdot \nabla_{y}U + \frac{\alpha}{\lambda^{2}}U\Delta_{(r,z)}\chi \nonumber\\
	&\quad - \frac{1}{\lambda^{2}}U\nabla_{(r,z)}\chi \cdot \nabla_{(r,z)}v_{0}  - \frac{\alpha(\alpha-1)}{\lambda^{4}}\nabla_{y} \cdot (U \nabla_{y}\Gamma_{0}) \chi\nonumber\\
	&\quad + \frac{\alpha }{\lambda^{4}}\big[\alpha(\chi - 1)U^{2} - \nabla_{y}U \cdot \nabla_{y}\mathcal{R}\big]\chi- \frac{\alpha - 1}{\lambda^{2}}U(y)\nabla_{(r,z}\chi \cdot \nabla_{(r,z)}v_{0},
\end{align}
with the variable $y$ defined as $y = \frac{(r,z)-\xi(t)}{\lambda}$ and where we used the notations
\begin{align}\label{Rdefinition} 
Z_{0}(y)=2U(y)+y\cdot \nabla_{y}U(y),  \quad \mathcal{R}(r,z) = v_{0} - \alpha \Gamma_{0}.
\end{align}
It is important to note that although the operator $S_{(2d)}$ appears to mirror the two dimensional case treated in \cite{BDdPM}, a significant distinction arises in the nonlocal terms, as in our setting we have
$v_{0}=(-\Delta_{\mathbb{R}^{3}})^{-1}u_{0}$.
This preliminary analysis highlights the critical importance of obtaining accurate expansions for the difference $\mathcal{R}(r,z)$.
This section is dedicated to derive a precise expansion for the solution, given by the Newtonian potential, of the following problem: 
\begin{align*}
	- \Delta_{\mathbb{R}^{3}} v_{0} = -\Delta_{r,z} v_{0} - \frac{1}{r} \partial_{r} v_{0} = \frac{\alpha(t)}{\lambda^{2}(t)} U \left( \frac{r - q_{1}}{\lambda(t)}, \frac{z}{\lambda(t)} \right) \chi_{0} \left( \frac{|(r - q_{1}, z)|}{\delta \sqrt{T - t}} \right), \quad (r,z) \in \mathbb{R}^{+}\times \mathbb{R}.
\end{align*}
In the rest of the paper, we will assume the estimate:
\begin{align} \label{EstiLambda}
	c_{1} \sqrt{T - t} \, e^{- \sqrt{\frac{|\ln(T - t)|}{2}}} \leq |\lambda(t)| \leq c_{2} \sqrt{T - t} \, e^{- \sqrt{\frac{|\ln(T - t)|}{2}}}, \quad \text{for some } c_{1}, c_{2} > 0
\end{align}
and the following bounds:
\begin{align} \label{EstiPar}
	\begin{cases}
		|\lambda \dot{\lambda}(t)| + (T - t) \sqrt{|\ln(T - t)|} \left| \frac{d}{dt}(\lambda \dot{\lambda})(t) \right| \lesssim e^{- \sqrt{2|\ln(T - t)|}}, \\[7pt]
		|\dot{\xi}| \lesssim \dfrac{e^{- \left(\frac{3}{2} + \gamma_{1} \right) \sqrt{2|\ln(T - t)|}}}{\sqrt{T - t}}, \quad \text{for some } \gamma_{1} > 0, \\[7pt]
		|\dot{\alpha}| \lesssim\dfrac{e^{- \left(\frac{3}{2} - \gamma_{2} \right) \sqrt{2|\ln(T - t)|}}}{T - t}, \quad \text{for any } \gamma_{2} > 0.
	\end{cases}
\end{align}
All these estimates will be rigorously proved later.
While a rigorous result will be stated in Section~\ref{ChemoEx}, we preview the main goal of this section, which is to establish that for some small $\varepsilon>0$ 
\begin{align*}
	\nabla_{(r,z)}v_{0}(r,z,t)=
	\begin{cases}
		\nabla_{(r,z)} \Gamma_0(r,z) +O(1) & \text{if } |(r,z) - \xi| \leq 2\sqrt{\delta (T - t)}, \\[5pt]
		\omega_1(r,z,t)\, \nabla_{(r,z)} \Gamma_0(r,z) +O(|\ln(|(r,z)-\xi)|) & \text{if } 2\sqrt{\delta (T - t)} \leq |(r,z) - \xi| \leq \varepsilon, \\[5pt]
	\end{cases}
\end{align*}
where $\omega_{1}$ is an explicit function satisfying $	|\omega_{1}(r,z) - 1| = O\left( \sqrt{(r - q_{1})^{2} + z^{2}} \right)$ $\text{as } |(r,z) - \xi| \to 0$.
To accomplish this, we will first review several classical results regarding the inverse Laplacian and elliptic integrals in Section~\ref{EllipticIntegrals}, which will then be utilized in Section~\ref{ChemoEx} to derive the desired expansion.
\subsection{Essential background on the inverse Laplacian and elliptic integrals}\label{EllipticIntegrals}
The first Lemma follows directly from a straightforward estimate of the inverse Laplacian when the right-hand side has compact support, and therefore we omit its proof.
\begin{lemma}\label{Pb310Salsa3D}
	Let $f\in C_{0}(\mathbb{R}^{3})$ and
	\begin{align*}
		u(x)=\frac{1}{4\pi}\int_{\mathbb{R}^{3}}\frac{1}{|x-y|} f(y)dy.
	\end{align*}
	We have 
	\begin{align}\label{expansionInvInf}
		u(x)=\frac{M}{4\pi}\frac{1}{|x|}+O(\|f\|_{L^{1}}|x|^{-2}) \ \ \ |x|\to \infty,
	\end{align}
	\begin{align}\label{expansionInvInfGrad}
		\nabla u(x)=-\frac{M}{4\pi}\frac{x}{|x|^{3}}+O(\|f\|_{L^{1}}|x|^{-3}) \ \ \ |x|\to \infty
	\end{align}
	where $M=\int_{\mathbb{R}^{3}}f(x)dx$. 	In particular, when $f(x)$ is axially symmetric we get the following expansions for $u$:
	\begin{align*}
		u(r,z)=\frac{M}{4\pi}\frac{1}{|(r,z)|}+O(\|f\|_{L^{1}}|(r,z)|^{-2}) \ \ \ |(r,z)|\to \infty,
	\end{align*}
	\begin{align*}
		\nabla_{(r,z)} u(r,z)=-\frac{M}{4\pi}\frac{(r,z)}{|(r,z)|^{3}}+O(\|f\|_{L^{1}}|(r,z)|^{-3}) \ \ \ |(r,z)|\to \infty.
	\end{align*}
\end{lemma}
\begin{proof}
	Omitted.
\end{proof}
\begin{lemma}\label{UniquenessInfInv}
	Let $f \in C_{0}(\mathbb{R}^{2})$. Then the Newtonian potential in $\mathbb{R}^{3}$ defines the unique solution $u$ to the equation $-\Delta u = f$ in $\mathbb{R}^{3}$ that decays.
\end{lemma}
\begin{proof}
	Omitted.
\end{proof}
The preceding two lemmas enable us to establish the following corollary.
\begin{corollary}\label{UniqueDecayCoro}
	Let $f \in C_{0}(\mathbb{R}^{3})$, and let $u$ be a solution to $-\Delta u = f$ in $\mathbb{R}^{3}$ that decays. Then $u$ satisfies \eqref{expansionInvInf} and \eqref{expansionInvInfGrad}.
\end{corollary}
\begin{proof}
	This is a direct consequence of Lemmas \ref{Pb310Salsa3D} and \ref{UniquenessInfInv}.
\end{proof}
The following result enables us to express the inverse Laplacian of axially symmetric functions in terms of the \emph{complete elliptic integral of the first kind}. This reformulation will be crucial for deriving precise asymptotic behavior.  
\begin{lemma}\label{ReprLemmaAS}
	Let $f(x)$ be a continuous and compactly supported axially symmetric function,
    $$
    f(x) = f(r, z), \quad x= (r e^{i\theta} , z).
    $$
    Then we have
	\begin{align*}
		(-\Delta_{\mathbb{R}^{3}})f(x)=\frac{1}{\pi}\int_{\mathbb{R}}\int_{0}^{\infty}\frac{f(\rho,w)\rho}{\sqrt{(z-w)^{2}+(r+\rho)^{2}}}K\left(\sqrt{\frac{4r\rho}{(z-w)^{2}+(r+\rho)^{2}}}\right)d\rho\,dw, \quad x=(r e^{i\theta} , z)
	\end{align*}
	where $K(k)$ denotes the \emph{complete elliptic integral of the first kind} (see (110.06) in \cite{BF}), defined by
	\begin{align}\label{ComIntFirKind}
		K(k)=\int_{0}^{\pi/2}\frac{d\gamma}{\sqrt{1-k^{2}\sin^{2}\gamma}}, \quad 0 \leq k < 1.
	\end{align}
\end{lemma}

\begin{proof}
	Assume $x=(re^{i\theta},z)$ and $y=(\rho e^{i\beta},w)$. Then we have
	\begin{align*}
		|x-y|^{2} = r^{2} + \rho^{2} - 2r\rho \cos(\beta - \theta) + (z-w)^{2}.
	\end{align*}
	Due to the symmetry of the integrand, it is straightforward to verify that
	\begin{align*}
		v(r,z) &= \frac{1}{4\pi} \int_{\mathbb{R}^{3}}\frac{f(y)}{|x - y|}dy \\
		&= \frac{1}{4\pi} \int_{\mathbb{R}}\int_{0}^{\infty}\int_{0}^{2\pi} \frac{f(\rho,w)\rho}{\sqrt{(z-w)^{2} + r^{2} + \rho^{2} - 2r\rho \cos\gamma}}\,d\gamma\,d\rho\,dw \\
		&= \frac{1}{4\pi} \int_{\mathbb{R}}\int_{0}^{\infty} \frac{f(\rho,w)\rho}{\sqrt{(z-w)^{2} + r^{2} + \rho^{2}}} \int_{0}^{2\pi} \frac{d\gamma}{\sqrt{1 - \frac{2r\rho}{(z-w)^{2} + r^{2} + \rho^{2}}\cos\gamma}}\,d\rho\,dw \\
		&= \frac{1}{4\pi} \int_{\mathbb{R}}\int_{0}^{\infty} \frac{f(\rho,w)\rho}{\sqrt{(z-w)^{2} + r^{2} + \rho^{2}}}I(\rho,w)\,d\rho\,dw.
	\end{align*}
	We aim to rewrite $I(\rho,w)$ in terms of the function defined in \eqref{ComIntFirKind}. Let $m := \frac{2r\rho}{(z-w)^{2} + r^{2} + \rho^{2}}$. Then
	\begin{align*}
		I(\rho,w) &= \int_{0}^{2\pi} \frac{d\gamma}{\sqrt{1 - m\cos\gamma}} = 2 \int_{0}^{\pi} \frac{d\gamma}{\sqrt{1 - m\cos\gamma}} \\
		&= 4 \int_{0}^{\pi/2} \frac{d\theta}{\sqrt{1 - m(1 - 2\sin^{2}\theta)}} = 4 \int_{0}^{\pi/2} \frac{d\theta}{\sqrt{(1 - m) + 2m \sin^{2}\theta}} \\
		&= \frac{4}{\sqrt{1 - m}} \int_{0}^{\pi/2} \frac{d\theta}{\sqrt{1 + \frac{2m}{1 - m} \sin^{2}\theta}},
	\end{align*}
	where in the third equality we used the change of variable $\gamma = 2\theta$ and the identity $\cos(2\theta) = 1 - 2\sin^{2}\theta$.
	Since $0 \le m < 1$, the coefficient $\frac{2m}{1 - m}$ is positive. To match the definition of $K(k)$, we now apply the \emph{imaginary modulus transformation}. For this, recall the generalized elliptic integral (see (110.02) in \cite{BF}):
	\begin{align*}
		F(\phi,k) := \int_{0}^{\phi} \frac{d\gamma}{\sqrt{1 - k^{2}\sin^{2}\gamma}}.
	\end{align*}
	This is known as the \emph{incomplete elliptic integral of the first kind}. In particular, $F(\pi/2, k) = K(k)$. The imaginary modulus transformation (see (160.02) in \cite{BF}) gives:
	\begin{align*}
		F(\phi, ik) = k_1' F(\beta, k_1), \quad k_1 = \frac{k}{\sqrt{1 + k^{2}}}, \quad k_1' = \sqrt{1 - k_1^{2}}, \quad \beta = \arcsin\left(\frac{\sqrt{1 + k^{2}}}{\sqrt{1 + k^{2}\sin^{2}\phi}} \sin\phi\right).
	\end{align*}
	In particular, for $\phi = \pi/2$, we get
	\begin{align*}
		F(\pi/2, ik) = \frac{1}{\sqrt{1 + k^{2}}} F\left(\frac{\pi}{2}, \frac{k}{\sqrt{1 + k^{2}}}\right) = \frac{1}{\sqrt{1 + k^{2}}} K\left(\frac{k}{\sqrt{1 + k^{2}}}\right).
	\end{align*}
	Thus, we obtain
	\begin{align*}
		\int_{0}^{\pi/2} \frac{d\theta}{\sqrt{1 + a\sin^{2}\theta}} = F\left(\frac{\pi}{2}, i\sqrt{a}\right) = \frac{1}{\sqrt{1 + a}} K\left(\sqrt{\frac{a}{1 + a}}\right).
	\end{align*}
	In our case, $a = \frac{2m}{1 - m}$, so
	\begin{align*}
		I(\rho,w) = \frac{4}{\sqrt{1 - m}}  \frac{1}{\sqrt{1 + \frac{2m}{1 - m}}} K\left(\sqrt{\frac{\frac{2m}{1 - m}}{1 + \frac{2m}{1 - m}}}\right) = \frac{4}{\sqrt{1 + m}} K\left(\sqrt{\frac{2m}{1 + m}}\right).
	\end{align*}
	Since $m = \frac{2r\rho}{(z - w)^{2} + r^{2} + \rho^{2}}$, we compute:
	\begin{align*}
		1 + m = \frac{(z - w)^{2} + (r + \rho)^{2}}{(z - w)^{2} + r^{2} + \rho^{2}}, \quad \frac{2m}{1 + m} = \frac{4r\rho}{(z - w)^{2} + (r + \rho)^{2}},
	\end{align*}
	and therefore,
	\begin{align*}
		I(\rho,w) = 4\sqrt{\frac{(z - w)^{2} + r^{2} + \rho^{2}}{(z - w)^{2} + (r + \rho)^{2}}}K\left(\sqrt{\frac{4r\rho}{(z - w)^{2} + (r + \rho)^{2}}}\right),
	\end{align*}
	which completes the proof.
\end{proof}

In the remark below, we recall the expansion (900.06) from \cite{BF} for the function $K(x)$ near the point $x = 1$. For additional details, see also \cite{Ka,Ve,KaSi}.
\begin{remark}
	Let $0 < x < 1$. Then,
	\begin{align}\label{ExpEllipticIntegral}
		K(x) = \sum_{m=0}^{\infty} \frac{(\frac{1}{2})_{m} (\frac{1}{2})_{m}}{m! \, m!} (1 - x^2)^m \left( \ln \left( \frac{1}{\sqrt{1 - x^2}} \right) + d(m) \right),
	\end{align}
	where $(\frac{1}{2})_{m}$ denotes the Pochhammer symbol, given by $\frac{(2m)!}{4^m m!}$, and the sequence $d(m)$ satisfies the recurrence relation
	\[
	d(m+1) = d(m) - \frac{2}{(2m+1)(2m+2)}, \quad \text{with } d(0) = 2 \ln 2.
	\]
\end{remark}
\begin{proposition}
	For $0 < 1 - x^2 \ll 1$, the function $K(x)$ admits the following asymptotic expansion:
	\begin{align}\label{ExpansionK1}
		K(x) = \ln \left( \frac{4}{\sqrt{1 - x^2}} \right) + \frac{1}{4} (1 - x^2) \left[ \ln \left( \frac{4}{\sqrt{1 - x^2}} \right) - \frac{1}{6} \right] + O\left( |1 - x^2|^2 \, |\ln(1 - x^2)| \right).
	\end{align}
\end{proposition}
\begin{proof}
	Let $h(m) = \frac{\left(\frac{1}{2}\right)_{m} \left(\frac{1}{2}\right)_{m}}{(m!)^2}$. Using the expansion in \eqref{ExpEllipticIntegral}, we can express $K(x)$ as:
	\begin{align*}
		K(x) &= \sum_{m=0}^{\infty} h(m) (1 - x^2)^m \left[ \ln \left( \frac{1}{\sqrt{1 - x^2}} \right) + d(m) \right] \\
		&= h(0) \left[ \ln \left( \frac{1}{\sqrt{1 - x^2}} \right) + d(0) \right] + h(1)(1 - x^2) \left[ \ln \left( \frac{1}{\sqrt{1 - x^2}} \right) + d(1) \right] + \mathcal{E}(x),
	\end{align*}
	where the remainder term $\mathcal{E}(x)$ is given by:
	\begin{align*}
		\mathcal{E}(x) &= (1 - x^2)^2 \ln \left( \frac{1}{\sqrt{1 - x^2}} \right) \sum_{j=1}^{\infty} h(j+1)(1 - x^2)^j \\
		&\quad + (1 - x^2)^2 \sum_{j=1}^{\infty} h(j+1)(1 - x^2)^j d(j+1).
	\end{align*}
	Since both series in $\mathcal{E}(x)$ are uniformly bounded for $0 < 1 - x^2 \ll 1$, the result follows.
\end{proof}
\subsection{Expansion of the inverse Laplacian in cylindrical coordinates}\label{ChemoEx}
In the previous section, we introduced some general results. In this section, our goal is to demonstrate how these results can be employed to obtain meaningful expansions for the solution of an elliptic problem, which we will regard as a local \emph{small} perturbation of its two dimensional counterpart.
Before proceeding with the three dimensional expansion, we present a Lemma that will be useful in controlling certain error terms.
\begin{lemma}\label{Second3Dcorrection}
	Let $\varepsilon > 0$ be fixed. Suppose $T > 0$ is sufficiently small, and let $f(x,t)$ be an axially symmetric function satisfying
	\begin{align*}
		|f(x,t)| \lesssim \frac{1}{1 + \frac{|(r,z) - \xi|}{\lambda}} \, \chi_{0}\left( \frac{|(r,z) - \xi|}{\sqrt{\delta(T - t)}} \right), \quad \text{where } x = (r e^{i\theta}, z).
	\end{align*}
	Then the function
	\begin{align*}
		v(x,t) = \frac{1}{4\pi} \int_{\mathbb{R}^{3}} \frac{f(y,t)}{|x - y|} \, dy
	\end{align*}
	satisfies the following estimate:
	\begin{align*}
		\big( |(r,z) - \xi| + \lambda(t) \big)\left| \nabla_{(r,z)} v(r,z) \right| + |v(r,z)| \lesssim \lambda(t) \sqrt{T - t} \begin{cases}
			|\ln(T - t)|, & \text{if } |(r,z) - \xi| \le \varepsilon, \\
			\frac{1}{|(r,z) - \xi|}, & \text{if } |(r,z) - \xi| \ge \varepsilon.
		\end{cases}
	\end{align*}
\end{lemma}
\begin{proof}
	Following the approach used in the proof of Lemma \ref{ReprLemmaAS}, for $x=(re^{i\theta},z)$ we can write
	\begin{align*}
		\int_{\mathbb{R}^{3}} \frac{f(y,t)}{|x-y|}dy = \int_{0}^{2\pi}\int_{0}^{\infty}\int_{\mathbb{R}} \frac{f(\rho,w)}{\sqrt{r^{2}+\rho^{2}-2r\rho \cos\gamma + (z-w)^{2}}}\rho d\rho\, dw\, d\gamma.
	\end{align*}
	First, we observe that if $|(r,z)-\xi|\ge 8\sqrt{2}\delta \sqrt{T-t}$, since in the region of integration we have $|\rho- q_1| < 2 \sqrt{\delta (T-t)}$ and $|z| < 2 \sqrt{\delta (T-t)}$, then 
	\begin{align*}
		\frac{1}{2}\big[(q_{1}-r)^{2}+z^{2}\big] &= \frac{1}{2}|(r,z)-\xi|^{2} \ge 4\sqrt{2}\delta\sqrt{T-t}|(r,z)-\xi| \\
		&\ge 4\delta\sqrt{T-t}|r-q_{1}| + 4\delta\sqrt{T-t}|z| \\
		&\ge 2|\rho-q_{1}||r-q_{1}| + 2|z||w|.
	\end{align*}
	If $T$ is sufficiently small, then $|(r,z)-\xi| \ge \varepsilon \gg 8\sqrt{2} \delta\sqrt{T-t}$. Consequently,
	\begin{align*}
		\rho^{2}+r^{2}-2r\rho \cos \gamma+(z-w)^{2} &= (\rho-r)^{2}+(z-w)^{2}+2r\rho(1-\cos \gamma) \\
		&\ge (\rho-q_{1})^{2}+w^{2}+(q_{1}-r)^{2}+2(\rho-q_{1})(q_{1}-r)+z^{2}-2zw \\
		&\ge (\rho-q_{1})^{2}+w^{2}+\frac{1}{2}\big[(q_{1}-r)^{2}+z^{2}\big] \\
		&\ge \frac{1}{2}\big[(q_{1}-r)^{2}+z^{2}\big].
	\end{align*}
	Thus, for $|(r,z)-\xi|\ge \varepsilon$, observing that for $T$ sufficiently small $\lambda y_{1}+q_{1}\le 2|\xi|\lesssim 1$, we obtain
	\begin{align*}
	|v(x)| &\lesssim \frac{1}{|(r,z)-\xi|} \int_{0}^{\infty}\int_{\mathbb{R}} \frac{\rho}{1+\frac{|(\rho,w)-\xi|}{\lambda(t)}} \chi_{0}\left(\frac{|(\rho,w)-\xi|}{ \sqrt{\delta(T-t)}}\right) dw\, d\rho \\
		&\lesssim \frac{\lambda^{2}(t)}{|(r,z)-\xi|} \int_{B_{\frac{2\delta\sqrt{T-t}}{\lambda}}} \frac{\lambda y_1 +q_1 }{1+|y|}dy \lesssim \frac{\lambda(t)\sqrt{T-t}}{|(r,z)-\xi|}.
	\end{align*}
	Now, introduce a function $\beta(t)\to 0$ to be chosen later. For $|(r,z)-\xi|\le \varepsilon$, split the integral:
	\begin{align*}
		|v(x)| \le \int_{\beta(t)}^{2\pi-\beta(t)}(...)d\gamma + \int_{[0,\beta(t)]\cup [2\pi-\beta(t)]}(...)d\gamma =: I^{(1)}_{\beta(t)}(r,z) + I^{(2)}_{\beta(t)}(r,z).
	\end{align*}	
	To estimate $I^{(1)}_{\beta(t)}(r,z)$, note that
	\begin{align*}
		\rho^{2}+r^{2}-2r\rho \cos \gamma+(z-w)^{2} \ge 2r\rho(1-\cos \gamma),
	\end{align*}
	and
	\begin{align*}
		\left|\int_{\beta(t)}^{2\pi-\beta(t)}\frac{d\gamma}{\sqrt{1-\cos \gamma}}\right| \lesssim \ln \beta(t).
	\end{align*}
	Hence,
	\begin{align}\label{Ieps1}
		|I_{\beta(t)}^{(1)}(r,z)| &\lesssim \ln \beta(t) \lambda^{2}(t) \int_{B_{2\delta\sqrt{T-t}/\lambda}} \frac{1}{1+|y|}dy \lesssim \lambda(t)\ln\beta(t)\sqrt{T-t}.
	\end{align}	
	For $I_{\beta(t)}^{(2)}(r,z)$, observe that
	\begin{align*}
		\rho^{2}+r^{2}-2r\rho \cos \gamma+(z-w)^{2} \ge (\rho-r)^{2}+(z-w)^{2}.
	\end{align*}
	This yields 
	\begin{align*}
		|I_{\beta(t)}^{(2)}(r,z)| &\lesssim \beta(t)\int_{0}^{\infty}\int_{\mathbb{R}} \frac{\rho}{\sqrt{(\rho-r)^{2}+(z-w)^{2}}} \frac{1}{1+ \frac{|(\rho,w)-\xi|}{\lambda}}  \chi_{0}\left(\frac{|(\rho,w)-\xi|}{ \sqrt{\delta(T-t)}}\right)dw\, d\rho.
	\end{align*}
	Now let $a = \frac{|(\rho,w)-(r,z)|}{\lambda}$ and $\bar{a} = \frac{|(\rho,w)-\xi|}{\lambda}$. Then,
	\begin{align*}
		\frac{1}{1+\bar{a}} - \frac{1}{1+a} &= \frac{a-\bar{a}}{(1+a)(1+\bar{a})} \le \frac{|(r,z)-\xi|/\lambda}{1+a}.
	\end{align*}
	Moreover, recalling that for small $T$ we have $\lambda y_{1}+q_{1}\le 2|\xi|\lesssim 1$ and $|(\rho,w)-(r,z)| \le 2\varepsilon$, we have	\begin{align}\label{Ieps2}
		|I_{\beta(t)}^{(2)}(r,z)| &\lesssim \beta(t)\lambda(t)\left(1+O\left(\frac{|(r,z)-\xi|}{\lambda}\right)\right) \int_{B_{\frac{2\varepsilon}{\lambda(t)}}} \frac{1}{|y|(1+|y|)}dy \nonumber \\
		&\lesssim \beta(t)\lambda(t)|\ln \lambda(t)|\left(1+O\left(\frac{|(r,z)-\xi|}{\lambda}\right)\right) \lesssim \beta(t)|\ln \lambda(t)|.
	\end{align}
	Combining \eqref{Ieps1} and \eqref{Ieps2}, we choose $\beta(t) = \sqrt{T-t}\lambda(t)$. \newline
	The proof is concluded by rescaling and applying standard elliptic estimates.
\end{proof}
In the next Lemma, we give a representation formula for the two dimensional Newtonian potential, which will be instrumental in deriving the desired expansion in three dimensions.
\begin{lemma}
	Let $ f \in C_{0}(\mathbb{R}^{2})$, and suppose \( f(x) = f(|x - \xi|) \) for some fixed \( \xi \in \mathbb{R}^2 \). Then the two dimensional Newtonian potential of \( f \) is given by
	\begin{align} \label{NPvsODE}
		(-\Delta)^{-1} f(x) 
		:= -\frac{1}{2\pi} \int_{\mathbb{R}^2} \ln |x - y| \, f(y) \, dy 
		= -\int_0^r \frac{1}{u} \int_0^u f(s) s \, ds \, du 
		- \int_0^\infty \ln \rho \, f(\rho) \rho \, d\rho,
	\end{align}
	where \( r = |x - \xi| \).
\end{lemma}
\begin{proof}
	Without loss of generality, we may assume \( \xi = 0 \). Then, for \( x = re^{i\theta} \), we compute:
	\begin{align*}
		u(re^{i\theta}) 
		&= -\frac{1}{2\pi} \int_0^\infty \int_0^{2\pi} \ln \sqrt{r^2 + \rho^2 - 2r\rho \cos(\theta - \phi)} \, f(\rho) \rho \, d\phi \, d\rho \\
		&= -\frac{1}{2\pi} \int_0^\infty \int_0^{2\pi} \ln \sqrt{r^2 + \rho^2 - 2r\rho \cos \phi} \, f(\rho) \rho \, d\phi \, d\rho \\
		&= -\frac{1}{2\pi} \int_0^\infty I(r,\rho) f(\rho) \rho \, d\rho,
	\end{align*}
	where \( I(r,\rho) := \int_0^{2\pi} \ln \sqrt{r^2 + \rho^2 - 2r\rho \cos \phi} \, d\phi \).
	Then:
	\begin{align*}
		I(r,\rho) 
		&= \frac{1}{2} \int_0^{2\pi} \ln \left[ r^2 \left(1 + \frac{\rho^2}{r^2} - 2\frac{\rho}{r} \cos \phi \right) \right] d\phi \\
		&= \frac{1}{2} \left[ 4\pi \ln r + \int_0^{2\pi} \ln \left( 1 + \frac{\rho^2}{r^2} - 2\frac{\rho}{r} \cos \phi \right) d\phi \right].
	\end{align*}
	We now apply formula 4.224–14 in \cite{GR}:
	\begin{align*}
		\int_0^{n\pi} \ln(a^2 + b^2 - 2ab \cos \phi) \, d\phi = 
		\begin{cases}
			n\pi \ln b^2, & \text{if } a^2 \geq b^2 > 0, \\
			n\pi \ln a^2, & \text{if } b^2 > a^2 > 0.
		\end{cases}
	\end{align*}
	Choosing \( a = \frac{\rho}{r} \), \( b = 1 \), we distinguish two cases:
	if \( \rho \leq r \), then \( I(r,\rho) = 2\pi \ln r \),
	if \( \rho > r \), then \( I(r,\rho) = 2\pi \ln \rho \).
	Thus, we obtain:
	\begin{align*}
		(-\Delta)^{-1} f(x) 
		= -\int_r^\infty \ln \rho \, f(\rho) \rho \, d\rho 
		- \ln r \int_0^r f(\rho) \rho \, d\rho.
	\end{align*}
	Finally, we manipulate the expression by adding and subtracting the full integral:
	\begin{align*}
		(-\Delta)^{-1} f(x) 
		&= \left[ -\int_0^\infty \ln \rho \, f(\rho) \rho \, d\rho \right] 
		+ \int_0^r \ln \rho \, f(\rho) \rho \, d\rho 
		- \ln r \int_0^r f(\rho) \rho \, d\rho \\
		&= -\int_0^r \frac{1}{u} \int_0^u f(s) s \, ds \, du 
		- \int_0^\infty \ln \rho \, f(\rho) \rho \, d\rho.
	\end{align*}
\end{proof}
\begin{theorem}\label{Expansionv0}
	Assume that \eqref{EstiLambda} and \eqref{EstiPar} hold. Let \( v \) be the solution, given by the Newtonian potential, of the equation
	\begin{align*}
		-\Delta_{(r,z)} v_{0} - \frac{1}{r} \, \partial_{r} v_{0} 
		= \frac{\alpha(t)}{\lambda^{2}(t)} \, U\left( \frac{r - q_{1}(t)}{\lambda(t)}, \frac{z}{\lambda(t)} \right) 
		\chi_{0} \left( \frac{|(r,z) - \xi(t)|}{\sqrt{T - t}} \right), 
		\quad (r,z) \in \mathbb{R}^{+} \times \mathbb{R}.
	\end{align*}
	Then, if \( |(r,z) - \xi| \le \varepsilon \ll 1 \), the function \( v \) satisfies the following expansion:
	\begin{align}\label{Expansion3d}
		v_{0}(r,z) = 
		& \alpha(t)\left[ \Gamma_{0}^{R}(r,z) + (4 - M_{0}(t)) \ln \sqrt{T - t} + c_{0}(t) \right] 
		\chi_{0}\left( \frac{|(r,z) - \xi|}{\sqrt{\delta(T - t)}} \right) \nonumber \\
		& + \alpha(t)\omega_{1}(r,z) \left[ \Gamma_{0}^{R}(r,z) + (4 - M_{0}(t)) \ln \sqrt{T - t} + c_{0}(t) + c_{1}(r,t) \right] 
		\left( 1 - \chi_{0} \left( \frac{|(r,z) - \xi|}{\sqrt{\delta(T - t)}} \right) \right) \nonumber \\
		& + O(1)
	\end{align}
	where
	\begin{align*}
		\begin{cases}
			\Gamma_{0}^{R}(r,z) 
			:= \Gamma_{0} \left( \frac{(r,z) - \xi}{\lambda} \right) - 4 \log \lambda, \\[5pt]					
			c_{0}(t) = O\left( \frac{\lambda^{2}}{T - t} \right), \quad 
			c_{1}(r,z,t) = O\left( \frac{\lambda^{2}}{T - t} \right), \\[5pt]			
			M_{0}(t) = 4 \left( 1 + O\left( \frac{\lambda^{2}}{T - t} \right) \right), \\[5pt]			
			\omega_{1}(r,z) 
			= \frac{2q_{1}}{\sqrt{z^{2} + (r + q_{1})^{2}}}
			\left( 1 + \frac{1}{4} \frac{z^{2} + (r - q_{1})^{2}}{z^{2} + (r + q_{1})^{2}} \right)
			\Rightarrow 
			|\omega_{1}(r,z) - 1| = O\left( \sqrt{(r - q_{1})^{2} + z^{2}} \right)
			\quad \text{as } |(r,z) - \xi| \to 0.
		\end{cases}
	\end{align*}
\end{theorem}
\begin{proof}
	In the following, we factor out $\alpha(t)$ without altering the notation for the solution $v_0$. The final expression will then be obtained by multiplying the resulting expansion by $\alpha(t)$.
We proceed in four steps.\newline
	\underline{First Step:} We begin with the natural first approximation:
	\begin{align*}
		v_{1}(r,z) = \left[ \Gamma_{0} \left( \frac{|(r,z) - \xi|}{\lambda} \right) + 4 \ln \left( \frac{\sqrt{\delta(T - t)}}{\lambda} \right) \right] 
		\chi_{0} \left( \frac{|(r,z) - \xi|}{\sqrt{\delta(T - t)}} \right).
	\end{align*}
	The corresponding error can be computed in the usual way:
	\begin{align*}
		\left( -\Delta_{(r,z)} - \frac{1}{r} \partial_{r} \right)(v - v_{1}) 
		= \frac{1}{r} \partial_{r} v_{1} + \mathcal{E}_{2d}(r,z),
	\end{align*}
	where
	\begin{align} \label{2dError3}
		\mathcal{E}_{2d} \left( \frac{|(r,z) - \xi|}{\sqrt{T - t}} \right) 
		&= \left[ \Gamma_{0} \left( \frac{|(r,z) - \xi|}{\lambda} \right) - \Gamma_{0} \left( \frac{\sqrt{T - t}}{\lambda} \right) \right] 
		\Delta_{(r,z)} \chi \nonumber \\
		&\quad + 2 \nabla_{(r,z)} \left[ \Gamma_{0} \left( \frac{|(r,z) - \xi|}{\lambda} \right) - \Gamma_{0} \left( \frac{\sqrt{T - t}}{\lambda} \right) \right] 
		\cdot \nabla_{(r,z)} \chi.
	\end{align}
	We observe that \( \mathcal{E}_{2d}(r,z) \) coincides precisely with the error one would ecounter in the two dimensional equation.\newline
	\underline{Second Step:} In this step, we eliminate the term \( \frac{1}{r} \partial_{r} v_{1} \). This is achieved by applying Lemma~\ref{Second3Dcorrection}, after noting that
	\begin{align*}
		\left| \frac{1}{r} \partial_{r} v_{1} \right| 
		\lesssim \frac{1}{\lambda}  \frac{1}{1 + \frac{|(r,z) - \xi|}{\lambda}} 
		\chi_{0} \left( \frac{|(r,z) - \xi|}{\sqrt{T - t}} \right).
	\end{align*}
	We then introduce a correction term \( v_{2}(r,z) \), which satisfies the following estimate:
	\begin{align*}
		|v_{2}(r,z)| \lesssim \sqrt{T - t}  
		\begin{cases}
			|\ln(T - t)|, & \text{if } |(r,z) - \xi| \leq \varepsilon, \\
			\dfrac{1}{|(r,z) - \xi|}, & \text{if } |(r,z) - \xi| \geq \varepsilon.
		\end{cases}
	\end{align*}
	This correction can be naturally incorporated into the error terms introduced in the statement of the theorem. \newline
	\underline{Third Step:} In this step, our goal is to eliminate the self-similar two-dimensional error term:
	\begin{align*}
		-\Delta_{\mathbb{R}^{3}}v_{2} = \mathcal{E}_{2d}(x) = \mathcal{E}_{2d}(r,z).
	\end{align*}
	Using Lemma \ref{ReprLemmaAS}, we can represent $v_2$ as
	\begin{align*}
		v_{2}(r,z) = \frac{1}{\pi}\int_{\mathbb{R}} \int_{0}^{\infty} \frac{\mathcal{E}_{2d}(\rho,w)\rho}{\sqrt{(z-w)^{2}+(r+\rho)^{2}}} K\left( \sqrt{ \frac{4r\rho}{(z-w)^{2}+(r+\rho)^{2}} } \right) \, d\rho\, dw.
	\end{align*}
	Our strategy is to apply the expansion \eqref{ExpansionK1} in the region where $|(r,z)-\xi|\leq \varepsilon$ for some small $\varepsilon > 0$. In this regime, we observe:
	\begin{align}\label{onex2}
		x = \sqrt{ \frac{4r\rho}{(z-w)^{2}+(r+\rho)^{2}} } \quad \Rightarrow \quad 1 - x^2 = \frac{(z-w)^2 + (r - \rho)^2}{(z-w)^2 + (r + \rho)^2} \ll 1,
	\end{align}
	where the final inequality follows under the assumptions $|(r,z)-\xi| \le \varepsilon \ll 1$ and sufficiently small $T$. Before proceeding with the expansion, we perform several auxiliary computations for simplification.
	Let us define the total mass
	\begin{align*}
		M_0 := \int_{\mathbb{R}^2} \mathcal{E}_{2d}(\rho,w)\, d\rho\, dw.
	\end{align*}
	Introducing the scaling variable $s = \frac{|(\rho,w)-\xi|}{\sqrt{\delta(T-t)}}$, a straightforward expansion yields:
	\begin{align}\label{E2dAsy}
		\mathcal{E}_{2d}(\rho,w) = \frac{1}{\delta(T-t)}\left[ -4\log s\left( \chi''(s) + \frac{1}{s}\chi'(s) \right) - \frac{8}{s}\chi'(s) \right] + O\left( \frac{\lambda^2}{\delta(T-t)} \mathds{1}_{1\leq s \leq 2} \right),
	\end{align}
	which implies
	\begin{align}\label{M0}
		M_0 = 4 \left( 1 + O\left( \frac{\lambda^2}{\delta(T-t)} \right) \right).
	\end{align}
	We next analyze the inverse Laplacian of $\mathcal{E}_{2d}$. By using the representation formula \eqref{NPvsODE}, we compute:
	\begin{align}\label{2dInverse}
		(-\Delta_{\mathbb{R}^2})^{-1} \mathcal{E}_{2d}(\rho,w) 
		&= (-4 \log s + c_1(r,z))(1 - \chi_0(s)) - \int_0^{\infty} \ln u\, \mathcal{E}_{2d}\left( \frac{u}{\sqrt{\delta(T-t)}} \right) u\, du \nonumber \\
		&= -M_0 \ln \sqrt{\delta(T - t)} + c_0(t) + (-4 \log s + c_1(r,z))(1 - \chi_0(s)),
	\end{align}
	where $c_1(r,z) = O\left( \frac{\lambda^2}{T - t} \right)$, and
	\begin{align*}
		c_0(t) = \int_0^{\infty} \ln\left( \frac{u}{\sqrt{\delta(T - t)}} \right) \mathcal{E}_{2d}\left( \frac{u}{\sqrt{\delta(T - t)}} \right) u\, du.
	\end{align*}
	Using \eqref{E2dAsy}, this integral becomes:
	\begin{align*}
		c_0(t) &= \int_0^{\infty} \left[ -4\ln s\left( \chi''(s) + \frac{1}{s} \chi'(s) \right) - \frac{8}{s} \chi'(s) \right] s \ln s\, ds + O\left( \frac{\lambda^2}{\delta(T - t)} \right) \\
		&= -4 \int_0^{\infty} \frac{d}{ds} \left( s (\ln s)^2 \chi'(s) \right) ds + O\left( \frac{\lambda^2}{\delta(T - t)} \right) = O\left( \frac{\lambda^2}{\delta(T - t)} \right).
	\end{align*}
	We now make three useful observations that will be instrumental in controlling the remainder terms.
	First, we have the estimate:
	\begin{align}\label{err1control}
		\left| \frac{\rho}{\sqrt{(z - w)^2 + (r + \rho)^2}} - \frac{q_1}{\sqrt{z^2 + (r + q_1)^2}} \right| \lesssim \sqrt{T - t}.
	\end{align}
	Second, the following approximation holds:
	\begin{align}\label{err2control}
		\left| \frac{(z - w)^2 + (r - \rho)^2}{(z - w)^2 + (r + \rho)^2} - \frac{z^2 + (r - q_1)^2}{z^2 + (r + q_1)^2} \right| \lesssim \sqrt{T - t}.
	\end{align}
	The third observation concerns the control of a logarithmic integral involving the error term \( \mathcal{E}_{2d} \). We claim that:
	\begin{align}\label{logmass}
		\int_{\mathbb{R}^2} \mathcal{E}_{2d}(\rho, w) \left| \ln\left( (\rho - r)^2 + (z - w)^2 \right) \right| \lesssim |\ln(T - t)|.
	\end{align}
	To see this, we split the integral into two regions. For the inner region where \( \sqrt{(\rho - r)^2 + (z - w)^2} \le \beta(t) \), we estimate:
	\begin{align*}
		\int_{\sqrt{(\rho - r)^2 + (z - w)^2} \le \beta(t)} \mathcal{E}_{2d}(\rho, w) \left| \ln\left( (\rho - r)^2 + (z - w)^2 \right) \right| 
		\lesssim \frac{1}{T - t} \int_{0}^{\beta(t)} s \ln s \, ds 
		\lesssim \frac{\beta^2(t) |\ln \beta(t)|}{T - t}.
	\end{align*}
	For the outer region where \( \sqrt{(\rho - r)^2 + (z - w)^2} \ge \beta(t) \), we use the bound:
	\begin{align*}
		\int_{\sqrt{(\rho - r)^2 + (z - w)^2} \ge \beta(t)} \mathcal{E}_{2d}(\rho, w) \left| \ln\left( (\rho - r)^2 + (z - w)^2 \right) \right| 
		\lesssim |\ln \beta(t)|.
	\end{align*}
	By choosing \( \beta(t) = \sqrt{T - t} \), we obtain the desired estimate \eqref{logmass}.
	With these preparations, and applying the expansion \eqref{ExpansionK1} using \eqref{onex2}, we write
	\begin{align}\label{firstexpv2}
		v_2(r,z) &= \frac{\ln 4}{\pi} \int_{\mathbb{R}^2} \frac{\mathcal{E}_{2d}(\rho,w)\rho}{\sqrt{(z - w)^2 + (r + \rho)^2}}\, d\rho dw \nonumber \\
		&\quad - \frac{1}{\pi} \int_{\mathbb{R}^2} \frac{\mathcal{E}_{2d}(\rho,w)\rho}{\sqrt{(z - w)^2 + (r + \rho)^2}} \ln \sqrt{ \frac{(z - w)^2 + (r - \rho)^2}{(z - w)^2 + (r + \rho)^2} }\, d\rho dw \nonumber \\
		&\quad + \frac{\ln 4}{4\pi} \int_{\mathbb{R}^2} \frac{(z - w)^2 + (r - \rho)^2}{(z - w)^2 + (r + \rho)^2} \cdot \frac{\mathcal{E}_{2d}(\rho,w)\rho}{\sqrt{(z - w)^2 + (r + \rho)^2}}\, d\rho dw \nonumber \\
		&\quad - \frac{1}{4\pi} \int_{\mathbb{R}^2} \frac{(z - w)^2 + (r - \rho)^2}{(z - w)^2 + (r + \rho)^2} \cdot \frac{\mathcal{E}_{2d}(\rho,w)\rho}{\sqrt{(z - w)^2 + (r + \rho)^2}} \ln \sqrt{ \frac{(z - w)^2 + (r - \rho)^2}{(z - w)^2 + (r + \rho)^2} }\, d\rho dw \nonumber \\
		&\quad + O\left( \left(z^2 + (r - q_1)^2\right)^2 |\ln(z^2 + (r - q_1)^2)|^2 \right) + O(\sqrt{T - t}),
	\end{align}
	where the error term is controlled using the boundedness of the derivative of $f(x) = x^2 \ln x$ near zero, together with \eqref{err1control}–\eqref{err2control}.
	Combining and simplifying these terms, we obtain the refined expansion:
	\begin{align}\label{secexpv2}
		v_2(r,z) &= \frac{2q_1}{\sqrt{z^2 + (r + q_1)^2}} \left( 1 + \frac{1}{4}  \frac{z^2 + (r - q_1)^2}{z^2 + (r + q_1)^2} \right) (-\Delta_{\mathbb{R}^2})^{-1} \mathcal{E}_{2d}(\rho,w) \nonumber \\
		&\quad + \mathcal{V}(r,z) M_0(t) + O\left( \left(z^2 + (r - q_1)^2\right)^2 |\ln(z^2 + (r - q_1)^2)|^2 \right) + O(\sqrt{T - t} |\ln(T - t)|),
	\end{align}
	where, recalling \eqref{M0}, the function \( \mathcal{V} \) is defined by
	\begin{align}\label{defV}
		\mathcal{V}(r,z) := \frac{1}{\pi}  \frac{q_1}{\sqrt{z^2 + (r + q_1)^2}} \left( \ln 4 + \ln \sqrt{z^2 + (r + q_1)^2} \right)\left[ 1+  \frac{1}{4}  \frac{z^2 + (r - q_1)^2}{z^2 + (r + q_1)^2}  \right].
	\end{align}
	We remark that the final logarithmic term in \eqref{secexpv2} originates from the second and fourth integrals in \eqref{firstexpv2} thanks to the estimate \eqref{logmass}. \newline
	\underline{Fourth Step:} We now collect all the expansions obtained so far. Recall from Lemma \ref{Second3Dcorrection} that the term \( v_1 \) can be absorbed into the error as it satisfies the estimate \( v_1 = O(\sqrt{T-t}|\ln(T-t)|) \). Let us introduce the notation
	\[
	\omega_{1}(r,z) := \frac{2q_{1}}{\sqrt{z^{2}+(r+q_{1})^{2}}} \left(1+\frac{1}{4}\frac{z^{2}+(r-q_{1})^{2}}{z^{2}+(r+q_{1})^{2}} \right),
	\]
	so that the expression for \( v_1 + v_2 \) becomes
	\begin{align*}
		v_1(r,z) + v_2(r,z) =\;& \left[-4\ln \sqrt{\delta(T-t)} + (4 - M_0(t))\ln \sqrt{\delta(T-t)} + c_0(t) \right] \chi_{0}\left( \frac{|(r,z)-\xi|}{\sqrt{\delta(T-t)}} \right) \\
		&+ \omega_1(r,z) \left[ -4\ln |(r,z)-\xi| + (4 - M_0(t))\ln \sqrt{\delta(T-t)} + c_0(t) + c_1(r,z) \right]\cdot \\
		&\quad \cdot \left(1 - \chi_{0}\left( \frac{|(r,z)-\xi|}{\sqrt{\delta(T-t)}} \right)\right) + \mathcal{V}(r,z) M_0(t) \\
		&+ O\left((z^2 + (r - q_1)^2)^2 |\ln(z^2 + (r - q_1)^2)|^2\right) + O\left(\sqrt{T-t}|\ln(T-t)|\right).
	\end{align*}
	Now we can add the contribution from \( v_0(r,z) \), and absorb the difference
	$-4\ln |(r,z)-\xi| - \Gamma_0(r,z) - 4\ln \lambda$
	into \( c_1(r,z) \). We can conclude the proof by cutting $v_{1}+v_{2}+v_{3}$ in the region \( |(r,z) - \xi| \le \varepsilon \), erasing the new error by the Newtonian potential, and then applying Corollary~\ref{UniqueDecayCoro}.
\end{proof}
What remains is to establish the asymptotic expansion for the gradient of $v_{0}$, which is precisely the focus of the following Proposition.
\begin{proposition}\label{3dGRADIENTExpansion}
	Let $v_{0}$ and $\omega_1(r,z)$ be the functions defined in Theorem~\ref{Expansionv0}. Then, for $\varepsilon > 0$ sufficiently small and $M > 0$ sufficiently large, the following expansion holds:
	\begin{align}\label{ExpansionGradv0}
		\nabla_{(r,z)} v_{0}(r,z,t) = 
		\begin{cases}
			\nabla_{(r,z)} \Gamma_0(r,z) +O(1) & \text{if } |(r,z) - \xi| \leq 2\sqrt{\delta (T - t)}, \\[5pt]
			\omega_1(r,z,t)\, \nabla_{(r,z)} \Gamma_0(r,z) +O(|\ln(|(r,z)-\xi)|) & \text{if } 2\sqrt{\delta (T - t)} \leq |(r,z) - \xi| \leq \varepsilon, \\[5pt]
			O(1) & \text{if } \varepsilon \leq |(r,z) - \xi| \leq M, \\[5pt]
			-4\pi q_1 \left(1 + O\left(\frac{\lambda^2}{T - t} \right)\right) \frac{(r,z)}{|(r,z)|^3} \left(1 + O\left(\frac{1}{M}\right)\right) & \text{if } |(r,z) - \xi| \geq M.
		\end{cases}
	\end{align}
\end{proposition}

\begin{proof}
	We begin by differentiating the equation
	\begin{align*}
		-\Delta_{(r,z)} v_{0} - \frac{1}{r} \, \partial_{r} v_{0} 
		= \frac{1}{\lambda^{2}(t)} \, U\left( \frac{r - q_{1}(t)}{\lambda(t)}, \frac{z}{\lambda(t)} \right) 
		\chi_{0} \left( \frac{|(r,z) - \xi(t)|}{\sqrt{T - t}} \right), 
		\quad (r,z) \in \mathbb{R}^{+} \times \mathbb{R}.
	\end{align*}
	We omit the detailed analysis in the region \( |(r,z) - \xi| \le \varepsilon \), as it follows the same arguments used in the proof of Theorem~\ref{Expansionv0}. We just focus on outlining the main steps needed to estimate \( \nabla_{(r,z)} v_{2}(r,z) \).
	Recall that
	\begin{align*}
		\nabla_{(r,z)}v_{2}(r,z) = \frac{1}{\pi} \int_{\mathbb{R}} \int_{0}^{\infty} \frac{\nabla_{(\rho,w)} \mathcal{E}_{2d}(\rho,w) \, \rho}{\sqrt{(z - w)^2 + (r + \rho)^2}} \, K\left( \sqrt{ \frac{4r\rho}{(z - w)^2 + (r + \rho)^2} } \right) \, d\rho \, dw.
	\end{align*}
	For clarity, we restrict attention to the contribution from the first term in the expansion \eqref{ExpansionK1}; the remaining terms can be handled in a similar fashion.
	Our goal is to derive an expansion for the following integral:
	\begin{align*}
		I(r,z) = \frac{1}{\pi} \int_{\mathbb{R}} \int_{0}^{\infty} \frac{\nabla_{(\rho,w)} \mathcal{E}_{2d}(\rho,w) \, \rho}{\sqrt{(z - w)^2 + (r + \rho)^2}} \, K_{1}\left( \sqrt{ \frac{4r\rho}{(z - w)^2 + (r + \rho)^2} } \right) \, d\rho \, dw,
	\end{align*}
	where
	\begin{align*}
		K_{1}\left( \sqrt{ \frac{4r\rho}{(z - w)^2 + (r + \rho)^2} } \right) = \ln 4 + \ln \sqrt{(z - w)^2 + (r + \rho)^2} - \ln |(r,z) - (\rho,w)|.
	\end{align*}
	The first two terms in \( K_1 \) can be handled directly as in the proof of Theorem~\ref{Expansionv0} by invoking estimate~\eqref{err1control} and noting that
	\begin{align*}
		\int_{\mathbb{R}} \int_{0}^{\infty} |\nabla_{(\rho,w)} \mathcal{E}_{2d}(\rho,w)| \, dw \, d\rho = O\left( \frac{1}{\sqrt{T - t}} \right).
	\end{align*}
	The last term can be expanded by performing an integration by parts and then applying estimate~\eqref{logmass}.
	To prove the estimate for $|(r,z)-\xi| \ge \varepsilon$, it is enough to recall the following elementary fact: let $f \in L^{1}(\mathbb{R}^{3})$ with compact support. If $x \notin \mathrm{supp}(f)$, then
	\begin{align}\label{gradoutsupp}
		|\nabla (-\Delta_{\mathbb{R}^{3}})^{-1}f(x)| \lesssim \int_{\mathrm{supp}(f)} |f(y)| \frac{1}{|x - y|^{2}} \, dy \lesssim \frac{1}{d(x, \mathrm{supp}(f))} \|f\|_{L^{1}} < \infty.
	\end{align}
	We now apply \eqref{gradoutsupp}, Lemma \ref{Pb310Salsa3D} and Corollary \ref{UniqueDecayCoro}, observing that
	\begin{align*}
		\int_{\mathbb{R}^{3}} u_{0}(x) \, dx &= 2\pi \int_{0}^{\infty} \int_{\mathbb{R}} \frac{1}{\lambda^{2}} U\left(\frac{(r,z)-\xi(t)}{\lambda(t)}\right) \chi_{0}\left(\frac{(r,z)-\xi(t)}{\sqrt{\delta(T-t)}}\right) r \, dr \, dz \\
		&= -16\pi q_{1} \left(1 + O\left(\frac{\lambda^{2}}{T-t}\right)\right),
	\end{align*}
	and that similarly $\|u_{0}\|_{L^{1}} < \infty$.
\end{proof}
\section{The tailored Correction $\varphi_{\lambda}$}\label{3dphilambdaSec}
As a consequence of Proposition \ref{3dGRADIENTExpansion}, we know that locally $\nabla_{(r,z)} v_{0} \approx \nabla_{(r,z)} \Gamma_0$. Following the approach in \cite{DdPDMW,BDdPM}, we proceed accordingly. The analysis of the chemoattractant's behavior away from the singularity will be addressed in the second part of this section. \newline
Our goal here is to eliminate the leading-order terms in \eqref{S2dERR}. To this end, we introduce a radial correction in the shifted variables $(r,z)-\xi(t)$, defined by the solution to the following equation:
\begin{align}\label{philambdaEqt2D}
	\begin{cases}
		\partial_{t} \varphi_{\lambda}^{(6)} = \Delta_{(r,z)} \varphi_{\lambda}^{(6)} + 4\frac{(r,z)}{|(r,z)|^{2}} \cdot \nabla_{(r,z)} \varphi_{\lambda}^{(6)} + E(r,z,t), & \text{in } (-\varepsilon(T),T) \times \mathbb{R}^{2}, \\
		\varphi_{\lambda}^{(6)}(\cdot, -\varepsilon(T)) = 0, & \text{in } \mathbb{R}^{2},
	\end{cases}
\end{align}
which is solved using Duhamel's formula. Here, the source term $E(r,z,t)$ is given by
\begin{align*}
	E(r,z,t) = \frac{\dot{\lambda}}{\lambda} Z_{0}(\bar{y}) \chi_{0}(\bar{w}) 
	- \frac{1}{2\lambda^{2}(T-t)} U(\bar{y}) \nabla_{\bar{w}} \chi_{0}(\bar{w}) \cdot \bar{w} 
	+ \tilde{E}(r,z;\lambda),
\end{align*}
where
\begin{align*}
	\tilde{E}(r,z;\lambda) =\ &\frac{2}{\lambda^{3} \sqrt{\delta(T-t)}} \nabla_{\bar{w}} \chi_{0}(\bar{w}) \cdot \nabla_{\bar{y}} U(\bar{y}) 
	+ \frac{1}{\delta(T-t)} \frac{1}{\lambda^{2}} \Delta_{\bar{w}} \chi_{0}(\bar{w}) U(\bar{y}) \\
	&- \frac{1}{\alpha} \frac{1}{\lambda^{2} \sqrt{\delta(T-t)}} U(\bar{y}) \nabla_{\bar{w}} \chi_{0}(\bar{w}) \cdot \nabla_{(r,z)} v_{0},
\end{align*}
and the rescaled variables are defined as
\begin{align*}
	\bar{w} = \frac{(r,z)-\xi(T)}{\sqrt{\delta(T-t)}}, \quad \bar{y} = \frac{(r,z)-\xi(T)}{\lambda}.
\end{align*}
The initial time $\varepsilon(T)$ appearing in \eqref{philambdaEqt2D} is a small, explicitly defined parameter whose role is explained in Section \ref{Sectionalpha0lambda0} and it is essential for establishing Proposition \ref{Proplambda0alpha0}. Using the assumptions \eqref{EstiLambda}, \eqref{EstiPar}, and Proposition \ref{3dGRADIENTExpansion}, we obtain the following estimate:
\begin{align*}
	|E(r,z,t)| \le C \frac{1}{\lambda^{2}} \frac{1}{\left(1 + \frac{|(r,z)-\xi(T)|^{2}}{\lambda^{2}} \right)^{2}} \chi_{0} \left( \frac{|(r,z)-\xi(T)|}{2\sqrt{\delta(T-t)}} \right).
\end{align*}
The equation \eqref{philambdaEqt2D} has been thoroughly studied in \cite{BDdPM}.
\begin{lemma}\label{estimatePhilambda}
	Assume $T$ and $\varepsilon(T)$ are sufficiently small, and let $\varphi_{\lambda}^{(6)}$ be the solution to \eqref{philambdaEqt2D}, where $\lambda$ satisfies \eqref{EstiLambda} and \eqref{EstiPar}. Then, for any $(x,t) \in \mathbb{R}^{3} \times (-\varepsilon(T), T)$ with $x = (re^{i\theta}, z)$, the following estimate holds:
	\begin{align*}
		|\varphi_{\lambda}^{(6)}| + (|(r,z)-\xi(T)| + \lambda) |\nabla \varphi_{\lambda}^{(6)}| \le C
		\begin{cases}
			\frac{e^{-\sqrt{2|\ln(T-t)|}}}{\lambda^{2} + |(r,z)-\xi(T)|^{2}}, & \text{if } |(r,z)-\xi(T)| \le \sqrt{T-t}, \\
			\frac{e^{-\sqrt{2|\ln(|(r,z)-\xi|^{2})|}}}{|(r,z)-\xi(T)|^{2}}, & \text{if } \sqrt{T-t} \le |(r,z)-\xi(T)| \le 2\varepsilon(T), \\
			\frac{e^{-\sqrt{2|\ln(\varepsilon(T))|}}}{\varepsilon(T)} \frac{e^{-| (r,z)-\xi(T) |^{2}}}{4(t + 2\varepsilon(T))}, & \text{if } |(r,z)-\xi(T)| \ge \sqrt{\varepsilon(T)}.
		\end{cases}
	\end{align*}
	Moreover, we also have the following bound:
	\begin{align*}
		|\nabla \varphi_{\lambda}^{(6)}| + (|(r,z)-\xi(T)| + \lambda) |D^{2} \varphi_{\lambda}^{(6)}| 
		\le C e^{-\sqrt{2|\ln(T-t)|}} \frac{|(r,z)-\xi(T)|}{(\lambda^{2} + |(r,z)-\xi(T)|^{2})^{2}}, \quad |(r,z)-\xi(T)| \le \sqrt{T-t}.
	\end{align*}
\end{lemma}
\begin{proof}
	See the proof of Lemma 3.1 in \cite{BDdPM}.
\end{proof}
The approach in this three dimensional setting is to consider the ansatz
\begin{align*}
	u_{1} = u_{0} + \varphi_{\lambda}^{(6)} \hat{\chi} = u_{0} + \varphi_{\lambda},
\end{align*}
where $\hat{\chi} = \chi_{0} \left( \frac{|(r,z) - \xi|}{2 \sqrt{\delta(T - t)}} \right)$. We then compute
\begin{align*}
	S(u_{1}) = S_{(2d)}(u_{1}) + \frac{1}{r} \partial_{r} u_{0} + \frac{1}{r} \partial_{r} \varphi_{\lambda},
\end{align*}
where, letting $\rho = |(r,z) - \xi|$ and $w=\frac{(r,z)-\xi}{\sqrt{\delta(T-t)}}$, the two dimensional part is given by
\begin{align*}
	S_{(2d)}(u_{1}) &= -\frac{\dot{\alpha}}{\lambda^{2}} U(y) \chi_{0}(w) 
	+ (\alpha - 1) \frac{\dot{\lambda}}{\lambda^{3}} Z_{0}(y) \chi_{0}(w) 
	+ \frac{\alpha}{\lambda^{3}} \dot{\xi} \cdot \nabla_{y} U(y) \chi_{0}(w) \\
	&\quad + \frac{\alpha}{\lambda^{2} \sqrt{\delta(T - t)}} U(y) \dot{\xi} \cdot \nabla_{w} \chi_{0}(w)
	- \frac{\alpha - 1}{2 \lambda^{2} (T - t)} U(y) \nabla_{w} \chi_{0}(w) \cdot w \\
	&\quad + \frac{2(\alpha - 1)}{\lambda^{3} \sqrt{\delta(T - t)}} \nabla_{w} \chi_{0}(w) \cdot \nabla_{y} U(y)
	+ \frac{\alpha - 1}{\delta(T - t)} \frac{1}{\lambda^{2}} \Delta_{w} \chi_{0}(w) U(y) \\
	&\quad - \frac{1 - 1/\alpha}{\lambda^{2} \sqrt{\delta(T - t)}} U(y) \nabla_{w} \chi_{0}(w) \cdot \nabla_{(r,z)} v_{0} \\
	&\quad - \frac{\alpha(\alpha - 1) \chi_{0}(w)}{\lambda^{4}} \nabla_{y} \cdot (U \nabla_{y} \Gamma_{0}) 
	+ \frac{\alpha \chi_{0}(w)}{\lambda^{4}} \left[ \alpha(\chi_{0}(w) - 1) U^{2} - \nabla_{y} U \cdot \nabla_{y} \mathcal{R} \right] \\
	&\quad - \frac{\alpha - 1}{\lambda^{2} \sqrt{\delta(T - t)}} U(y) \nabla_{w} \chi_{0}(w) \cdot \nabla_{(r,z)} v_{0} \\
	&\quad - \frac{4}{\rho} \partial_{\rho} \varphi_{\lambda} 
	- \text{div}_{(r,z)}  (\varphi_{\lambda} \nabla_{(r,z)} v_{0}) 
	- \text{div}_{(r,z)}  (u_{0} \nabla_{(r,z)} \psi_{\lambda}) 
	- \text{div}_{(r,z)} (\varphi_{\lambda} \nabla_{(r,z)} \psi_{\lambda}) \\
	&\quad + E((r - z) - \xi, t) - E((r, z) - \xi(T), t).
\end{align*}
We recall the definitions:
\begin{align*}
	v_{0} = (-\Delta_{\mathbb{R}^{3}})^{-1} u_{0}, \quad 
	\psi_{\lambda} = (-\Delta_{\mathbb{R}^{3}})^{-1} \varphi_{\lambda}, \quad 
	\text{div}_{(r,z)}((A_{r}, A_{\theta}, A_{z})) = \frac{1}{r} \partial_{r}(r A_{r}) + \frac{1}{r} \partial_{\theta}(A_{\theta}) + \partial_{z} A_{z}.
\end{align*}
Before proceeding with error estimates, we must first control $\psi_{\lambda}$ in a manner analogous to our treatment of $v_{0}$.
\begin{proposition}\label{3dGRADIENTExpansionpsilambda}
	Assume that conditions \eqref{EstiLambda} and \eqref{EstiPar} hold. Consider the function $\psi_{\lambda}$ defined as the Newtonian potential associated to the following equation:
	\begin{align*}
		\left(-\Delta_{(r,z)} - \frac{1}{r} \partial_{r} \right) \psi_{\lambda} 
		= \varphi_{\lambda}^{(6)}\left(|(r,z) - \xi|\right) 
		\chi_{0} \left( \frac{|(r,z) - \xi|}{2\sqrt{\delta(T - t)}} \right), 
		\quad x = (r e^{i\theta}, z).
	\end{align*}
	Then, the gradient $\nabla_{(r,z)} \psi_{\lambda}$ satisfies the following pointwise estimate:
	\begin{align*}
		|\nabla_{(r,z)} \psi_{\lambda}| \lesssim e^{-\sqrt{2|\ln(T-t)|}}
		\begin{cases}
			\displaystyle\frac{1}{\lambda} 
			\frac{|\ln (1+\rho)|}{1 + \rho}, 
			& \text{if } |(r,z) - \xi| \le \sqrt{2\delta(T - t)}, \\[8pt]
			\displaystyle\frac{\sqrt{|\ln(T - t)|}}{|(r,z) - \xi|}, 
			& \text{if }   |(r,z) - \xi| \ge\sqrt{2\delta(T - t)} , \\
		\end{cases}
	\end{align*}
	where $\rho = \frac{ |(r,z) - \xi|}{\lambda}$.
\end{proposition}
\begin{proof}
	The proof follows from the argument used in Theorem~\ref{Expansionv0} and Proposition~\ref{3dGRADIENTExpansion}. As in those cases, we start by examining the associated two dimensional problem:
	\begin{align*}
		-\Delta_{(r,z)} \psi_{\lambda}^{(2d)} = \varphi_{\lambda}^{(6)}\left(|(r,z) - \xi|\right) 
		\chi_{0} \left( \frac{|(r,z) - \xi|}{2\sqrt{\delta(T - t)}} \right).
	\end{align*}
	Letting $\rho = \frac{|(r,z) - \xi|}{\lambda}$ denote the rescaled radial variable, we can explicitly construct a solution as
	\begin{align*}
		\tilde{\psi}_{\lambda}^{(2d)}(\rho) 
		= -\int_{0}^{\rho} \frac{1}{u} \int_{0}^{u} \lambda^{2} \varphi_{\lambda}(s,t) \chi(s,t) \, s \, ds \, du 
		+ \int_{0}^{\frac{\delta \sqrt{T - t}}{\lambda}} \frac{1}{u} \int_{0}^{u} \lambda^{2} \varphi_{\lambda}(s,t) \chi(s,t) \, s \, ds \, du.
	\end{align*}
	The function $\tilde{\psi}_{\lambda}^{(2d)}(\rho)$ can be estimated using Lemma \ref{estimatePhilambda}. 
	We then define our first approximation by
	\begin{align*}
		\psi_{\lambda}^{(2d)}(\rho) 
		= \tilde{\psi}_{\lambda}^{(2d)}(\rho) 
		\chi_{0} \left( \rho  \frac{\lambda}{2\delta\sqrt{T - t}} \right).
	\end{align*}
	To obtain estimates away from the singularity, we proceed in the same manner as in Theorem \ref{Expansionv0} and Proposition \ref{3dGRADIENTExpansion}. For brevity, we omit the detailed computations.
\end{proof}
The following Lemma provides a bound for the new error term that arises after introducing the correction $\varphi_{\lambda}$.
\begin{lemma}\label{EstimateErr1}
	Assume that conditions \eqref{EstiLambda} and \eqref{EstiPar} hold. Then, for all $(r,z,t) \in \mathbb{R}^{+}\times \mathbb{R}\times (0,T)$, the following estimate holds:
	\begin{align*}
		|\lambda^{4} S(u_{1})(r,z,t)| 
		\le C\, e^{-\sqrt{2|\ln(T - t)|}} \frac{\ln(2 + |y|)}{1 + |y|^{6}} 
		\chi_{0} \left( \frac{|(r,z) - \xi|}{2\sqrt{\delta(T - t)}} \right), \ \ \ y=\frac{(r,z)-\xi}{\lambda(t)} .
	\end{align*}
\end{lemma}
\begin{proof}
	The main difference compared to the proof of Lemma 3.3 in \cite{BDdPM} is that the current setting involves localization, as $\varphi_{\lambda}$ has been cut off. However, all computations from the original argument adapt easily to this case, since $\nabla_{(r,z)}v_{0}$ and $\nabla_{(r,z)} \psi_{\lambda}$ remain locally the same functions as in the two-dimensional situation. \newline
	The only additional terms that arise are
	\begin{align*}
		\lambda^{4} \left( \frac{1}{r} \partial_{r} u_{0} + \frac{1}{r} \partial_{r} \varphi_{\lambda} \right).
	\end{align*}
	We estimate each of these terms as follows:
	\begin{align*}
		\left| \lambda^{4} \frac{1}{r} \partial_{r} u_{0} \right| 
		&\lesssim  \lambda \frac{1}{1 + |y|^{5}} \chi 
		\leq \lambda e^{\sqrt{2|\ln(T - t)|}} \frac{1}{1 + |y|^{6}} \ll 
		e^{-\sqrt{2|\ln(T - t)|}} \frac{\ln(2 + |y|)}{1 + |y|^{6}} 
		\chi_{0} \left( \frac{|(r,z) - \xi|}{2\sqrt{\delta(T - t)}} \right),
	\end{align*}
	and similarly,
	\begin{align*}
		\left| \lambda^{4} \frac{1}{r} \partial_{r} \varphi_{\lambda} \right| 
		&\lesssim \lambda e^{-2\sqrt{2|\ln(T - t)|}} \frac{1}{1 + |y|^{3}} \ll 
		e^{-\sqrt{2|\ln(T - t)|}} \frac{\ln(2 + |y|)}{1 + |y|^{6}} 
		\chi_{0} \left( \frac{|(r,z) - \xi|}{2\sqrt{\delta(T - t)}} \right).
	\end{align*}
\end{proof} 
\section{The inner-outer gluing system}\label{innerouterSec}
We aim to construct a solution to the Keller-Segel system \eqref{KS3d} as a small perturbation of $u_{1} = u_{0} + \varphi_{\lambda}$. Specifically, we express it as
\begin{align}\label{Solu}
	u=	u_{1} + \Phi = u_{0} + \varphi_{\lambda} + \Phi = u_{0} + \varphi_{\lambda} + \frac{1}{\lambda^{2}} \phi^{i} \chi + \varphi^{o},
\end{align}
where $\chi = \chi_{0}\left(\frac{|(r,z)-\xi(t)|}{\sqrt{\delta(T-t)}}\right)$. Recalling the definition of the error operator \eqref{erroroperator}, our objective is to determine $\Phi$ such that
\begin{align*}
	S(u_{1} + \Phi) = 0.
\end{align*}
Letting $\Psi(r,z) = (-\Delta_{\mathbb{R}^{3}})(\Phi(r,z))$, we compute
\begin{align*}
	S(u_{1} + \Phi) = S(u_{1}) - \partial_{t}\left(\frac{1}{\lambda^{2}} \phi^{i} \chi \right) - \partial_{t} \varphi^{o} + \mathcal{L}_{u_{1}}\left[\frac{1}{\lambda^{2}} \phi^{i} \chi \right] + \mathcal{L}_{u_{1}}[\varphi^{o}] - \text{div}_{(r,z)}(\Phi \nabla_{(r,z)} \Psi),
\end{align*}
where the linearized operator $\mathcal{L}_{u_{1}}$ is defined as
\begin{align*}
	\mathcal{L}_{u_{1}}[\phi] = \Delta_{(r,z)} \phi + \frac{1}{r} \partial_{r} \phi - \text{div}_{(r,z)}(\phi \nabla_{(r,z)} v_{1}) - \text{div}_{(r,z)}(u_{1} \nabla_{(r,z)} \psi), \quad \psi(r,z) = (-\Delta_{\mathbb{R}^{3}})(\phi(r,z)).
\end{align*}
In what follows, we introduce the following notation:
\begin{align*}
	v_{0} &= \alpha \Gamma_{0} + \mathcal{R}, \ \
	v_{0} =  (-\Delta_{\mathbb{R}^{3}})^{-1} (\frac{\alpha}{\lambda^{2}}U \chi), \ \
	\psi_{\lambda} = (-\Delta_{\mathbb{R}^{3}})^{-1} \varphi_{\lambda}.
\end{align*}
We now expand
\begin{align*}
	\mathcal{L}_{u_{1}}\left[\frac{1}{\lambda^{2}} \phi^{i} \chi \right] =\;& \chi \frac{1}{\lambda^{2}} \Delta_{(r,z)} \phi^{i} + \frac{2}{\lambda^{2}} \nabla_{(r,z)} \chi \cdot \nabla_{(r,z)} \phi^{i} + \frac{1}{\lambda^{2}} \phi^{i} \Delta_{(r,z)} \chi \\
	&+ \frac{1}{\lambda^{2}} \frac{1}{r} \partial_{r}(\phi^{i} \chi) - \text{div}_{(r,z)}\left(\frac{1}{\lambda^{2}} \phi^{i} \chi \nabla_{(r,z)} v_{1}\right) - \text{div}_{(r,z)}(u_{1} \nabla_{(r,z)} \psi^{i}),
\end{align*}
and, recalling the definition of $u_0$ and $v_0$ in \eqref{def-u0-v0}, evaluate
\begin{align*}
	\text{div}_{(r,z)}\left(\frac{1}{\lambda^{2}} \phi^{i} \chi \nabla_{(r,z)} v_{1}\right) =\;& \text{div}_{(r,z)}\left(\frac{1}{\lambda^{2}} \phi^{i} \nabla_{(r,z)} v_{0}\right)\chi + \text{div}_{(r,z)}\left(\frac{1}{\lambda^{2}} \phi^{i} \nabla_{(r,z)} \psi_{\lambda} \right)\chi \\
	&+ \frac{1}{\lambda^{2}} \phi^{i} \nabla_{(r,z)} v_{1} \cdot \nabla_{(r,z)} \chi \\
	=\;& \nabla_{(r,z)}\left(\frac{1}{\lambda^{2}} \phi^{i}\right) \cdot \nabla_{(r,z)} v_{0} \chi + \text{div}_{(r,z)}\left(\frac{1}{\lambda^{2}} \phi^{i} \nabla_{(r,z)} \psi_{\lambda}\right)\chi \\
	&+ \frac{1}{\lambda^{2}} \phi^{i} \nabla_{(r,z)} v_{1} \cdot \nabla_{(r,z)} \chi - \frac{1}{\lambda^{2}} \phi^{i} u_{0} \chi \\
	=\;& \nabla_{(r,z)}\left(\frac{1}{\lambda^{2}} \phi^{i}\right) \cdot \nabla_{(r,z)}(\alpha \Gamma_{0}) \chi + \text{div}_{(r,z)}\left(\frac{1}{\lambda^{2}} \phi^{i} \nabla_{(r,z)} \psi_{\lambda}\right)\chi \\
	&+ \frac{1}{\lambda^{2}} \phi^{i} \nabla_{(r,z)} v_{1} \cdot \nabla_{(r,z)} \chi - \frac{1}{\lambda^{2}} \phi^{i} u_{0} \chi + \nabla_{(r,z)}\left(\frac{1}{\lambda^{2}} \phi^{i}\right) \cdot \nabla_{(r,z)} \mathcal{R} \chi \\
	=\;& \nabla_{(r,z)}\left(\frac{1}{\lambda^{2}} \phi^{i}\right) \cdot \nabla_{(r,z)} \Gamma_{0} \chi + \text{div}_{(r,z)}\left(\frac{1}{\lambda^{2}} \phi^{i} \nabla_{(r,z)} \psi_{\lambda}\right)\chi \\
	&+ \frac{1}{\lambda^{2}} \phi^{i} \nabla_{(r,z)} v_{1} \cdot \nabla_{(r,z)} \chi - \frac{1}{\lambda^{4}} \phi^{i} U \chi + \nabla_{(r,z)}\left(\frac{1}{\lambda^{2}} \phi^{i}\right) \cdot \nabla_{(r,z)} \mathcal{R} \chi \\
	&+ (\alpha - 1) \nabla_{(r,z)}\left(\frac{1}{\lambda^{2}} \phi^{i}\right) \cdot \nabla_{(r,z)} \Gamma_{0} \chi - \frac{\alpha - 1}{\lambda^{4}} \phi^{i} U \chi^{2}+\frac{1}{\lambda^{4}}\phi^{i}U\chi(1-\chi).
\end{align*}
Before carrying out the expansion, we introduce the following notation:
\begin{align}\label{2dInvLapInn}
	\hat{\psi}(r,z) = (-\Delta_{(r,z)})^{-1} \left(\frac{1}{\lambda^{2}} \phi^{i}(r,z)\right), \quad (r,z) \in \mathbb{R}^{2},
\end{align}
which is natural since the inner solution $\phi^{i}$ is defined in $\mathbb{R}^2$.
We then have 
\begin{align*}
	\text{div}_{(r,z)}(u_{1} \nabla_{(r,z)} \psi^{i}) =\;& \nabla_{(r,z)} u_{1} \cdot \nabla_{(r,z)} \psi^{i} - u_{1} \frac{1}{\lambda^{2}} \phi^{i} \chi \\
	=\;& \nabla_{(r,z)} u_{0} \cdot \nabla_{(r,z)} \hat{\psi} + \nabla_{(r,z)} u_{0} \cdot \nabla_{(r,z)}(\psi^{i} - \hat{\psi}) + \nabla_{(r,z)} \varphi_{\lambda} \cdot \nabla_{(r,z)} \psi^{i} \\
	& - u_{0} \frac{1}{\lambda^{2}} \phi^{i} \chi - \varphi_{\lambda} \frac{1}{\lambda^{2}} \phi^{i} \chi \\
	=\;& \frac{\alpha}{\lambda^{2}} \nabla_{(r,z)} U \cdot \nabla_{(r,z)} \hat{\psi} \chi + \frac{\alpha}{\lambda^{2}} U \nabla_{(r,z)} \chi \cdot \nabla_{(r,z)} \hat{\psi} + \nabla_{(r,z)} u_{0} \cdot \nabla_{(r,z)} (\psi^{i} - \hat{\psi}) \\
	& + \nabla_{(r,z)} \varphi_{\lambda} \cdot \nabla_{(r,z)} \psi^{i} - u_{0} \frac{1}{\lambda^{2}} \phi^{i} \chi - \varphi_{\lambda} \frac{1}{\lambda^{2}} \phi^{i} \chi \\
	=\;& \frac{1}{\lambda^{2}} \nabla_{(r,z)} U \cdot \nabla_{(r,z)} \hat{\psi} \chi + \frac{\alpha - 1}{\lambda^{2}} \nabla_{(r,z)} U \cdot \nabla_{(r,z)} \hat{\psi} \chi - \frac{\alpha}{\lambda^{4}} U \phi^{i} \chi^{2} \\
	& + \frac{\alpha}{\lambda^{2}} U \nabla_{(r,z)} \chi \cdot \nabla_{(r,z)} \hat{\psi} + \nabla_{(r,z)} u_{0} \cdot \nabla_{(r,z)} (\psi^{i} - \hat{\psi}) + \nabla_{(r,z)} \varphi_{\lambda} \cdot \nabla_{(r,z)} \psi^{i} - \varphi_{\lambda} \frac{1}{\lambda^{2}} \phi^{i} \chi \\
	=\;& \frac{1}{\lambda^{2}} \nabla_{(r,z)} U \cdot \nabla_{(r,z)} \hat{\psi} \chi - \frac{1}{\lambda^{4}} U \phi^{i} \chi - \frac{\alpha - 1}{\lambda^{4}} U \phi^{i} \chi^{2} - \frac{1}{\lambda^{4}} U \phi^{i} \chi (\chi - 1) \\
	& + \frac{\alpha - 1}{\lambda^{2}} \nabla_{(r,z)} U \cdot \nabla_{(r,z)} \hat{\psi} \chi + \frac{\alpha}{\lambda^{2}} U \nabla_{(r,z)} \chi \cdot \nabla_{(r,z)} \hat{\psi} + \nabla_{(r,z)} u_{0} \cdot \nabla_{(r,z)} (\psi^{i} - \hat{\psi}) \\
	& + \nabla_{(r,z)} \varphi_{\lambda} \cdot \nabla_{(r,z)} \psi^{i} - \varphi_{\lambda} \frac{1}{\lambda^{2}} \phi^{i} \chi.
\end{align*}
Lastly, we note that
\begin{align*}
	\partial_{t}\left(\frac{1}{\lambda^{2}}\phi^{i} \chi\right) &= \frac{1}{\lambda^{2}} \phi^{i} \partial_{t} \chi + \frac{1}{\lambda^{2}} \partial_{t} \phi^{i} \chi - \frac{2 \dot{\lambda}}{\lambda^{3}} \phi^{i} \chi + \frac{1}{\lambda^{2}} \nabla_{y} \phi^{i} \cdot \partial_{t} y \, \chi \\
	&= \frac{1}{\lambda^{2}} \partial_{t} \phi^{i} \chi - \frac{\lambda \dot{\lambda}}{\lambda^{4}} \bigl[ 2 \phi^{i} + y \cdot \nabla_{y} \phi^{i} \bigr]\chi - \frac{1}{\lambda^{3}} \nabla_{y} \phi^{i} \cdot \dot{\xi} \chi + \frac{1}{\lambda^{2}} \phi^{i} \partial_{t} \chi,
\end{align*}
where $y={(r,z)-\xi \over \lambda}$. We now turn to the analysis of the operator
\begin{align*}
	\mathcal{L}_{u_{1}}[\varphi^{o}] = \Delta_{(r,z)} \varphi^{o} + \frac{1}{r} \partial_{r} \varphi^{o} - \text{div}_{(r,z)} \bigl(\varphi^{o} \nabla_{(r,z)} v_{1} \bigr) - \text{div}_{(r,z)} \bigl(u_{1} \nabla_{(r,z)} \psi^{o} \bigr).
\end{align*}
Observe that
\begin{align*}
	\text{div}_{(r,z)} \bigl(\varphi^{o} \nabla_{(r,z)} v_{1}\bigr) &= \nabla_{(r,z)} \varphi^{o} \cdot \nabla_{(r,z)} v_{1} - u_{1} \varphi^{o} \\
	&= \nabla_{(r,z)} \varphi^{o} \cdot \nabla_{(r,z)} v_{0} + \nabla_{(r,z)} \varphi^{o} \cdot \nabla_{(r,z)} \psi_{\lambda} - u_{1} \varphi^{o},
\end{align*}
and also
\begin{align*}
	\text{div}_{(r,z)} \bigl(u_{1} \nabla_{(r,z)} \psi^{o}\bigr) = \nabla_{(r,z)} u_{1} \cdot \nabla_{(r,z)} \psi^{o} - u_{1} \varphi^{o}.
\end{align*}
We now define the inner equation, which governs all the terms multiplied by the cut-off function:
\begin{equation}\label{innerEq0}
\begin{aligned}
	\frac{1}{\lambda^{2}} \partial_{t} \phi^{i} \chi 
	&= \frac{1}{\lambda^{2}} \Delta_{(r,z)} \phi^{i} \chi
	+ \frac{1}{\lambda^{4}} B[\phi^{i}] \chi 
	-   \nabla_{(r,z)}\left( \frac{1}{\lambda^{2}} \phi^{i} \right) \cdot \nabla_{(r,z)} \Gamma_{0} \chi 
	\\
    &- \frac{1}{\lambda^{2}} \nabla_{(r,z)} U \cdot \nabla_{(r,z)} \hat{\psi} \chi 
	+ \frac{2}{\lambda^{4}} U \phi^{i} \chi 
	+ E(r,z,t),
\end{aligned}
\end{equation}
where
\begin{align}\label{OperatorB}
	B[\phi^{i}] = \lambda \dot{\lambda} \left( 2 \phi^{i} + y \cdot \nabla_{y} \phi^{i} \right),
\end{align}
and the error term \( E(r,z,t) \) is given by
\begin{align*}
	E(r,z,t) = \ & S(u_{1}) \chi 
	+ \frac{1}{r} \partial_{r} \left( \frac{1}{\lambda^{2}} \phi^{i} \right) \chi 
	+ \frac{1}{\lambda^{3}} \nabla_{y} \phi^{i} \cdot \dot{\xi} \chi 
	- \frac{\alpha - 1}{\lambda^{4}} U \phi^{i} \chi 
	+ 2 \frac{\alpha - 1}{\lambda^{4}} \phi^{i} U \chi 
	+ \frac{\alpha - 1}{\lambda^{4}} U \phi^{i} \chi^{2} \\
	& -(\alpha - 1) \nabla_{(r,z)} \left( \frac{1}{\lambda^{2}} \phi^{i} \right) \cdot \nabla_{(r,z)} \Gamma_{0} \chi-\nabla_{(r,z)}u_{0}\cdot \nabla_{(r,z)}(\psi^{i}-\hat{\psi})\chi \\
	& + \left( - \nabla_{(r,z)} \varphi^{o} \cdot \nabla_{(r,z)} \psi_{\lambda} + u_{1} \varphi^{o} \right) \chi 
	+ \left( - \nabla_{(r,z)} u_{1} \cdot \nabla_{(r,z)} \psi^{o} + u_{1} \varphi^{o} \right) \chi \\
	& + \left( - \nabla_{(r,z)} \varphi_{\lambda} \cdot \nabla_{(r,z)} \psi^{i} + \varphi_{\lambda} \frac{1}{\lambda^{2}} \phi^{i} \chi \right) \chi 
	- \text{div}_{(r,z)} \left( \frac{1}{\lambda^{2}} \phi^{i} \nabla_{(r,z)} \psi_{\lambda} \right) \chi \\
	& - \text{div}_{(r,z)} \left( \left( \frac{1}{\lambda^{2}} \phi^{i} \chi + \varphi^{o} \right) \nabla_{(r,z)} (\psi^{i} + \psi^{o}) \right) \chi.
\end{align*}
By collecting all the remaining terms, we obtain the outer problem:
\begin{align}\label{outerEq0}
	\partial_{t} \varphi^{o} = \Delta_{(r,z)} \varphi^{o} + \frac{1}{r} \partial_{r} \varphi^{o} - \nabla_{(r,z)} \varphi^{o} \cdot \nabla_{(r,z)} v_{0} + F(r,z,t)
\end{align}
where
\begin{align*}
	F(r,z,t)=&S(u_{1})(1-\chi)+\frac{1}{\lambda^{2}}\phi^{i}\partial_{r}\chi\\
	&+\frac{2}{\lambda^{2}}\nabla_{(r,z)}\chi \cdot \nabla_{(r,z)}\phi^{i}+\frac{1}{\lambda^{2}}\phi^{i}\Delta_{(r,z)}\chi-\frac{1}{\lambda^{2}}\phi^{i}\nabla_{(r,z)}v_{1}\cdot \nabla_{(r,z)}\chi +\\
	&-\nabla_{(r,z)}(\frac{1}{\lambda^{2}}\phi)\cdot \nabla_{(r,z)}\mathcal{R}\chi+\frac{2}{\lambda^{4}}U\phi^{i}\chi(\chi-1)-\frac{\alpha}{\lambda^{2}}U\nabla_{(r,z)}\chi \cdot \nabla_{(r,z)}\hat{\psi}-\\
	&-\nabla_{(r,z)}u_{0}\cdot \nabla_{(r,z)}(\psi^{i}-\hat{\psi})(1-\chi)+(-\nabla_{(r,z)}\varphi_{\lambda}\cdot \nabla_{(r,z)}\psi^{i}+\varphi_{\lambda}\frac{1}{\lambda^{2}}\phi^{i}\chi)(1-\chi)-\\
	&-\frac{1}{\lambda^{2}}\phi^{i}\partial_{t}\chi+(-\nabla_{(r,z)}\varphi^{o}\cdot \nabla_{(r,z)}\psi_{\lambda}+u_{1}\varphi^{o})(1-\chi)+(-\nabla_{(r,z)}u_{1}\cdot \nabla_{(r,z)}\psi^{o}+u_{1}\varphi^{o})(1-\chi)-\\
	&-\text{div}_{(r,z)}((\frac{1}{\lambda^{2}}\phi^{i}\chi+\varphi^{o})\nabla_{(r,z)}(\psi^{i}+\psi^{o}))(1-\chi).
\end{align*}
The inner equation \eqref{innerEq0} can be slightly reformulated as
\begin{align}\label{innerEq1}
	\lambda^{2}\partial_{t}\phi^{i}\chi = L[\phi^{i}]\chi + B[\phi^{i}]\chi + \hat{E}_{1}\tilde{\chi} + \hat{E}_{2}[\phi^{i},\varphi^{o},\mathbf{p}]\,\tilde{\chi},
\end{align}
where the operator $L$ is given by
\begin{align}\label{OperatorL}
	L[\phi] = \Delta_{y}\phi - \nabla_{y}\cdot(U \nabla \psi) - \nabla_{y}\cdot(\phi \nabla \Gamma_0), \quad \psi = (-\Delta_{\mathbb{R}^{2}})^{-1}(\phi), \quad y = \frac{(r,z)-\xi(t)}{\lambda},
\end{align}
and we introduced the cut-off function
\begin{align*}
	\tilde{\chi} = \chi_{0}\left(\frac{|(r,z)-\xi(t)|}{2\sqrt{\delta(T-t)}}\right),
\end{align*}
with $\chi_0$ defined in \eqref{cutoff}. The operator $B$ is defined in \eqref{OperatorB}, and we denote
\begin{align}\label{E1andP}
	\hat{E}_{1}(y,t) = \lambda^{4} S(u_{1}(\mathbf{p})) \chi, \qquad \mathbf{p} = (\lambda, \alpha, \xi).
\end{align}
It is straightforward to verify that if $\phi^{i}$ and $\varphi^{o}$ are solutions of \eqref{innerEq1} and \eqref{outerEq0}, respectively, then the function $u$ defined by \eqref{Solu} satisfies \eqref{KS3d}. To solve this system certain conditions are required  and this is precisely the role of the parameter $\mathbf{p}$ introduced in \eqref{E1andP}. The initial conditions, as well as further modifications to equations \eqref{innerEq1} and \eqref{outerEq0}, will be discussed in the next Sections. 
\subsection{The choice of $\alpha_{0}$, $\lambda_{0}$, $\xi_{0}$}\label{Sectionalpha0lambda0}
In this Section we introduce the first approximation of the parameters $\alpha$, $\lambda$, and $\xi$ in the context of the elliptic equation
\begin{align}\label{ellipticEq}
	L[\phi^{i}] = -\lambda^{4} S(u_{1}(\mathbf{p})) \tilde{\chi}, \quad \text{in } \mathbb{R}^{2},
\end{align}
where $\mathbf{p}$ is defined as in \eqref{E1andP}, and $L$ is the operator introduced in \eqref{OperatorL}. We also use the cut-off function $\tilde{\chi}$ defined earlier. 
We observe that the operator $L$ is invariant under dilations and translations. As will be shown in Section~\ref{lastimprovsolSec}, it will be necessary to introduce a final refinement to the ansatz $u_{1}$ in order to eliminate the leading order term on the right-hand side of \eqref{innerEq1}.
We now recall the statement of Lemma 4.1 in \cite{BDdPM} (also Lemma 5.1 in \cite{DdPDMW}), where the interested reader can find the proof.
\begin{lemma}\label{Lemma51}
	Let $h(y)$ be a radial function satisfying
	\begin{align*}
		\|(1+|y|)^{\gamma} h(y)\|_{L^{\infty}(\mathbb{R}^{2})} < \infty,
	\end{align*}
	for some $\gamma > 4$, and
	\begin{align}\label{zeromassCondLemma51}
		\int_{\mathbb{R}^{2}} h(y) \, dy = 0,
	\end{align}
	\begin{align}\label{zerosecmomCondLemma51} \int_{\mathbb{R}^{2}} h(y) |y|^{2} \, dy = 0.
	\end{align}
	Then there exists a radial solution $\phi(y)$ to the equation
	\begin{align*}
		L[\phi] = h \quad \text{in } \mathbb{R}^{2},
	\end{align*}
	such that
	\begin{align*}
		|\phi(y)| \leq C \left\|(1+|y|)^{\gamma} h(y) \right\|_{L^{\infty}(\mathbb{R}^{2})} \cdot \frac{1}{(1+|y|)^{\gamma-2}}, \quad \text{if } \gamma \neq 6,
	\end{align*}
	and
	\begin{align*}
		|\phi(y)| \leq C \left\|(1+|y|)^{\gamma} h(y) \right\|_{L^{\infty}(\mathbb{R}^{2})} \cdot \frac{\log(2+|y|)}{1+|y|^{4}}, \quad \text{if } \gamma = 6,
	\end{align*}
	with the additional condition
	\begin{align}\label{zeromassLemma51}
		\int_{\mathbb{R}^{2}} \phi(y) \, dy = 0.
	\end{align}
\end{lemma}
\begin{proof}
	See the proof of Lemma 5.1 in \cite{DdPDMW}.
\end{proof}
We remark that condition \eqref{zerosecmomCondLemma51} is crucial in order to obtain \eqref{zeromassLemma51}. Indeed, it can be shown that the mass of the solution is proportional to the second moment of the right-hand side. 
Later we will show that, at leading order, $\mathbf{p} \approx \mathbf{p}_{0}= (\lambda_{0}, \alpha_{0}, (q_{1}(T),0))$ for some parameters $\lambda_{0}$ and $\alpha_{0}$ satisfying \eqref{EstiLambda} and \eqref{EstiPar}.
In order to cancel the dominant terms on the right-hand side of \eqref{innerEq1}, we aim to apply Lemma~\ref{Lemma51} to solve the elliptic equation
\begin{align}\label{EllipticEq0}
	L[\phi^{i}] = -\lambda_{0}^{4} S(u_{1}(\mathbf{p}_{0})) \tilde{\chi}, \quad \text{in } \mathbb{R}^{2},
\end{align}
where we use the cut-off function 
\begin{align}\label{cutoff1}
	\tilde{\chi} = \chi_{0}\left( \frac{|(r, z) - (q_{1}(T), 0)|}{2\sqrt{\delta(T - t)}} \right).
\end{align}
The following result will allow us to choose $\alpha_{0}$ and $\lambda_{0}$ such that, for all $t \in (0, T)$,
\begin{align}\label{zeromassErr1}
	\int_{\mathbb{R}^{2}} S(u_{1}(\mathbf{p}_{0})) \tilde{\chi} \, dr\,dz = 0,
\end{align}
and for some $\sigma > 0$,
\begin{align}\label{smallSecMomErr1}
	\int_{\mathbb{R}^{2}} S(u_{1}(\mathbf{p}_{0})) |(r, z) - (q_{1}(T), 0)|^{2} \, dr\,dz = O\left(e^{-(\frac{3}{2} + \sigma)\sqrt{2|\ln(T - t)|}} \right).
\end{align}
Notice that in \eqref{smallSecMomErr1} we do not impose the vanishing of the second moment. Instead, we aim to make it sufficiently small, as we can only solve the corresponding equation up to a controlled error. For this reason, in Section \ref{finalinneroutSec}, we will project the right-hand side, introducing a small error that is well controlled within our framework. \newline
In the proposition below, it is crucial to appropriately choose the initial time $-\varepsilon(T)$ in \eqref{philambdaEqt2D}. We provide here an intuitive explanation of its role, while a detailed analysis is presented in Section 4 of \cite{BDdPM}. \newline
\underline{The role of $\varepsilon(T)$:} As explained in Section 4 of \cite{BDdPM}, in order to obtain \eqref{zeromassErr1} and \eqref{smallSecMomErr1}, it is necessary to solve the following \emph{nonlocal} equation for $\lambda_{0}\dot{\lambda}_{0}(t)$ over the interval $t \in (0, T)$:
\begin{align*}
	-\int_{-\varepsilon(T)}^{t-\lambda_{0}^{2}}\frac{\lambda_{0}\dot{\lambda}_{0}(s)}{t-s}\,ds + (\gamma + 1 - \ln 4)\lambda_{0}\dot{\lambda}_{0}(t) + \int_{-\varepsilon(T)}^{T}\frac{\lambda_{0}\dot{\lambda}_{0}(s)}{T-s}\,ds \approx 0
\end{align*}
where $\gamma$ denotes the Euler–Mascheroni constant. To make sense of this equation, it is clear, by simply evaluating at $t=0$, that a proper understanding of equation \eqref{philambdaEqt2D} for negative times is required.

\begin{proposition}\label{Proplambda0alpha0}
	There exist functions $\lambda_{0}(t)$ and $\alpha_{0}(t)$ such that, for some $T \ll \varepsilon(T) \ll 1$ as defined in \eqref{philambdaEqt2D}, and for any $\rho > 0$, the following properties hold:
	\begin{align*}
		\int_{\mathbb{R}^{2}} S(u_{1}(\mathbf{p}_{0})) \, \tilde{\chi} \, dr\,dz = 0, \quad t \in (0, T),
	\end{align*}
	and
	\begin{align*}
		\int_{\mathbb{R}^{2}} S(u_{1}(\mathbf{p}_{0})) \, |(r, z) - (q_{1}(T), 0)|^{2} \, dr\,dz = O\left(e^{-(2 - \rho)\sqrt{2|\ln(T - t)|}}\right), \quad t \in (0, T).
	\end{align*}
	Moreover, as $t \to T$, the parameter $\lambda_{0}(t)$ admits the asymptotic behavior:
	\begin{align*}
		\lambda_{0}(t) &= 2 e^{-\frac{\gamma + 2}{2}} \sqrt{T - t} \, e^{-\sqrt{\frac{|\ln(T - t)|}{2}}} \left(1 + O\left(\frac{1}{|\ln(T - t)|^{1/8}}\right) \right), \\
		\dot{\lambda}_{0}(t) &= - e^{-\frac{\gamma + 2}{2}} \frac{1}{\sqrt{T - t}} \, e^{-\sqrt{\frac{|\ln(T - t)|}{2}}} \left(1 + O\left(\frac{1}{|\ln(T - t)|^{1/8}} \right)\right), \\
		|\ddot{\lambda}_{0}(t)| &\lesssim \frac{1}{(T - t)^{3/2}} \, e^{-\sqrt{\frac{|\ln(T - t)|}{2}}}.
	\end{align*}
	In addition, the function $\alpha_{0}(t)$ satisfies the decay estimate:
	\begin{align*}
		|\alpha_{0}(t) - 1| \lesssim e^{-(2 - \rho)\sqrt{2|\ln(T - t)|}}.
	\end{align*}
\end{proposition}
\begin{proof}
	The key observation is that, due to the presence of the cut-off function, we can approximate the term $S(u_{1}(\mathbf{p}_{0}))$ by $S^{(2d)}(\mathbf{p}_{0})$ (see the definition in \eqref{S2dERR}). A similar approximation holds for the three-dimensional inverse Laplacian operator, which can be effectively replaced by its two-dimensional counterpart in the region of interest. Furthermore, within this region, the function $\varphi_{\lambda}$ coincides with the solution to the equation \eqref{philambdaEqt2D}. \newline
	With these preliminary reductions in place, the detailed proof follows directly from the arguments presented in Section 4 of \cite{BDdPM}.
\end{proof}
\subsection{A further improvement of the approximation}\label{lastimprovsolSec} We now introduce the elliptic approximation by considering the solution to the equation
\begin{align}\label{EllipticEq1}
	L[\phi_{0}^{i}] = -\lambda_{0}^{4}S(u_{1}(\mathbf{p}_{0}))\tilde{\chi} + c_{0}(t)W_{2}(\bar{y}_{0}), \qquad \bar{y}_{0} = \frac{(r,z) - (q_{1}(T), 0)}{\lambda_{0}} \in \mathbb{R}^{2},
\end{align}
where \( L \) is the operator defined in \eqref{OperatorL}. In \eqref{EllipticEq1}, the time variable \( t \) is regarded as a parameter. The function \( W_{2}(\bar{y}_{0}) \) is a fixed, smooth, radial function with compact support, chosen to satisfy the normalization conditions:
\begin{align}\label{W2}
	\int_{\mathbb{R}^{2}} W_{2}(y) \, dy = 0, \qquad \int_{\mathbb{R}^{2}} W_{2}(y) |y|^{2} \, dy = 1.
\end{align}
From Proposition \ref{Proplambda0alpha0} we know that for any \( t \in (0, T) \) and for any \( \rho > 0 \), the correction coefficient satisfies
\begin{align}\label{c0estimate}
	|c_{0}(t)| \leq C e^{-(2 - \rho)\sqrt{2|\ln(T - t)|}}.
\end{align}
Moreover, we have the mass cancellation condition
\[
\int_{\mathbb{R}^{2}} S(u_{1}(\mathbf{p}_{0})) \tilde{\chi} \, dr \, dz = 0.
\]
Thanks to Lemma \ref{Lemma51} and Lemma \ref{EstimateErr1}, we can construct a solution \( \phi_{0}^{i} \) to \eqref{EllipticEq1} satisfying the decay and orthogonality estimates:
\begin{align}\label{estimateEllCorr}
	|\phi_{0}^{i}(\bar{y})| \leq C e^{-\sqrt{2|\ln(T-t)|}} \sqrt{|\ln(T-t)|} \frac{\ln(2 + |\bar{y}|)}{1 + |\bar{y}|^{4}}, \qquad \int_{\mathbb{R}^{2}} \phi_{0}^{i}(y) \, dy = 0.
\end{align}
\subsection{Reformulation of the system}\label{RefInnerOutGluingSystem}
We decompose the inner solution and the parameters as
\begin{align}\label{ImprovedSols}
	\begin{cases}
		\phi^{i} = \phi_{0}^{i} + \phi, \\
		\mathbf{p} = \mathbf{p}_{0} + \mathbf{p}_{1},
	\end{cases}
\end{align}
where \( \mathbf{p}_{0} = (\lambda_{0}(t), \alpha_{0}(t), (q_{1}(T), 0)) \) is defined in Section~\ref{Proplambda0alpha0}, and \( \mathbf{p}_{1} = (\lambda_{1}(t), \alpha_{1}(t), \xi_{1}(t)) \) denotes the correction. Recalling that \( \phi_{0}^{i} \) solves equation \eqref{EllipticEq1}, substituting the ansatz \eqref{ImprovedSols} into \eqref{innerEq1} and \eqref{outerEq0} leads to the following system:
\begin{align}\label{innerEq2}
	\lambda^{2} \partial_{t} \phi \, \chi = L[\phi] \, \chi + B[\phi^{i}] \, \chi + \hat{E}_{2} \, \tilde{\chi}_{2} + \hat{E}_{3}[\phi, \varphi^{o}, \mathbf{p}] \, \tilde{\chi},
\end{align}
\begin{align}\label{outerEq1}
	\partial_{t} \varphi^{o} = \Delta_{(r,z)} \varphi^{o} + \frac{1}{r} \partial_{r} \varphi^{o} - \nabla_{(r,z)} v_{0} \cdot \nabla_{(r,z)} \varphi^{o} + G_{1}(\phi, \varphi^{o}, \mathbf{p}_{1}).
\end{align}
In equation \eqref{innerEq2} \( \tilde{\chi} \) is the cut-off introduced in \eqref{cutoff1}. We now introduce an additional cut-off:
\begin{align}\label{cutoff2}
	\tilde{\chi}_{2}(r, z, t) = \chi_{0}\left( \frac{|(r, z) - \xi|}{\sqrt{T - t}} e^{\mu \sqrt{2|\ln(T - t)|}} \right),
\end{align}
where \( \chi_{0} \) is the function defined in \eqref{cutoff} and \( \mu > 0 \) is a small parameter to be fixed.
We define the error term
\begin{align}\label{E2hat}
	\hat{E}_{2} = E_{2} \chi = \left( -\lambda^{2} \partial_{t} \phi_{0}^{i} + B[\phi_{0}^{i}] + c_{0}(t) W_{2}(y) \right) \chi,
\end{align}
and the nonlinear remainder
\begin{align}\label{E(1)}
	\hat{E}^{(1)}(\phi, \varphi^{o}, \mathbf{p}_{1}) = \, & \left[ \hat{E}_{3}(\phi_{0}^{i} + \phi, \varphi^{o}, \mathbf{p}_{0} + \mathbf{p}_{1}) + \lambda^{4} S(u_{1}(\mathbf{p}_{0} + \mathbf{p}_{1})) \chi - \lambda_{0}^{4} S(u_{1}(\mathbf{p}_{0})) \tilde{\chi} \chi \right] \nonumber \\
	& + \left( \hat{E}_{2}(\bar{y}) - \hat{E}_{2}(y) \right) \tilde{\chi}_{2},
\end{align}
where the variable \( \bar{y} \) arises from the change of variables around \( \mathbf{p}_{0} \), as defined previously.
Finally, the outer right-hand side is given by
\begin{align}\label{outererrorFinal}
	g(\phi, \varphi^{o}, \mathbf{p}_{1}) = G_{1}(\phi_{0}^{i} + \phi, \varphi^{o}, \mathbf{p}_{0} + \mathbf{p}_{1}) + \lambda^{-4} E_{2} (1 - \tilde{\chi}_{2}) \chi,
\end{align}
where \( \hat{E}_{1} \) and \( G_{1} \) are the nonlinear terms appearing on the right-hand sides of equations \eqref{innerEq2} and \eqref{outerEq1},respectively. \newline
Assuming estimate \eqref{EstiPar}, we observe that for all \( |y| \le \frac{2\sqrt{\delta(T - t)}}{\lambda(t)} \), the following estimate holds:
\begin{align*}
	|E_{2}(y, t)| \lesssim e^{-2\sqrt{2|\ln(T - t)|}} \sqrt{|\ln(T - t)|} \frac{\ln(2 + |y|)}{1 + |y|^{4}} + e^{-(2 - \rho)\sqrt{2|\ln(T - t)|}} |W_{2}(y)|.
\end{align*}
We remark that, in order to control \( \partial_{t} \phi_{0}^{i} \), it is necessary to estimate the time derivative of the right-hand side of \eqref{ellipticEq}. This can be done directly thanks to the available estimates on the time derivatives of \( \lambda_{0} \) and \( \alpha_{0} \) (see Corollary 4.7 in \cite{BDdPM} for further details). The presence of the cut-off \( \tilde{\chi}_{2} \) allows us to obtain
\begin{align}\label{EstimateE2Err}
	|E_{2}(y, t)\tilde{\chi}_{2}| \lesssim \frac{e^{-\nu \sqrt{2|\ln(T - t)|}}}{(1 + |y|)^{6 + \sigma}}
\end{align}
for any \( \nu < 1 + 2\mu - \frac{\sigma}{2} + \sigma \mu \). We will choose \( \mu \) and \( \sigma \) to be small positive constants such that
\[
2\mu - \frac{\sigma}{2} > 0,
\]
so that
\[
1 < \nu < 1 + 2\mu - \frac{\sigma}{2}.
\]
At the same time, we observe that
\begin{align*}
	|\lambda^{-4} E_{2} (1 - \tilde{\chi}_{2}) \chi| \le \frac{e^{-2(1 - 2\mu)\sqrt{2|\ln(T - t)|}}}{(T - t)^{2}} |\ln(T - t)| \, \chi.
\end{align*}
We anticipate (as clarified in greater detail in \cite{BDdPM}) that, in order to solve the final inner-outer system it is necessary to fix \( \nu \) such that condition \eqref{EstimateE2Err} is satisfied ensuring also that
\[
2(1 - 2\mu) > \nu + \frac{1}{2},
\]
in order to control the error in the outer region. For this reason we will take \( \nu > 1 \) but close to \( 1 \), and choose \( \mu > 0 \) sufficiently small.

\subsection{The inner-outer gluing system}\label{finalinneroutSec}
In this Section we decompose \( \phi \) into a radial part \( [\phi]_{\text{rad}} \), defined by
\begin{align}\label{radialinn}
	[\phi]_{\text{rad}}(|y|, t) = \frac{1}{2\pi} \int_{0}^{2\pi} \phi(|y|e^{i\theta}, t)\, d\theta,
\end{align}
and a remainder term \( \phi_{3} \) with no radial mode:
\begin{align}\label{noradialinn}
	\phi_{3} = \phi - [\phi]_{\text{rad}}.
\end{align}
We further decompose the radial component as
\begin{align*}
	[\phi]_{\text{rad}} = \phi_{1} + \phi_{2}.
\end{align*}
Here, the function \( \phi_{1} \) will solve an equation involving the radial part of the right-hand side of \eqref{innerEq2}, projected so that it has zero mass and zero second moment. The function \( \phi_{2} \), on the other hand, will correct for the remaining second moment, which is not identically zero but is sufficiently small.
We now introduce another cut-off function:
\begin{align}\label{cutoff3}
	\hat{\chi}(r, z, t) = \chi_{0}\left(\frac{|(r, z) - \xi(t)|}{4\sqrt{\delta(T - t)}}\right),
\end{align}
which allows us to localize the operator \( B \) by observing that
\[
B[\phi]\chi = B[\phi \hat{\chi}]\chi.
\]
By performing this localization of \( B \), we introduce additional terms in the equations: specifically, the operator \( \mathcal{E}[\phi_{1}] \) appears in the equation for \( \phi_{1} \), and \( \mathcal{F}[\phi_{3}] \) arises in the equation for \( \phi_{3} \). For further discussion, see Section~\ref{reviewInnerTheories}.
These operators are the same as those studied in Sections 10 and 11 of \cite{DdPMW}.\newline
Recalling \eqref{E2hat} and \eqref{E(1)}, and setting \( E^{(2)} = E^{(1)} + \hat{E}_{2}\tilde{\chi}_{2} \), we define the notation
\begin{align*}
	\hat{E}_{4} = \hat{E}_{3}[\phi_{1} + \phi_{2} + \phi_{3}, \varphi^{o}, \mathbf{p}_{1}] - \mathcal{E}[\phi_{1}]\chi - \mathcal{F}[\phi_{3}]\chi.
\end{align*}
We can then split the equation \eqref{innerEq2} as follows:
\begin{align*}
	\lambda^{2} \partial_{t}(\phi_{1} + \phi_{2})\chi 
	= L[\phi_{1} + \phi_{2}]\chi 
	+ B[(\phi_{1} + \phi_{2})\hat{\chi}]\chi 
	+ \mathcal{E}[\phi_{1}]\chi 
	+ \left( 	\hat{E}_{4}[\phi_{1} + \phi_{2} + \phi_{3}, \varphi^{o}, \mathbf{p}_{1}]_{\text{rad}} \right) \tilde{\chi},
\end{align*}
\begin{align*}
	\lambda^{2} \partial_{t} \phi_{3}\chi 
	= L[\phi_{3}]\chi 
	+ B[\phi_{3} \hat{\chi}]\chi 
	+ \mathcal{F}[\phi_{3}]\chi 
	+ \left( 	\hat{E}_{4}[\phi_{1} + \phi_{2} + \phi_{3}, \varphi^{o}, \mathbf{p}_{1}] - 	\hat{E}_{4}[\cdot]_{\text{rad}} \right) \tilde{\chi},
\end{align*}
where \( 	\hat{E}_{4}[\cdot]_{\text{rad}} \) denotes the radial projection, as defined in \eqref{radialinn}. \newline
Finally, simplifying the cut-off \( \chi \) and defining \( F_{0} = \hat{E}_{4}/\chi \), we obtain:
\begin{align*}
	\lambda^{2} \partial_{t}(\phi_{1} + \phi_{2}) 
	= L[\phi_{1} + \phi_{2}] 
	+ B[(\phi_{1} + \phi_{2})\hat{\chi}] 
	+ \mathcal{E}[\phi_{1}] 
	+ \left[ F_{0}(\phi_{1} + \phi_{2} + \phi_{3}, \varphi^{o}, \mathbf{p}_{1}) \right]_{\text{rad}} \tilde{\chi},
\end{align*}
\begin{align*}
	\lambda^{2} \partial_{t} \phi_{3} 
	= L[\phi_{3}] 
	+ B[\phi_{3} \hat{\chi}] 
	+ \mathcal{F}[\phi_{3}] 
	+ \left( F_{0}(\phi_{1} + \phi_{2} + \phi_{3}, \varphi^{o}, \mathbf{p}_{1}) - \left[ F_{0}(\cdot) \right]_{\text{rad}} \right) \tilde{\chi}.
\end{align*}
As in \cite{BDdPM}, for any function \( h(y,t) \) with sufficient spatial decay, we define:
\begin{align}
	m_{0}[h](t) = \int_{\mathbb{R}^{2}} h(y,t)\, dy, \qquad 
	m_{2}[h](t) = \int_{\mathbb{R}^{2}} h(y,t) |y|^{2}\, dy,
\end{align}
\begin{align*}
	m_{1,j}[h](t) = \int_{\mathbb{R}^{2}} h(y,t) y_{j}\, dy, \quad j = 1, 2,
\end{align*}
which correspond to the mass, second moment, and center of mass of \( h \), respectively. Let \( W_{0} \in C^{\infty}(\mathbb{R}^{2}) \) be a radial function with compact support such that
\begin{align*}
	\int_{\mathbb{R}^{2}} W_{0}(y)\, dy = 1, \qquad 
	\int_{\mathbb{R}^{2}} W_{0}(y) |y|^{2}\, dy = 0.
\end{align*}
Let \( W_{1,j} \), for \( j = 1,2 \), be smooth compactly supported functions of the form \( W_{1,j}(y) = \tilde{W}(|y|) y_{j} \), satisfying
\begin{align*}
	\int_{\mathbb{R}^{2}} W_{1,j}(y) y_{j}\, dy = 1.
\end{align*}
We also recall the definition of \( W_{2} \) from \eqref{W2}. Then, we have the following properties:
\( h - m_{0}[h] W_{0} \) has zero mass,
\( h - m_{2}[h] W_{2} \) has zero second moment,
\( h - m_{1,1}[h] W_{1,1} - m_{1,2}[h] W_{1,2} \) has zero center of mass. \newline
We can now express the system as follows:
\begin{align}\label{innerEqFINmode0}
	\begin{cases}
		\lambda^{2} \partial_{t} \phi_{1} = L[\phi_{1}] + B[\phi_{1}] + \mathcal{E}[\phi_{1}] + f - m_{0}[f] W_{0} - m_{2}[f] W_{2}, & \text{in } \mathbb{R}^{2} \times (0,T), \\
		\phi_{1}(\cdot, 0) = \phi_{1,0}, & \text{in } \mathbb{R}^{2},
	\end{cases}
\end{align}
\begin{align}\label{innerEQfinalMode0NOORT}
	\begin{cases}
		\lambda^{2} \partial_{t} \phi_{2} = L[\phi_{2}] + B[\phi_{2} \hat{\chi}] + m_{2}[f] W_{2}, & \text{in } \mathbb{R}^{2} \times (0,T), \\
		\phi_{2}(\cdot, 0) = \phi_{2,0}, & \text{in } \mathbb{R}^{2},
	\end{cases}
\end{align}
\begin{align}\label{innerEqfinalMode1}
	\begin{cases}
		\lambda^{2} \partial_{t} \phi_{3} = L[\phi_{3}] + B[\phi_{3} \hat{\chi}] + \mathcal{F}[\phi_{3}] + f_{3} - m_{1,j}[f_{3}] W_{1,j}, & \text{in } \mathbb{R}^{2} \times (0,T), \\
		\phi_{3}(\cdot, 0) = \phi_{3,0}, & \text{in } \mathbb{R}^{2},
	\end{cases}
\end{align}
where the source terms \( f \) and \( f_{3} \) are defined by
\begin{align*}
	f[\phi_{1} + \phi_{2} + \phi_{3}, \varphi^{o}, \mathbf{p}_{1}] 
	&= \left[ F_{0}(\phi_{1} + \phi_{2} + \phi_{3}, \varphi^{o}, \mathbf{p}_{1}) \right]_{\text{rad}} \tilde{\chi}, \\
	f_{3}[\phi_{1} + \phi_{2} + \phi_{3}, \varphi^{o}, \mathbf{p}_{1}] 
	&= \left( F_{0}[\phi_{1} + \phi_{2} + \phi_{3}, \varphi^{o}, \mathbf{p}_{1}] - \left[ F_{0}(\cdot) \right]_{\text{rad}} \right) \tilde{\chi}.
\end{align*}
We aim to determine parameters \( \alpha_{1} \), \( \xi_{1}(t) \) (defined in \( (0,T) \)), and \( \lambda_{1}(t) \) (defined in \( (-\varepsilon(T), T) \)) such that the following orthogonality conditions are satisfied:
\begin{align*}
	\begin{cases}
		m_{0}[f[\phi_{1} + \phi_{2} + \phi_{3}, \varphi^{o}, \mathbf{p}_{1}]] = 0, \\
		m_{1,j}[f_{3}[\phi_{1} + \phi_{2} + \phi_{3}, \varphi^{o}, \mathbf{p}_{1}]] = 0, & \text{for } j = 1, 2, \quad t \in (0,T).
	\end{cases}
\end{align*}
Using the definition given in \eqref{outererrorFinal}, we observe that the inner system is coupled with the following outer equation:
\begin{align}\label{outerEqfinal}
	\begin{cases}
		\partial_{t} \varphi^{o} = \Delta_{(r,z)} \varphi^{o} + \frac{1}{r} \partial_{r} \varphi^{o} - \nabla v_{0} \cdot \nabla \varphi^{o} + g(\phi_{1} + \phi_{2} + \phi_{3}, \varphi^{o}, \mathbf{p}_{1}), \quad \text{in } \mathbb{R}^{+} \times \mathbb{R} \times (0, T), \\
		\varphi^{o}(\cdot, 0) = 0, \quad \text{in } \mathbb{R}^{+} \times \mathbb{R}.
	\end{cases}
\end{align}
\underline{The role of $\lambda_{1}(t)$:} As anticipated in previous sections, the function $\lambda_{1}$ is chosen to ensure that the second moment becomes sufficiently small. To achieve this, we select $\lambda_{1}$ so that the following approximate identity holds:
\begin{align}\label{ExLambda1}
  & 4 \int_{\mathbb{R}^{2}}(\varphi_{\lambda_{0}+\lambda_{1}} - \varphi_{\lambda_{0}})\,dx 
  - 4 \int_{\mathbb{R}^{2}}(\varphi_{\lambda_{0}+\lambda_{1}}(T) - \varphi_{\lambda_{0}}(T))\,dx \notag \\
  &\quad - 64\pi\beta \frac{(\lambda_{0}+\lambda_{1})^{2} - \lambda_{0}^{2}}{\delta(T - t)} 
  \approx \left(-4 + \frac{1}{\pi} \int_{\mathbb{R}^{2}}\varphi_{\lambda}(T)\,dx\right) 
  \int_{t}^{T} \int_{\mathbb{R}^{2}} U \varphi_{\lambda} \tilde{\chi} \,dy\,ds,
\end{align}
where 
\[
\beta = \int_{0}^{\infty} \frac{1 - \chi_{0}(s)}{s^{3}}\,ds.
\]
The task of finding an approximate solution to \eqref{ExLambda1}, with sufficient control of its derivative, has been addressed in Lemma 5.5 of \cite{BDdPM}. It is important to note that while $\lambda_{1}$ is defined on the interval $(-\varepsilon(T), T)$, the right-hand side of \eqref{ExLambda1} is only well-defined for $t \in (0, T)$. This technical mismatch can be resolved by trivially extending $\varphi^{o}$ to zero for negative times.

\subsection{Review of the inner theories in \cite{BDdPM}}\label{reviewInnerTheories}
The goal of this section is to provide, for completeness, a review of the forced inner theories developed in \cite{BDdPM}. These inner constructions were strongly inspired by the framework introduced in \cite{DdPDMW}. However, due to the slower temporal decay of the coefficient in the \( B \) operator, it became necessary to introduce an additional source term. This source term has already been discussed in the previous Section. In the reminding of this section, we will also compute the three dimensional Laplacian of the solution to the inner equations, providing estimates that will be essential for the gluing scheme.
\newline 
We begin by introducing the time rescaled variable $\tau$, defined by
\begin{align}\label{tauvar}
	\tau(t) = \tau_0 + \int_0^t \frac{ds}{\lambda^2(s)}, \quad t \in (0, T),
\end{align}
where $\tau_0$ is taken to be large. Observe that $\tau \to \infty$ as $t \to T$, so we are now reformulating the inner problem as an infinite-time equation. The corresponding inner equations will be analyzed in the $(y, \tau)$ variables.
Given positive constants $\nu$, $p$, $\epsilon$, and a real number $m \in \mathbb{R}$, we define the norm for a function $h = h(y, \tau)$ as
\begin{align}\label{RHSinnernorm}
	\|h\|_{0; \nu, m, p, \epsilon} = \inf \left\{ K > 0 \;\middle|\; |h(y, \tau)| \le \frac{K}{\tau^{\nu} (\ln \tau)^m} \cdot \frac{1}{(1+|y|)^p}
	\begin{cases}
		1 & \text{if } |y| \le \sqrt{\tau}, \\
		\frac{\tau^{\epsilon/2}}{|y|^{\epsilon}} & \text{if } |y| \ge \sqrt{\tau}
	\end{cases} \right\}.
\end{align}
Similarly, for solutions $\phi = \phi(y, t)$ we define the norm
\begin{align}\label{SolutionINNERNORM}
	\|\phi\|_{1; \nu, m, p, \epsilon} = \inf \left\{ K > 0 \;\middle|\; |\phi(y, t)| + (1+|y|) |\nabla_y \phi(y, t)| \le \frac{K}{\tau^{\nu} (\ln \tau)^m} \cdot \frac{1}{(1+|y|)^p}
	\begin{cases}
		1 & \text{if } |y| \le \sqrt{\tau}, \\
		\frac{\tau^{\epsilon/2}}{|y|^{\epsilon}} & \text{if } |y| \ge \sqrt{\tau}
	\end{cases} \right\}.
\end{align}
At this stage, we make the initial data in equations \eqref{innerEqFINmode0}, \eqref{innerEQfinalMode0NOORT} and \eqref{innerEqfinalMode1} explicit. For that purpose, we introduce the function $\tilde{Z}_0$ defined as
\begin{align}\label{Z0tilde}
	\tilde{Z}_0(|y|) = \big(Z_0(|y|) - m_{Z_0} U\big) \chi_0\left( \frac{|y|}{\sqrt{\tau_0}} \right),
\end{align}
where the constant $m_{Z_0}$ is chosen so that $\int_{\mathbb{R}^2} \tilde{Z}_0 = 0$ and 
\begin{align*}
	Z_{0}(y)=2U+y\cdot \nabla U.
\end{align*}
We also adopt the following notation for the cut-off function:
\begin{align}\label{MhatFunction}
	\chi(r, z, t) = \chi_0\left( \frac{|x - \xi|}{\sqrt{\delta(T - t)}} \right) = \chi_0\left( \frac{|y|\lambda}{M(\tau)} \right).
\end{align}
It is worth noting that in the original inner equations \eqref{innerEqFINmode0}, \eqref{innerEQfinalMode0NOORT}, and \eqref{innerEqfinalMode1}, the operator $B$ is cut off using $\hat{\chi}$, which acts over a slightly larger region. However, this distinction has no impact on the analysis, as is made clear in Sections 7–11 of \cite{BDdPM}. \newline
\underline{The role of $\delta$}: Although a detailed and rigorous explanation is provided in \cite{BDdPM}, we offer here a formal overview of the role played by the parameter $\delta$. The inner theories developed in this section are inspired by those introduced in \cite{DdPDMW}. When considering the infinite-time problem, the corresponding inner equation, after performing the time rescaling, is nearly identical to the finite-time case, with one key distinction involving the operator $B$. Due to the different asymptotic behaviors of the parameter $\lambda(t)$, we obtain the following:
\begin{align*}
	\lambda \dot{\lambda}(2\phi + y \cdot \nabla \phi) \approx 
	\begin{cases}
		-\dfrac{1}{2\tau \ln \tau}(2\phi + y \cdot \nabla \phi), & \text{in the infinite-time case},\\
		-\dfrac{\ln \tau}{\tau}(2\phi + y \cdot \nabla \phi), & \text{in the finite-time case}.
	\end{cases}
\end{align*}
As demonstrated in the proof of Lemma 9.9 of \cite{BDdPM}, in order to derive suitable a priori estimates for the solution $\phi$, it is necessary that
\[
|B[\phi]| \ll \frac{1}{|y|^{2}} \left(|\phi| + |y \cdot \nabla \phi|\right).
\]
This is precisely achieved by cutting off the operator 
$B$, which involves introducing a small parameter 
$\delta$ and, as a result, leads to the appearance of additional source terms.
\newline
In what follows, we will also assume the following estimate on $\lambda(t)$:
\begin{align}\label{AssLambda}
	|\lambda(t) - \lambda_0(t)| \lesssim e^{-\sqrt{2|\ln(T - t)|}}, \quad 
	|\lambda(t) \dot{\lambda}(t) - \lambda_0(t) \dot{\lambda}_0(t)| \lesssim e^{-\frac{3}{2} \sqrt{2|\ln(T - t)|}}.
\end{align}
This assumption will be justified in Section~\ref{proofThm1Sec}.
\begin{proposition}\label{PropMode0Mass0ANTICIP}
	Let $\lambda = \lambda_{0} + \lambda_{1}$, where $\lambda_{0}$ is given by Proposition~\ref{Proplambda0alpha0}, and suppose that assumption~\eqref{AssLambda} holds. Let $\sigma \in (0,1)$, $\epsilon > 0$ with $\sigma + \epsilon < 2$, and let $1 < \nu < \frac{7}{4}$, $m \in \mathbb{R}$, and $q \in (0,1)$. 
	Then, for $\tau_0$ sufficiently large and for every radially symmetric function $h = h(|y|,\tau)$ satisfying
	\[
	\|h\|_{0;\nu,m,6+\sigma,\epsilon} < \infty, \quad \text{and} \quad \int_{\mathbb{R}^{2}} h(y,\tau)\,dy = 0 \quad \text{for all } \tau \in (\tau_{0},\infty),
	\]
	there exists a constant $c_{1} \in \mathbb{R}$ and a solution $\phi = \phi(y,\tau) = \mathcal{T}_{p}^{i,2}[h]$ to the problem
	\[
	\begin{cases}
		\partial_{\tau} \phi = L[\phi] + B[\phi \hat{\chi}] + h, \\
		\phi(\cdot,\tau_0) = c_1 \tilde{Z}_0,
	\end{cases}
	\]
	which defines a linear operator in $h$ and satisfies the following estimates:
	\[
	\|\phi\|_{1;\nu-1,m+q+1,4,2+\sigma+\epsilon} \le \frac{C}{(\ln \tau_{0})^{1-q}} \|h\|_{0;\nu,m,6+\sigma,\epsilon},
	\]
	\[
	|c_{1}| \le C \frac{1}{\tau_{0}^{\nu-1} (\ln \tau_{0})^{m+2}} \|h\|_{0;\nu,m,6+\sigma,\epsilon}.
	\]
\end{proposition}
\begin{proof}
	See the proof of Proposition 5.1 in \cite{BDdPM}
\end{proof}
\begin{proposition}\label{PropMode0Mass0}
	Let $\lambda = \lambda_0 + \lambda_1$, where $\lambda_0$ is defined in Proposition~\ref{Proplambda0alpha0}, and assume that condition \eqref{AssLambda} holds. Let $\sigma \in (0,1)$ and $\epsilon > 0$ be such that $\sigma + \epsilon < 2$, and choose $\nu \in (1, \min(1 + \frac{\varepsilon}{2}, 3 - \frac{\sigma}{2}, \frac{3}{2}))$, $m \in \mathbb{R}$, and $q \in (0,1)$. 
	There exists a number $C>0$ such that for any sufficiently large $\tau_0$ and any radially symmetric function $h = h(|y|, \tau)$ satisfying
	\[
	\|h\|_{0;\nu, m, 6+\sigma, \epsilon} < \infty \quad \text{and} \quad \int_{\mathbb{R}^2} h(y, \tau)\, dy = \int_{\mathbb{R}^{2}} h(y, \tau)\, |y|^{2} dy = 0 \quad \text{for all } \tau > \tau_0,
	\]
	there exist a constant $c_1 \in \mathbb{R}$, a linear operator $\mathcal{E}$, and a function $\phi = \phi(y, \tau) = \mathcal{T}_p^{i,2}[h]$ solving the problem
	\[
	\begin{cases}
		\partial_{\tau} \phi = L[\phi] + B[\phi] + \mathcal{E}[\phi] + h, \\
		\phi(\cdot, \tau_0) = c_1 L[\tilde{Z}_0],
	\end{cases}
	\]
	which defines a linear operator. Moreover, the following estimates hold:
	\[
	\|\phi\|_{1; \nu - \frac{1}{2}, m + \frac{q+1}{2}, 4, 2} \le C \|h\|_{0; \nu, m, 6 + \sigma, \epsilon},
	\]
	\[
	|c_1| \le C \frac{1}{\tau_0^{\nu - 1} (\ln \tau_0)^{m + 2}} \|h\|_{0;\nu, m, 6 + \sigma, \epsilon}.
	\]
	In addition, the operator $\mathcal{E}[\phi]$ is radial, satisfies the orthogonality conditions
	\[
	\int_{\mathbb{R}^{2}} \mathcal{E}[\phi](y)\, dy = \int_{\mathbb{R}^{2}} \mathcal{E}[\phi](y)\, |y|^{2} dy = 0,
	\]
	and obeys the pointwise bound
	\begin{align*}
		|\mathcal{E}[\phi](y)| \le C \|h\|_{0; \nu, m, 6 + \sigma, \epsilon} \frac{1}{\tau^{\nu} \ln^{m+q} \tau}
		\begin{cases}
			\frac{1}{M^{2}(\tau)} \frac{1}{(1 + \rho)^{6}}, & |y| \le M(\tau), \\
			\frac{1}{1 + |y|^{6}}, & |y| \ge M(\tau).
		\end{cases}
	\end{align*}
\end{proposition}
\begin{proof}
	See the proof of Proposition 5.2. in \cite{BDdPM}.
\end{proof}

\begin{proposition}\label{PropMode1ANTICIP}
	Let $\lambda = \lambda_0 + \lambda_1$, where $\lambda_0$ is as defined in Proposition~\ref{Proplambda0alpha0}, and assume that condition \eqref{AssLambda} is satisfied. Let $0 < \varsigma < 1$, $0 < \epsilon < 2$, and choose $\nu \in \left(1, \min\left(1 + \frac{\varepsilon}{2}, \frac{3}{2} - \frac{\varsigma}{2} \right)\right)$ with $m \in \mathbb{R}$. Then, there exists a constant $C > 0$ such that, for sufficiently large $\tau_0$ and any function $h = h(y, \tau)$ satisfying
	\[
	[h]_{\mathrm{rad}} = 0, \quad \|h\|_{0;\nu, m, 5+\varsigma, \epsilon} < \infty, \quad \text{and} \quad \int_{\mathbb{R}^2} h(y, \tau) y_j \, dy = 0 \quad \text{for } j = 1, 2 \text{ and all } \tau \in (\tau_0, \infty),
	\]
	there exists an operator $\mathcal{F}[\phi]$ and a solution $\phi(y, \tau) = \mathcal{T}_{p}^{i,3}[h]$ to the problem
	\[
	\begin{cases}
		\partial_{\tau} \phi = L[\phi] + B[\phi \chi] + \mathcal{F}[\phi] + h, \\
		\phi(\cdot, \tau_0) = 0,
	\end{cases}
	\]
	which defines a linear operator. Furthermore, the following estimate holds:
	\[
	\|\phi\|_{1; \nu, m, 3+\varsigma, 2+\epsilon} \leq C \|h\|_{ 0;\nu, m, 5+\varsigma, \epsilon}.
	\]
	In addition, the operator $\mathcal{F}[\phi]$ has no radial part and satisfies the bound
	\begin{equation}\label{EstimateFFORGLUING}
		|\mathcal{F}[\phi]| \leq \frac{1}{\tau^{\nu+1} (\ln \tau)^{m-1}} \cdot \frac{1}{M^{\varsigma}(\tau)} \left( |W_1(y)| + |W_2(y)| \right).
	\end{equation}
\end{proposition}
\begin{proof}
	See the proof of Proposition 5.3 in \cite{BDdPM}.
\end{proof}
In the next result, we quantify the error incurred by using the two dimensional inverse Laplacian of the inner solution in place of the full three dimensional inverse Laplacian. This estimate will be essential for implementing the fixed point argument in Section~\ref{proofThm1Sec}.
\begin{proposition}\label{ExpansionGradInner}
	Let $\psi$ denote the Newtonian potential solving the equation
	\begin{align*}
		-\Delta_{x} \psi = \frac{1}{\lambda^{2}} \, \phi\left( \frac{|(r,z)-\xi|}{\lambda} \right) \chi_{0}\left( \frac{|(r,z)-\xi|}{\sqrt{\delta(T-t)}} \right), \quad x = (r e^{i\theta}, z),
	\end{align*}
	where the function $\phi$ satisfies the decay condition
	\[
	|\phi(y)| \leq \frac{1}{1 + |y|^{k}}, \quad \text{for some } 3 < k \leq 4.
	\]
	Let $\hat{\psi}$ be the two dimensional inverse Laplacian defined in~\eqref{2dInvLapInn}, and let $\rho = |y|$. Then, for the same function $\omega_{1}(r,z)$ appearing in Theorem~\ref{Expansionv0}, the gradient of $\psi$ obeys the following pointwise bounds:
	\begin{align*}
		\nabla_{(r,z)}\psi =
		\begin{cases}
			\displaystyle \nabla_{(r,z)}\hat{\psi} + O(1), & \text{if } |(r,z) - \xi| \leq \delta \sqrt{T - t}, \\[1ex]
			\displaystyle \omega_{1}(r,z)\nabla_{(r,z)}\hat{\psi} + O(1), & \text{if } \delta \sqrt{T - t} \leq |(r,z) - \xi| \leq \varepsilon, \\[1ex]
			\displaystyle O\left( \frac{1}{|(r,z) - \xi|} \right), & \text{if } |(r,z) - \xi| \geq \varepsilon.
		\end{cases}
	\end{align*}
\end{proposition}

\begin{proof}
	The proof follows directly by suitably adapting the arguments used in the proof of Theorem~\ref{Expansionv0}, Proposition~\ref{3dGRADIENTExpansion}, and Proposition~\ref{3dGRADIENTExpansionpsilambda}. We omit the details.
\end{proof}
\section{The outer theory}\label{outerTheorySec}
To construct a solution to the inner-outer gluing system, it is essential to develop an \emph{outer theory}. Specifically, we seek sufficiently precise estimates for the solution of the following problem:
\begin{align}\label{phiouteqtOUTTH}
	\begin{cases}
		\partial_{t} \phi^{o} = \Delta_{(r,z)} \phi^{o} + \frac{1}{r} \partial_{r} \phi^{o} - \nabla_{(r,z)} v_{0} \cdot \nabla_{(r,z)} \phi^{o} + g(r,z,t), & \text{in } (0,T) \times \mathbb{R}^{+} \times \mathbb{R}, \\
		\phi^{o}(\cdot,0) = 0, & \text{in } \mathbb{R}^{+} \times \mathbb{R},
	\end{cases}
\end{align}
where $v_{0}(r,z) = (-\Delta_{\mathbb{R}^{3}}) u_{0}(r,z)$. We emphasize that Proposition~\ref{3dGRADIENTExpansion} will play a crucial role in the proof of the outer theory. \newline
For a function $g : \mathbb{R}^{+} \times \mathbb{R} \times (0,T) \to \mathbb{R}$, we define the norm $\|g\|_{\star,o}$ as the smallest constant $K$ such that the following estimate holds for all $(r,z,t) \in \mathbb{R}^{+} \times \mathbb{R} \times (0,T)$:
\begin{align*}
	|g(r,z,t)| \le K \begin{cases}
		\frac{e^{-a \sqrt{2|\ln(|(r,z) - \xi|^{2} + (T-t))|}}}{\left(|(r,z) - \xi|^{2} + (T-t)\right)^{2} |\ln((T-t) + |(r,z) - \xi|^{2})|^{b}}, & \text{if } |(r,z) - \xi| \le \sqrt{T}, \\
		\frac{e^{-a \sqrt{2|\ln T|}}}{T^{2} |\ln T|^{b}}, & \text{if } \sqrt{T} \le |(r,z) - \xi| \le M, \\
		\frac{e^{-a \sqrt{2|\ln T|}}}{T^{2} |\ln T|^{b}} e^{- \frac{|(r,z)|^{2}}{4(t+1)}}, & \text{if } |(r,z) - \xi| \ge M.
	\end{cases}
\end{align*}
Here, the constants $a > 0$ and $b \in \mathbb{R}$ will be fixed later in Section~\ref{proofThm1Sec}.
We also introduce a norm for the solution itself. Specifically, we define the norm $\|\phi\|_{\star\star,o}$ as the smallest constant $K$ such that the following estimate holds:
\begin{align}\label{OuterSolExpectedEstimates}
	|\phi(r,z,t)| + (\lambda + |(r,z) - \xi|) |\nabla_{(r,z)} \phi| \le K \begin{cases}
		\frac{e^{-a\sqrt{2|\ln(|(r,z) - \xi|^{2} + (T - t))|}}}{\left(|(r,z) - \xi|^{2} + (T - t)\right) |\ln((T - t) + |(r,z) - \xi|^{2})|^{b}}, & \text{if } |(r,z) - \xi| \le \sqrt{T}, \\
		\frac{e^{-a\sqrt{2|\ln T|}}}{T |\ln T|^{b}}, & \text{if } \sqrt{T} \le |(r,z) - \xi| \le M, \\
		\frac{e^{-a\sqrt{2|\ln T|}}}{T |\ln T|^{b}} e^{- \frac{|(r,z)|^{2}}{4(t+1)}}, & \text{if } |(r,z) - \xi| \ge M.
	\end{cases}
\end{align}
We remark that the approach used in \cite{BDdPM} differs from ours, as the approximation $\nabla_{x} v_{0} \approx \frac{1}{\lambda} \nabla_{y} \Gamma_{0}(y)$ was valid throughout the entire space in their setting. In contrast, here the approximation only holds locally, which leads to a fundamentally different behavior away from the singularity.
\begin{theorem}\label{outerSolvinTheory}
	Assume $a > 0$, $b \in \mathbb{R}$, and let $\lambda$, $\xi$ satisfy \eqref{EstiLambda} and \eqref{EstiPar}. Then, there exists a constant $C$ such that for $T$ sufficiently small, and for any $g$ with $\|g\|_{o} < \infty$, there exists a solution $\phi^{o} = \mathcal{T}_{p}^{o}[g]$ to problem \eqref{phiouteqtOUTTH}. This solution defines a linear operator in $g$ and satisfies the estimate
	\begin{align*}
		\|\phi^{o}\|_{\star\star,o} \le C \|g\|_{\star,o}.
	\end{align*}
\end{theorem}

\begin{proof}
We proceed to construct a barrier. Before doing so, we introduce three technical simplifications that will streamline our analysis:
\begin{itemize}
	\item Throughout the following, we assume $\xi(t) = \xi=(q_{1},q_{2}) \in (0,\infty)\times\mathbb{R}$. The time derivatives of $q$ exhibit significantly faster temporal decay and can therefore be neglected;
	\item We omit the logarithmic terms on the right-hand side of \eqref{phiouteqtOUTTH}, since they represent lower-order contributions that can be incorporated into the barrier;
	\item We adopt the following approximation:
	\begin{align}\label{SimplifiedExpansionGradient}
		\nabla_{(r,z)} v_{0}(r,z,t) = 
		\begin{cases}
			\nabla_{(r,z)} \Gamma_0(r,z)  & \text{if } |(r,z) - \xi| \leq 2\sqrt{\delta (T - t)}, \\[5pt]
			\omega_1(r,z,t)\, \nabla_{(r,z)} \Gamma_0(r,z) +O(|\ln(|(r,z)-\xi)|) & \text{if } 2\sqrt{\delta (T - t)} \leq |(r,z) - \xi| \leq \varepsilon, \\[5pt]
			O(1) & \text{if } \varepsilon \leq |(r,z) - \xi| \leq M, \\[5pt]
			-4\pi q_1  \frac{(r,z)}{|(r,z)|^3} & \text{if } |(r,z) - \xi| \geq M.
		\end{cases}
	\end{align}
	The contributions we are disregarding are explicitly detailed in Proposition~\ref{3dGRADIENTExpansion} and can be absorbed into parts of the barrier.
\end{itemize}
	The barrier we construct will be made by \emph{four} explicit functions whose properties will be described in what follows.\newline
\underline{The first term:}  
We consider the leading term in the barrier constructed in Proposition 12.1 of \cite{BDdPM}, given by
\begin{align*}
	\phi_{0} = \frac{e^{-a \sqrt{2|\ln\left(|(r,z)-\xi|^{2} + (T - t)\right)|}}}{(T - t) + |(r,z) - \xi|^{2}} \, \chi_{0}\left( \frac{|(r,z) - \xi|}{\sqrt{T}} \right).
\end{align*}
As anticipated, the behavior of the system remains predominantly two dimensional in nature. We note the following identity:
\begin{align*}
	-\nabla_{y} \Gamma_{0} &= 4\lambda \frac{(r, z) - \xi}{|(r, z) - \xi|^{2}} + 4 \frac{y}{1 + |y|^{2}} - 4 \frac{y}{|y|^{2}} \\
	&= 4\lambda \frac{(r, z) - \xi}{|(r, z) - \xi|^{2}} - 4 \frac{y}{|y|^{2}}  \frac{1}{1 + |y|^{2}}.
\end{align*}
As a result, the final term can be neglected, as its sign renders its contribution negligible. Recall that the six dimensional radial Laplacian is given by
\[
\Delta_{6} = \partial_{\zeta}^{2} + \frac{5}{\zeta} \partial_{\zeta}, \quad \text{with } \zeta = |(r, z) - \xi|.
\]
Then, the following estimate holds:
\begin{align*}
	\left(\partial_{t} - \Delta_{6}\right)\left( \frac{e^{-a\sqrt{2|\ln(|(r, z) - \xi|^{2} + (T - t)|}}}{(T - t) + |(r, z) - \xi|^{2}} \right)
	\gtrsim  \frac{e^{-a\sqrt{2|\ln((T - t) + |(r, z) - \xi|^{2})|}}}{\left((T - t) + |(r, z) - \xi|^{2}\right)^{2}},
\end{align*}
for all points satisfying \( |(r, z) - \xi| \leq \sqrt{T} \).
In the same region, we observe the following estimates:
\begin{align*}
	\left| \frac{1}{r} \partial_{r} \left( \frac{e^{-a\sqrt{2|\ln((T - t) + |(r, z) - \xi|^{2})|}}}{(T - t) + |(r, z) - \xi|^{2}} \right) \right|
	\lesssim \sqrt{T}  \frac{e^{-a\sqrt{2|\ln((T - t) + |(r, z) - \xi|^{2})|}}}{\left( \lambda^{2} + |(r, z) - \xi|^{2} \right)^{2}},
\end{align*}
and
\begin{align*}
	\left| \left( \nabla_{(r, z)} v_{0} - \nabla_{(r, z)} \Gamma_{0} \right) \cdot \nabla_{(r, z)} \left( \frac{e^{-a\sqrt{2|\ln(|(r, z) - \xi|^{2} + (T - t)|}}}{(T - t) + |(r, z) - \xi|^{2}} \right) \right|
	\lesssim \sqrt{T} |\ln T|  \frac{e^{-a\sqrt{2|\ln((T - t) + |(r, z) - \xi|^{2})|}}}{\left( \lambda^{2} + |(r, z) - \xi|^{2} \right)^{2}},
\end{align*}
where the logarithmic factor arises from the intermediate expansion in~\eqref{SimplifiedExpansionGradient}.
Therefore, assuming \( T \ll 1 \) and \( |(r, z) - \xi| \le \sqrt{T} \), we conclude that
\begin{align*}
	\left( \partial_{t} - \Delta_{(r, z)} - \frac{1}{r} \partial_{r} + \nabla_{(r, z)} v_{0} \cdot \nabla_{(r, z)} \right) 
	\left( \frac{e^{-a\sqrt{2|\ln(|(r, z) - \xi|^{2} + (T - t)|}}}{(T - t) + |(r, z) - \xi|^{2}} \right)
	\gtrsim \frac{e^{-a\sqrt{2|\ln((T - t) + |(r, z) - \xi|^{2})|}}}{\left((T - t) + |(r, z) - \xi|^{2}\right)^{2}}.
\end{align*}
As a consequence, we obtain the following estimate:
\begin{align*}
	\left( \partial_{t} - \Delta_{(r,z)} - \frac{1}{r} \partial_{r} + \nabla_{(r,z)} v_{0} \cdot \nabla_{(r,z)} \right) \phi_{0}
	\gtrsim \frac{e^{-a\sqrt{2|\ln((T - t) + |(r,z) - \xi|^{2})|}}}{\left((T - t) + |(r,z) - \xi|^{2}\right)^{2}} \cdot \chi_{0} \left( \frac{|(r,z) - \xi|}{\sqrt{T}} \right) 
	+ \mathcal{E}_{0}(r,z,t),
\end{align*}
where the error term \(\mathcal{E}_{0}(r,z,t)\) is given by
\begin{align}\label{Err0Out}
	\mathcal{E}_{0}(r,z,t) \lesssim \frac{e^{-a\sqrt{|\ln T|}}}{T^{2}} \left( \left| \chi_{0}'\left( \frac{|(r,z) - \xi|}{\sqrt{T}} \right) \right| + \left| \chi_{0}''\left( \frac{|(r,z) - \xi|}{\sqrt{T}} \right) \right|  \right).
\end{align}
\underline{The second term:} From this point onward, the construction of the barrier deviates significantly from that in \cite{BDdPM}. In particular, in this intermediate region, we consider a barrier that is essentially constant:
\begin{align*}
	\phi_{1} = (t + T)  \frac{e^{-a\sqrt{2|\ln T|}}}{T^{2}}  \chi_{0} \left( \frac{|(r,z)|}{2(M + |\xi|)} \right).
\end{align*}
We observe that
\begin{align*}
	\left| \left( \Delta_{(r,z)} + \frac{1}{r} \partial_{r} \right) \phi_{1} \right| 
	= |\Delta_{x} \phi_{1}| 
	\lesssim \frac{e^{-a\sqrt{2|\ln T|}}}{T}  \frac{1}{M^{2}} \left( \left| \chi_{0}' \left( \frac{|(r,z)|}{2(M + |\xi|)} \right) \right| + \left| \chi_{0}'' \left( \frac{|(r,z)|}{2(M + |\xi|)} \right) \right| \right).
\end{align*}
Moreover, we have
\begin{align*}
	\left| \nabla_{(r,z)} v_{0} \cdot \nabla_{(r,z)} \phi_{1} \right| 
	\lesssim \frac{e^{-a\sqrt{2|\ln T|}}}{T}  \frac{1}{M^{3}} \left| \chi_{0}' \left( \frac{|(r,z)|}{2(M + |\xi|)} \right) \right|.
\end{align*}
We have just shown that
\begin{align*}
	\left( \partial_{t} - \Delta_{(r,z)} - \frac{1}{r} \partial_{r} + \nabla_{(r,z)} v_{0} \cdot \nabla_{(r,z)} \right) \phi_{1} 
	= \frac{e^{-a\sqrt{2|\ln T|}}}{T^{2}} \chi_{0} \left( \frac{|(r,z)|}{2(M + |\xi|)} \right) + \mathcal{E}_{1}(r, z, t),
\end{align*}
where the error term satisfies the bound
\begin{align} \label{Err1Out}
	\mathcal{E}_{1}(r, z, t) \lesssim \frac{e^{-a\sqrt{2|\ln T|}}}{T}  \frac{1}{M^{2}} \left| \chi_{0}' \left( \frac{|(r,z)|}{2(M + |\xi|)} \right) \right|.
\end{align}
 \underline{The third term:} We introduce a term with exponential decay, designed to absorb the error introduced previously:
\begin{align*}
	\phi_{2} = \frac{e^{-a\sqrt{2|\ln T|}}}{T} \frac{1}{M^{2}} e^{-\frac{|(r,z)|^{2}}{4(t + 1)}}.
\end{align*}
Let us consider $
		\hat{\phi}_{2}=e^{-\frac{|(r,z)|^{2}}{4(t+1)}}$, we see that
			\begin{align*}
			(\partial_{t}-\Delta_{(r,z)}-\frac{1}{r}\partial_{r})\hat{\phi}_{2}\gtrsim \frac{1}{(t+1)}e^{-\frac{|(r,z)|^{2}}{4(t+1)}}.
		\end{align*}
Now we observe that
\begin{align*}
	\nabla_{(r,z)} v_{0} \cdot \nabla_{(r,z)} \hat{\phi}_{2} = &\, 4\pi q_{1} \frac{1}{2(t+1)} e^{-\frac{|(r,z)|^{2}}{4(t+1)}}  \frac{1}{|(r,z)|} \left(1 - \chi_{0}\left( \frac{|(r,z)-\xi|}{M} \right)\right) \\
	& + O\left( e^{-\frac{|(r,z)|^{2}}{4(t+1)}} \frac{|(r,z)|}{t+1} \right) \chi_{0} \left( \frac{|(r,z)-\xi|}{M} \right) \\
	& + O\left( \frac{1}{\lambda}  \frac{|y|}{1 + |y|^{2}} \chi_{0} \left( \frac{|(r,z)-\xi|}{\varepsilon} \right) \right).
\end{align*}
We then obtain
\begin{align*}
	\left( \partial_{t} - \Delta_{(r,z)} - \frac{1}{r} \partial_{r} + \nabla_{(r,z)} v_{0} \cdot \nabla_{(r,z)} (\cdot) \right) \phi_{2}
	\geq C  \frac{e^{-a\sqrt{2|\ln T|}}}{T}  \frac{1}{M^{2}} e^{-\frac{|(r,z)|^{2}}{4(t+1)}} + \mathcal{E}_{2}^{(1)}(r,z,t) + \mathcal{E}_{2}^{(2)}(r,z,t).
\end{align*}
We have the following estimates for the error terms:
\begin{align}
	|\mathcal{E}_{2}^{(1)}(r,z,t)| &\lesssim \frac{1}{\lambda}  \frac{|y|}{1 + |y|^{2}}  \frac{e^{-a\sqrt{2|\ln T|}}}{M^{2} T} \, \chi_{0} \left( \frac{|(r,z) - \xi|}{\varepsilon} \right), \label{Error3Out1} \\
	|\mathcal{E}_{2}^{(2)}(r,z,t)| &\lesssim \frac{e^{-a\sqrt{2|\ln T|}}}{M^{2} T} \, \chi_{0} \left( \frac{|(r,z)|}{M + |\xi|} \right). \label{Error3Out2}
\end{align}
Let us investigate the behaviour of \( \mathcal{E}_{2}^{(1)}(r,z,t) \). If \( \sqrt{T} \leq |(r,z) - \xi| \leq \varepsilon \)
\begin{align*}
	\frac{1}{\lambda}  \frac{|y|}{1 + |y|^{2}}  \frac{e^{-a\sqrt{2|\ln T|}}}{M^{2} T}
	&\lesssim \frac{1}{M^{2}}  \frac{\sqrt{T}}{\sqrt{\lambda^{2} + |(r,z) - \xi|^{2}}}  \frac{e^{-a\sqrt{2|\ln T|}}}{T^{3/2}} \\
	&\lesssim \frac{\sqrt{T}}{M^{2}}  \frac{e^{-a\sqrt{2|\ln T|}}}{ T^{2}}.
\end{align*}
If \( \sqrt{\delta (T - t)} \leq |(r,z) - \xi| \leq \sqrt{T} \), we observe that
\begin{align*}
	\frac{1}{\lambda}  \frac{|y|}{1 + |y|^{2}}  \frac{e^{-a\sqrt{2|\ln T|}}}{M^{2} T}
	&\lesssim \frac{1}{M^{2}}  \frac{\sqrt{T}}{\sqrt{\lambda^{2} + |(r,z) - \xi|^{2}}}  \frac{e^{-a\sqrt{2|\ln T|}}}{T^{3/2}} \\
	&\lesssim \frac{\sqrt{T}}{M^{2}}  \frac{e^{-a\sqrt{2|\ln((T - t) + |(r,z) - \xi|^{2})|}}}{\left((T - t) + |(r,z) - \xi|^{2}\right)^{2}}.
\end{align*}
When \( |(r,z) - \xi| \leq \sqrt{\delta (T - t)} \) istead we have
\begin{align*}
	\frac{1}{\lambda} \frac{|y|}{1 + |y|^{2}} \ \frac{e^{-a\sqrt{2|\ln T|}}}{M^{2} T}
	&\lesssim \frac{1}{\sqrt{\lambda^{2} + |(r,z) - \xi|^{2}}}  \frac{e^{-a\sqrt{2|\ln T|}}}{M^{2} T} \\
	&\lesssim \frac{\sqrt{T - t}}{M^{2}}  \frac{e^{-a\sqrt{2|\ln(T - t)|}}}{\left( \lambda^{2} + |(r,z) - \xi|^{2} \right)^{2}}.
\end{align*}
It turns out to be useful, as a consequence of this analysis, to decompose the error \eqref{Error3Out1} into $\mathcal{E}_{2}^{(1)}(r,z,t)=\mathcal{E}_{2}^{(1,0)}(r,z,t)+\mathcal{E}_{2}^{(1,1)}(r,z,t)+\mathcal{E}_{2}^{(1,2)}(r,z,t)$ with
\begin{align}\label{OutError2Deco0}
	\mathcal{E}_{2}^{(1,0)}(r,z,t)\lesssim\frac{\sqrt{T}}{M^{2}}  \frac{e^{-a\sqrt{2|\ln T|}}}{ T^{2}}
\end{align}
\begin{align}\label{OutError2Deco}
	\mathcal{E}_{2}^{(1,1)}(r,z,t)\lesssim\frac{\sqrt{T}}{M^{2}}  \frac{e^{-a\sqrt{2|\ln((T - t) + |(r,z) - \xi|^{2})|}}}{\left((T - t) + |(r,z) - \xi|^{2}\right)^{2}}, \ \ \ \mathcal{E}_{2}^{(1,2)}(r,z,t)\lesssim \frac{\sqrt{T - t}}{M^{2}}  \frac{e^{-a\sqrt{2|\ln(T - t)|}}}{\left( \lambda^{2} + |(r,z) - \xi|^{2} \right)^{2}}.
\end{align}
This last term is more challenging to absorb due to its singular nature. Therefore, we introduce the final term in the barrier construction to control it.\newline
\underline{The fourth term:} We consider now 
\begin{align*}
	\phi_{3}=\frac{\sqrt{T-t}}{M^{2}}\frac{e^{-a\sqrt{2|\ln(T-t)|}}}{\lambda^{2}+|(r,z)-\xi|^{2}}\chi_{0}(\frac{|(r,z)-\xi|}{\sqrt{\delta(T-t)}})
\end{align*}
We introduce the notation $\hat{\phi}_{3} = \frac{1}{\lambda^{2}+|(r,z)-\xi|^{2}}$. With this, we compute
\begin{align*}
	\frac{1}{\lambda}\nabla_{y}\Gamma_{0} \cdot \nabla_{(r,z)} \hat{\phi}_{3} &= 8 \frac{|(r,z)-\xi|^{2}}{(\lambda^{2}+|(r,z)-\xi|^{2})^{3}} \ge 0.
\end{align*}
In particular, when $|(r,z)-\xi|\le 2\sqrt{\delta(T-t)}$, it follows that
\begin{align}
	\left(-\Delta_{(r,z)} - \frac{1}{r} \partial_{r} + \nabla v_{0} \cdot \nabla \right) \hat{\phi}_{3} \gtrsim \frac{1}{(\lambda^{2}+|(r,z)-\xi|^{2})^{2}}.
\end{align}
Moreover, we observe the following bound:
\begin{align*}
	\partial_{t} \phi_{3} \lesssim \delta \frac{\sqrt{T-t}}{M^{2}} \frac{e^{-a\sqrt{2|\ln(T-t)|}}}{(\lambda^{2}+|(r,z)-\xi|^{2})^{2}}.
\end{align*}
As a result, we deduce that
\begin{align*}
	\left(\partial_{t} - \Delta_{(r,z)} - \frac{1}{r} \partial_{r} + \nabla_{(r,z)} v_{0} \cdot \nabla_{(r,z)}\right) \phi_{3} \gtrsim \frac{\sqrt{T-t}}{M^{2}} \frac{e^{-a\sqrt{2|\ln(T-t)|}}}{(\lambda^{2}+|(r,z)-\xi|^{2})^{2}} \chi_{0}\left(\frac{|(r,z)-\xi|}{\sqrt{\delta(T-t)}}\right) + \mathcal{E}_{3}(r,z),
\end{align*}
where the error term $\mathcal{E}_{3}(r,z)$ satisfies the estimate
\begin{align*}
	|\mathcal{E}_{3}(r,z)| \lesssim \frac{1}{M^{2}} \frac{e^{-a\sqrt{2|\ln(T-t)|}}}{\delta^{4}(T-t)^{3/2}} \left( \left| \chi_{0}'\left(\frac{|(r,z)-\xi|}{\sqrt{\delta(T-t)}}\right) \right| + \left| \chi_{0}''\left(\frac{|(r,z)-\xi|}{\sqrt{\delta(T-t)}}\right) \right| \right).
\end{align*}
We now consider the barrier previously constructed, given by $\Phi = \sum_{i=0}^{3} c_{i} \phi_{i}$, where the constants $c_{i}$ are positive and will be chosen appropriately. \newline
We note the following estimate:
\begin{align*}
	\left(\partial_{t} - \Delta_{(r,z)} - \frac{1}{r} \partial_{r} + \nabla_{(r,z)} v_{0} \cdot \nabla_{(r,z)} \right) \Phi \ge 
	&\; c_{0} \frac{e^{-a\sqrt{2|\ln((T-t)+|(r,z)-\xi|^{2})|}}}{\left((T-t) + |(r,z)-\xi|^{2}\right)^{2}} \chi_{0}\left( \frac{|(r,z)-\xi|}{\sqrt{T}} \right) +  \\
	&+ c_{1} \frac{e^{-a\sqrt{2|\ln T|}}}{T^{2}} \chi_{0}\left( \frac{|(r,z)|}{2(M+|\xi|)} \right) +  \\
	&+ c_{2} \frac{e^{-a\sqrt{2|\ln T|}}}{T}  \frac{1}{M^{2}} e^{-\frac{|(r,z)|^{2}}{4(t+1)}}  +  \\
	&+ c_{3} \frac{\sqrt{T-t}}{M^{2}} \frac{e^{-a\sqrt{2|\ln(T-t)|}}}{\left( \lambda^{2} + |(r,z)-\xi|^{2} \right)^{2}} \chi_{0}\left( \frac{|(r,z)-\xi|}{\delta \sqrt{T-t}} \right) +\\
	& +c_{0}\mathcal{E}_{0}(r,z,t)+c_{1}\mathcal{E}_{1}(r,z,t)+c_{2}\mathcal{E}_{2}^{(1,0)}(r,z,t)+c_{2}\mathcal{E}_{2}^{(1,1)}(r,z,t)+\\
	&+c_{2}\mathcal{E}_{2}^{(1,2)}(r,z,t)+c_{2}\mathcal{E}_{(2)}^{2}(r,z,t)+c_{3}\mathcal{E}_{3}(r,z,t)
\end{align*}
It remains to determine the constants $c_{i}$. The strategy we adopt is as follows:
\begin{itemize}
	\item Choose $c_{3} > \|g\|_{o}$ so that the term $c_{2} \mathcal{E}_{2}^{(1,2)}(r,z)$ is canceled;
	\item Then select $c_{2} > \|g\|_{o}$ to eliminate the term $c_{1} \mathcal{E}_{1}(r,z)$;
	\item Finally, take $c_{1} > \|g\|_{o}$ to cancel $c_{0} \mathcal{E}_{0}(r,z)$.
\end{itemize}
The remaining error terms are $c_{2} \mathcal{E}_{2}^{(2)}(r,z,t)$, $c_{2}\mathcal{E}_{2}^{(1,0)}(r,z,t)$, $c_{2}\mathcal{E}_{2}^{(1,1)}(r,z,t)$ and $c_{3} \mathcal{E}_{3}(r,z,t)$. The first of these can be controlled by the contribution of $\phi_{1}$, provided $T$ is taken sufficiently small. Similarly, the second and the third can be absorbed by the contribution of $\phi_{0}$ and $\phi_{1}$. \newline
As a final observation we want to show that the contribution of $\phi_{3}$ is not effecting the expected main order term that is given by $\phi_{0}$: 
	\begin{align*}
		(T-t)\frac{e^{-a\sqrt{2|\ln(T-t)|}}}{\lambda^{2}+|(r,z)-\xi|^{2}}\lesssim\frac{e^{-a\sqrt{2|\ln(|(r,z)-q)|^{2}+(T-t)|}}}{(T-t)+|(r,z)-\xi|^{2}} \ \ \ \ \text{if }|(r,z)-\xi|\le 2\sqrt{\delta(T-t)}.
	\end{align*}
The proof is concluded by rescaling and applying standard elliptic estimates.	
\end{proof}
In the rest of this section, we derive an estimate that will be used to control certain errors in Section \ref{proofThm1Sec}. 
\begin{proposition}\label{3dgradientouter}
	Assume that \eqref{EstiLambda} and \eqref{EstiLambda} are satisfied. Let $\phi^{o}$ be such that $\|\phi^{o}\|_{\star\star,o}<\infty$, and define $\hat{\psi}^{o} = (-\Delta_{\mathbb{R}^{2}})^{-1}\phi^{o}$. Let \( \psi^{o} \) be the corresponding solution, determined by the Newtonian potential, to the equation
	\begin{align*}
		-\Delta_{(r,z)} \psi^{o} - \frac{1}{r} \, \partial_{r} \psi^{o}
		= \phi^{o}, 
		\quad (r,z) \in \mathbb{R}^{+} \times \mathbb{R}.
	\end{align*}
	Then, for $T$ sufficiently small and $M > 0$ sufficiently large, the gradient of \( \psi^{o} \) admits the following asymptotic expansion:
	\begin{align}\label{ExpansionGradpsiout}
		\nabla_{(r,z)} \psi^{o}(r,z,t) = 
		\begin{cases}
			\nabla_{(r,z)} \hat{\psi}^{o}(r,z) + O(1) & \text{if } |(r,z) - \xi| \leq 2\sqrt{\delta (T - t)}, \\[5pt]
			\omega_1(r,z,t)\, \nabla_{(r,z)} \hat{\psi}^{o}(r,z) + O(|\hat{\psi}^{o}|) & \text{if } 2\sqrt{\delta (T - t)} \leq |(r,z) - \xi| \leq \sqrt{T}, \\[5pt]
			O\left(\frac{e^{-a\sqrt{2|\ln T|}}}{T |\ln T|^{b}}\right) & \text{if } |(r,z) - \xi| \geq \sqrt{T}.
		\end{cases}
	\end{align}
\end{proposition}
\begin{proof}
	We begin by splitting $\phi^{o}$ into two parts:
	\begin{align*}
		\phi^{o}(r,z,t) = \phi^{o}_{(1)}(r,z,t) + \phi^{o}_{(2)}(r,z,t),
	\end{align*}
	where $\phi^{o}_{(1)}(r,z,t) = \phi^{o} \, \chi_{0}\left(\frac{|(r,z) - \xi|}{\sqrt{T}}\right)$ and $\phi^{o}_{(2)}(r,z,t) = \phi^{o} \left(1 - \chi_{0}\left(\frac{|(r,z) - \xi|}{\sqrt{T}}\right)\right)$. The term $\phi^{o}_{(1)}$ can be handled similarly to the arguments used in Theorem~\ref{Expansionv0} and Proposition~\ref{3dGRADIENTExpansion}. The second term can be estimated directly, taking advantage of the exponential decay previously established for $\phi^{o}$.
\end{proof}

\section{Proof of Theorem \ref{Mthm}}\label{proofThm1Sec}
The goal of this section is to complete the proof of Theorem~\ref{Mthm}. As discussed in the preceding sections, it suffices to solve the following set of equations:
\begin{align}\label{innerEqFINmode0THM}
	\begin{cases}
		\lambda^{2} \partial_{t} \phi_{1} = L[\phi_{1}] + B[\phi_{1}] + \mathcal{E}[\phi_{1}] + f - m_{0}[f] W_{0} - m_{2}[f] W_{2}, & \text{in } \mathbb{R}^{2} \times (0,T), \\
		\phi_{1}(\cdot, 0) = \phi_{1,0}, & \text{in } \mathbb{R}^{2},
	\end{cases}
\end{align}
\begin{align}\label{innerEQfinalMode0NOORTTHM}
	\begin{cases}
		\lambda^{2} \partial_{t} \phi_{2} = L[\phi_{2}] + B[\phi_{2} \hat{\chi}] + m_{2}[f] W_{2}, & \text{in } \mathbb{R}^{2} \times (0,T), \\
		\phi_{2}(\cdot, 0) = \phi_{2,0}, & \text{in } \mathbb{R}^{2},
	\end{cases}
\end{align}
\begin{align}\label{innerEqfinalMode1THM}
	\begin{cases}
		\lambda^{2} \partial_{t} \phi_{3} = L[\phi_{3}] + B[\phi_{3} \hat{\chi}] + \mathcal{F}[\phi_{3}] + f_{3} - m_{1,j}[f_{3}] W_{1,j}, & \text{in } \mathbb{R}^{2} \times (0,T), \\
		\phi_{3}(\cdot, 0) = \phi_{3,0}, & \text{in } \mathbb{R}^{2},
	\end{cases}
\end{align}
where the source terms \( f \) and \( f_{3} \) are given in Section~\ref{RefInnerOutGluingSystem}.
Additionally, we need to determine the parameters \( \alpha_{1} \), \( q_{1}(t) \) (for \( t \in (0,T) \)), and \( \lambda_{1}(t) \) (for \( t \in (-\varepsilon(T), T) \)) in such a way that the following orthogonality conditions are fulfilled:
\begin{align}\label{OrthoCondTHM}
	\begin{cases}
		m_{0}[f[\phi_{1} + \phi_{2} + \phi_{3}, \varphi^{o}, \mathbf{p}_{1}]] = 0, \\
		m_{1,j}[f_{3}[\phi_{1} + \phi_{2} + \phi_{3}, \varphi^{o}, \mathbf{p}_{1}]] = 0, & \text{for } j = 1, 2, \quad t \in (0,T).
	\end{cases}
\end{align}
This system is coupled with the following outer problem:
\begin{align}\label{outerEqfinalTHM}
	\begin{cases}
		\partial_{t} \varphi^{o} = \Delta_{(r,z)} \varphi^{o} + \frac{1}{r} \partial_{r} \varphi^{o} - \nabla v_{0} \cdot \nabla \varphi^{o} + g(\phi_{1} + \phi_{2} + \phi_{3}, \varphi^{o}, \mathbf{p}_{1}), & \text{in } \mathbb{R}^{+} \times \mathbb{R} \times (0,T), \\
		\varphi^{o}(\cdot, 0) = 0, & \text{in } \mathbb{R}^{+} \times \mathbb{R},
	\end{cases}
\end{align}
where the source term is given in Section \ref{RefInnerOutGluingSystem}.
For the sake of clarity, we briefly summarize the results obtained in Sections~\ref{reviewInnerTheories} and~\ref{outerTheorySec}. \newline
\underline{Inner solution operators:} Let $\nu$, $p$, and $\epsilon$ be positive constants, and let $m \in \mathbb{R}$. For a function $h = h(y, \tau)$, we introduce the norm
\begin{align}\label{RHSinnernorm}
	\|h\|_{0; \nu, m, p, \epsilon} = \inf \left\{ K > 0 \;\middle|\; |h(y, \tau)| \le \frac{K}{\tau^{\nu} (\ln \tau)^m} \cdot \frac{1}{(1+|y|)^p}
	\begin{cases}
		1 & \text{for } |y| \le \sqrt{\tau}, \\[4pt]
		\frac{\tau^{\epsilon/2}}{|y|^{\epsilon}} & \text{for } |y| \ge \sqrt{\tau}
	\end{cases} \right\}.
\end{align}
Analogously, for solutions $\phi = \phi(y, t)$, we define the corresponding norm as
\begin{align}\label{SolutionINNERNORM}
	\|\phi\|_{1; \nu, m, p, \epsilon} = \inf \left\{ K > 0 \;\middle|\; |\phi(y, t)| + (1+|y|)\, |\nabla_y \phi(y, t)| \le \frac{K}{\tau^{\nu} (\ln \tau)^m} \cdot \frac{1}{(1+|y|)^p}
	\begin{cases}
		1 & \text{for } |y| \le \sqrt{\tau}, \\[4pt]
		\frac{\tau^{\epsilon/2}}{|y|^{\epsilon}} & \text{for } |y| \ge \sqrt{\tau}
	\end{cases} \right\}.
\end{align}
In order to solve equation~\eqref{innerEqFINmode0THM}, we make use of Proposition~\ref{PropMode0Mass0}, taking into consideration the given initial data. The solution is given by $\phi_{1}=\mathcal{T}_{p}^{i,2}[f - m_{0}[f] W_{0} - m_{2}[f] W_{2}]$, for which the following estimate holds:
\[
\|\phi_{1}\|_{1; \nu - \frac{1}{2}, m + \frac{q+1}{2}, 4, 2} \le C \|f - m_{0}[f] W_{0} - m_{2}[f] W_{2}\|_{0; \nu, m, 6 + \sigma, \epsilon}.
\]
Similarly, equation~\eqref{innerEQfinalMode0NOORTTHM} can be addressed using Proposition~\ref{PropMode0Mass0ANTICIP}, again with the initial condition specified. The solution is given by $\phi_{2}=\mathcal{T}_{p}^{i,2}[m_{2}[f] W_{2}]$. In this case, we obtain
\[
\|\phi_{2}\|_{1;\nu-1,m+q+1,4,2+\sigma+\epsilon} \le \frac{C}{(\ln \tau_{0})^{1-q}} \| m_{2}[f] W_{2}\|_{0;\nu,m,6+\sigma,\epsilon},
\]
where $\tau_0= \tau (0)$, see \eqref{tauvar}.
Lastly, equation~\eqref{innerEqfinalMode1THM} is treated by applying Proposition~\ref{PropMode1ANTICIP}, here  $\phi_{3}=\mathcal{T}_{p}^{i,2}[f_{3} - m_{1,j}[f_{3}] W_{1,j}]$ and we get the estimate
\[
\|\phi_{3}\|_{1; \nu, m, 3+\varsigma, 2+\epsilon} \leq C \| f_{3} - m_{1,j}[f_{3}] W_{1,j}\|_{ 0;\nu, m, 5+\varsigma, \epsilon}.
\]
\newline
\underline{Outer solution operator:} Consider a function $g : \mathbb{R}^{+} \times \mathbb{R} \times (0,T) \to \mathbb{R}$. We define the norm $\|g\|_{\star,o}$ as the smallest constant $K$ such that, for all $(r,z,t) \in \mathbb{R}^{+} \times \mathbb{R} \times (0,T)$, the following inequality holds:
\begin{align*}
	|g(r,z,t)| \le K \begin{cases}
		\frac{e^{-a \sqrt{2|\ln(|(r,z) - \xi|^{2} + (T-t))|}}}{\left(|(r,z) - \xi|^{2} + (T-t)\right)^{2} |\ln((T-t) + |(r,z) - \xi|^{2})|^{b}}, & \text{if } |(r,z) - \xi| \le \sqrt{T}, \\[5pt]
		\frac{e^{-a \sqrt{2|\ln T|}}}{T^{2} |\ln T|^{b}}, & \text{if } \sqrt{T} \le |(r,z) - \xi| \le M, \\[5pt]
		\frac{e^{-a \sqrt{2|\ln T|}}}{T^{2} |\ln T|^{b}} e^{- \frac{|(r,z)|^{2}}{4(t+1)}}, & \text{if } |(r,z) - \xi| \ge M.
	\end{cases}
\end{align*}
We also define a corresponding norm for the solution $\phi = \phi(r,z,t)$. Specifically, the norm $\|\phi\|_{\star\star,o}$ is the smallest constant $K$ such that the following estimate holds:
\begin{align}\label{OuterSolExpectedEstimates}
	|\phi(r,z,t)| + (\lambda + |(r,z) - \xi|) |\nabla_{(r,z)} \phi| \le K \begin{cases}
		\frac{e^{-a\sqrt{2|\ln(|(r,z) - \xi|^{2} + (T - t))|}}}{\left(|(r,z) - \xi|^{2} + (T - t)\right) |\ln((T - t) + |(r,z) - \xi|^{2})|^{b}}, & \text{if } |(r,z) - \xi| \le \sqrt{T}, \\[5pt]
		\frac{e^{-a\sqrt{2|\ln T|}}}{T |\ln T|^{b}}, & \text{if } \sqrt{T} \le |(r,z) - \xi| \le M, \\[5pt]
		\frac{e^{-a\sqrt{2|\ln T|}}}{T |\ln T|^{b}} e^{- \frac{|(r,z)|^{2}}{4(t+1)}}, & \text{if } |(r,z) - \xi| \ge M.
	\end{cases}
\end{align}
To solve equation~\eqref{outerEqfinalTHM}, we apply Proposition~\ref{outerSolvinTheory}. As a result, the solution can be obtained by the operator $\varphi^{o}=\mathcal{T}_{p}^{o}[g]$ and we obtain the estimate
\begin{align*}
	\|\varphi^{o}\|_{\star\star,o} \le C \|g\|_{\star,o}.
\end{align*}
The proof proceeds via a fixed point argument. Since the majority of the terms in the right-hand side can be treated in the same way as in Section~5 of \cite{BDdPM}, we will concentrate only on the key differences.
We now define the norms that are instrumental for analyzing and solving the orthogonality conditions given in \eqref{OrthoCondTHM}.
\begin{align}\label{normforalpha1}
	\|\alpha_{1}\|^{(1)}_{C^{1};\nu+\frac{1}{2},m_{1}} = \sup_{t \in [0,T)} \left( e^{(\nu+\frac{1}{2}) \sqrt{2|\ln(T-t)|}} |\ln(T-t)|^{m_{1}} \left( |\alpha_1 (t)| + (T - t) |\dot{\alpha_1}(t)| \right) \right),
\end{align}
\begin{align}\label{normforxi1}
	\|\xi_{1}\|^{(2)}_{C^{1};\varrho,m_{2}} = \sup_{t \in [0,T)} \left( \frac{e^{\varrho \sqrt{2|\ln(T-t)|}} |\ln(T-t)|^{m_{2}}}{\sqrt{T - t}} \left( |\xi_{1}(t)| + (T - t) |\ln(T - t)|^{1/2} |\dot{\xi}_{1}(t)| \right) \right),
\end{align}
\begin{align}\label{normforlambda1}
	\|\lambda_{1}\|_{C^{2};\nu - \frac{1}{2},m_{3}} = \sup_{t \in [0,T)} \frac{e^{(\nu - \frac{1}{2}) \sqrt{2|\ln(T - t)|}} |\ln(T - t)|^{m_{3}}}{\sqrt{T - t}} \left( |\lambda_{1}(t)| + (T - t) |\dot{\lambda}_{1}(t)| + (T - t)^{2} |\ddot{\lambda}_{1}(t)| \right).
\end{align}

\begin{proof}[Proof of Theorem \ref{Mthm}]
We introduce the notation
\begin{align*}
	\vec{\phi} = (\phi_{1}, \phi_{2}, \phi_{3}, \varphi^{o}, \mathbf{p}),
\end{align*}
and define the operator
\begin{align*}
	\mathcal{A}[\vec{\phi}] = (\mathcal{A}_{i,1}[\vec{\phi}], \mathcal{A}_{i,2}[\vec{\phi}], \mathcal{A}_{i,3}[\vec{\phi}], \mathcal{A}_{o}[\vec{\phi}], \mathcal{A}_{p}[\vec{\phi}]),
\end{align*}
where the component $\mathcal{A}_{p}$ is given by
\begin{align*}
	\mathcal{A}_{p}[\vec{\phi}] = (\mathcal{A}_{p,\lambda}[\vec{\phi}], \mathcal{A}_{p,\alpha_{1}}[\vec{\phi}], \mathcal{A}_{p,\xi_{1}}[\vec{\phi}]).
\end{align*}
These operators correspond to the solution mappings for equations~\eqref{innerEqFINmode0THM}, \eqref{innerEQfinalMode0NOORTTHM}, \eqref{innerEqfinalMode1THM}, and~\eqref{outerEqfinalTHM}, along with the orthogonality conditions~\eqref{OrthoCondTHM}.Our goal is to find a fixed point $\vec{\phi}$ such that $\vec{\phi}=\mathcal{A}[\vec{\phi}]$. \newline 
To proceed, we define the functional spaces in which the operator $\mathcal{A}$ will act.
\begin{align*}
	X_{i,1} = \left\{ \phi_{1} \in L^{\infty}(\mathbb{R}^{2} \times (0,T)) \ \middle| \ 
	\begin{aligned}
		& \nabla_{y} \phi_{1} \in L^{\infty}(\mathbb{R}^{2} \times (0,T)), \ \|\phi_{1}\|_{1; \nu - \frac{1}{2}, \frac{q+1}{2}, 4, 2} < \infty, \\
		& \int_{\mathbb{R}^{2}} \phi_{1}(y,t) \, dy = 0, \quad \int_{\mathbb{R}^{2}} \phi_{1}(y,t) \, y \, dy = 0, \quad \int_{\mathbb{R}^{2}} \phi_{1}(y,t) |y|^{2} \, dy = 0, \quad t \ge 0
	\end{aligned}
	\right\},
\end{align*}
\begin{align*}
	X_{i,2} = \left\{ \phi_{2} \in L^{\infty}(\mathbb{R}^{2} \times (0,T)) \ \middle| \ 
	\begin{aligned}
		& \nabla_{y} \phi_{2} \in L^{\infty}(\mathbb{R}^{2} \times (0,T)), \ \|\phi_{2}\|_{1; \nu - \frac{1}{2}, \frac{q+1}{2}, 4, 2 + \sigma + \epsilon} < \infty, \\
		& \int_{\mathbb{R}^{2}} \phi_{2}(y,t) \, dy = 0, \quad \int_{\mathbb{R}^{2}} \phi_{2}(y,t) \, y \, dy = 0, \quad t \ge 0
	\end{aligned}
	\right\},
\end{align*}
\begin{align*}
	X_{i,3} = \left\{ \phi_{3} \in L^{\infty}(\mathbb{R}^{2} \times (0,T)) \ \middle| \ 
	\begin{aligned}
		& \nabla_{y} \phi_{3} \in L^{\infty}(\mathbb{R}^{2} \times (0,T)), \ \|\phi_{3}\|_{1; \nu, 0, 3 + \varsigma, 2 + \epsilon} < \infty, \\
		& \int_{\mathbb{R}^{2}} \phi_{3}(y,t) \, dy = 0, \quad \int_{\mathbb{R}^{2}} \phi_{3}(y,t) \, y \, dy = 0, \quad \int_{\mathbb{R}^{2}} \phi_{3}(y,t) |y|^{2} \, dy = 0, \quad t \ge 0
	\end{aligned}
	\right\},
\end{align*}
\begin{align*}
	X_{o} = \left\{ \varphi^{o} \in L^{\infty}(\mathbb{R}^{2} \times [0,T)) \ \middle| \ \nabla_{y} \varphi^{o} \in L^{\infty}(\mathbb{R}^{2} \times (0,T)), \ \|\varphi^{o}\|_{\star,0} < \infty \right\},
\end{align*}
\begin{align*}
	X_{p} = \left\{ (\lambda_{1}, \alpha_{1}, \xi_{1}) \in C^{1}([0,T)) \ \middle| \ 
	\|\alpha_{1}\|^{(1)}_{C^{1}; \nu + \frac{1}{2}, m_{1}} < \infty, \ 
	\|\xi_{1}\|^{(2)}_{C^{1}; \varrho, m_{2}} < \infty, \ 
	\|\lambda_{1}\|_{C^{2}; \nu + \frac{1}{2}, m_{3}} < \infty \right\}.
\end{align*}
The fixed point strategy we will adopt depends on the appropriate selection of the constants $\nu$, $q$, $\sigma$, $\zeta$, $\epsilon$, $a$, $b$, $\rho$, $m_{1}$, $m_{2}$, $m_{3}$.
\newline
As noted earlier, in light of the analysis developed in Section~5 of \cite{BDdPM}, it is sufficient to address the modifications introduced by the additional terms appearing on the right-hand sides. \newline
The most delicate contributions arise from the nonlocal term and the operator \( \frac{1}{r} \partial_{r} \). However, for the inner problems and the associated orthogonality conditions, the nonlocal terms do not introduce substantial differences, owing to the asymptotic expansions provided in Propositions~\ref{3dGRADIENTExpansionpsilambda}, \ref{ExpansionGradInner}, and~\ref{3dgradientouter}. \newline
We now estimate the following term:
\begin{align*}
	\left|\lambda^{4} \frac{1}{r} \partial_{r}\left(\frac{1}{\lambda^{2}} \phi_{1}\right)\right| \lesssim \lambda(t) \cdot \frac{\|\phi_{1}\|_{1; \nu - \frac{1}{2}, m + \frac{q+1}{2}, 4, 2}}{\tau^{\nu} (\ln \tau)^m} \frac{1}{(1+|y|)^{5}} 
	\begin{cases}
		1 & \text{if } |y| \le \sqrt{\tau}, \\[4pt]
		\frac{\tau}{|y|^{2}} & \text{if } |y| \ge \sqrt{\tau}.
	\end{cases}
\end{align*}
By performing the change of variables as in~\eqref{tauvar}, we observe that \( |\lambda(t)| \lesssim e^{-\ln^{2} \tau} \cdot \frac{\ln \tau}{\tau} \). Therefore, the contribution of this term, as well as those of its mass, and its first and second moments, may be regarded as lower order perturbations. A similar estimate holds for the term  \( \lambda^{4} \frac{1}{r} \partial_r \left( \frac{1}{\lambda^{2}} (\phi_{2} + \phi_{3}) \right) \), whose contribution can likewise be treated as a lower order perturbation. The treatment of the inner equations and the orthogonality conditions follows the same approach as in \cite{DdPMW}. For this reason, we omit the detailed exposition here. \newline
The outer equation requires special attention due to the appearance of additional terms. As an illustrative example, consider the quantity
\[
\nabla_{(r,z)}\phi^{o} \cdot \nabla_{(r,z)}\psi_{\lambda}(1 - \chi).
\]
Using Proposition~\ref{3dGRADIENTExpansionpsilambda}, we estimate this term in the following spatial regimes:

\medskip
\noindent
\textbf{Case 1:} If $\sqrt{\delta(T - t)} \le |(r,z) - \xi| \le \sqrt{T}$, then
\begin{align*}
	|\nabla_{(r,z)}\phi^{o} \cdot \nabla_{(r,z)}\psi_{\lambda}(1 - \chi)|
	&\lesssim \|\phi^{o}\|_{\star,o}  \frac{e^{-\sqrt{2|\ln(T - t)|}}}{|(r,z) - \xi|} \sqrt{|\ln(T - t)|} \\
	&\quad \cdot \frac{e^{-a\sqrt{2|\ln(|(r,z) - \xi|^{2} + (T - t))|}}}
	{\left(|(r,z) - \xi|^{2} + (T - t)\right) |\ln((T - t) + |(r,z) - \xi|^{2})|^{b}} \\
	&\lesssim e^{-\sqrt{2|\ln T|}} \sqrt{|\ln T|} 
	\frac{e^{-a\sqrt{2|\ln(|(r,z) - \xi|^{2} + (T - t))|}}}
	{(|(r,z) - \xi|^{2} + (T - t))^{2}} \|\phi^{o}\|_{\star,o}.
\end{align*}

\medskip
\noindent
\textbf{Case 2:} If $\sqrt{T} \le |(r,z) - \xi| \le M$, then
\begin{align*}
	|\nabla_{(r,z)}\phi^{o} \cdot \nabla_{(r,z)}\psi_{\lambda}(1 - \chi)|
	&\lesssim \|\phi^{o}\|_{\star,o} \frac{e^{-a\sqrt{2|\ln T|}}}{T}  \frac{e^{-\sqrt{2|\ln(T - t)|}}}{|(r,z) - \xi|} \\
	&\lesssim \sqrt{T} \frac{e^{-a\sqrt{2|\ln T|}}}{T^{2}} \|\phi^{o}\|_{\star,o}.
\end{align*}

\medskip
\noindent
\textbf{Case 3:} If $|(r,z) - \xi| \ge M$, then
\begin{align*}
	|\nabla_{(r,z)}\phi^{o} \cdot \nabla_{(r,z)}\psi_{\lambda}(1 - \chi)|
	&\lesssim \|\varphi^{o}\|_{\star,o}  \frac{e^{-a\sqrt{2|\ln T|}}}{T} 
 \frac{e^{-\sqrt{2|\ln(T - t)|}}}{|(r,z) - \xi|}  e^{- \frac{|(r,z)|^{2}}{4(t+1)}} \\
	&\lesssim \frac{T}{M}  e^{-\sqrt{2|\ln T|}} \frac{e^{-a\sqrt{2|\ln T|}}}{T^{2}}  e^{- \frac{|(r,z)|^{2}}{4(t+1)}}\|\varphi^{o}\|_{\star,o}.
\end{align*}
Now, observing that $\varphi_{\lambda}$ and the inner solutions are localized, the remaining terms can be handled in an analogous manner. 
We conclude the argument by  taking $T$ sufficiently small and selecting the constants in accordance with the choices made in \cite{BDdPM}.

\end{proof}

\subsection{The multi-ring framework}\label{multiringsproof}
We omit most of the details, as they follow directly from our construction. To illustrate the methodology, we focus on the case $k = 2$. \newline
We define the cutoff functions
\begin{align}
	\chi_{j}(x,t) = \chi_{0}\left(\frac{|x - \xi_{j}(t)|}{\sqrt{\delta(T - t)}}\right), \quad j = 1, 2.
\end{align}
Recall that
\begin{align*}
	\xi_{1}(t) \to \xi_{1}(T) = (q^{(1)}_{1}(T), q^{(1)}_{2}(T)) \in \mathbb{R}^{+} \times \mathbb{R}, \quad
	\xi_{2}(t) \to \xi_{2}(T) = (q^{(2)}_{1}(T), q^{(2)}_{2}(T)) \in \mathbb{R}^{+} \times \mathbb{R}.
\end{align*}
As a first ansatz, we consider the decomposition
\begin{align*}
	u_{1}(x,t) = u_{1}^{(1)}(x,t) + u_{1}^{(2)}(x,t),
\end{align*}
where
\begin{align*}
	u_{1}^{(j)}(x,t) = \frac{\alpha_{j}(t)}{\lambda_{j}(t)^{2}} U\left(\frac{x - \xi_{j}(t)}{\lambda_{j}(t)}\right) \chi_{j}(x,t)
	+ \varphi_{\lambda_{j}}(|x - q_{j}|, t), \quad j = 1, 2.
\end{align*}
Minor modifications are needed in the proof of Theorem~\ref{Expansionv0} and Proposition~\ref{3dGRADIENTExpansion}. Specifically, the function $\omega_{1}(r,z)$ must be replaced by
\begin{align*}
	\omega^{(i)}(r,z) = \frac{2 q^{(i)}_{1}}{|(r,z) + \xi_{i}|}
	\left( 1 + \frac{1}{4} \frac{|(r,z) - \xi_{i}|^{2}}{|(r,z) + \xi_{i}|^{2}} \right),
\end{align*}
which satisfies the estimate
\begin{align*}
	|\omega^{(i)}(r,z) - 1| = O\left(|(r,z) - \xi_{i}|\right).
\end{align*}
We proceed following the approach described in Section \ref{innerouterSec}. We express
\begin{align*}
	\Phi(x,t) = \sum_{j=1,2} \left[ \frac{1}{\lambda_{j}^{2}} \phi^{i}_{(j)}(x,t)\chi_{j} + \varphi^{o}_{(j)}(x,t) \right],
\end{align*}
and we require that
\begin{align*}
	S(u_{1} + \Phi) = 0.
\end{align*}
To this end, we linearize the operator $S$ around $u_{1}^{(1)}$, yielding
\begin{align*}
	S(u_{1} + \Phi) = &\, S(u_{1}^{(1)}) - \partial_{t} \left( \frac{1}{\lambda^{2}_{1}} \phi^{i}_{(1)} \chi_{1} \right) - \partial_{t} \varphi^{o}_{(1)} + \mathcal{L}_{u_{1}^{(1)}}\left[ \frac{1}{\lambda^{2}_{1}} \phi^{i}_{(1)} \chi_{1} \right] + \mathcal{L}_{u_{1}^{(1)}}[ \varphi^{o}_{(1)} ] \\
	& - \partial_{t} \left( u_{1}^{(2)} + \frac{1}{\lambda^{2}_{2}} \phi^{i}_{(2)} \chi_{2} + \varphi^{o}_{(2)} \right) + \mathcal{L}_{u_{1}^{(1)}}\left[ u_{1}^{(2)} + \frac{1}{\lambda^{2}_{2}} \phi^{i}_{(2)} \chi_{2} + \varphi^{o}_{(2)} \right] \\
	& - \text{div}_{(r,z)} \left( (u_{1}^{(2)} + \Phi) \nabla_{(r,z} (-\Delta_{\mathbb{R}^{3}})^{-1}(u_{1}^{(2)} + \Phi) \right),
\end{align*}
where we recall the definition:
\begin{align*}
	\mathcal{L}_{u_{1}}[\phi] = \Delta_{(r,z)} \phi + \frac{1}{r} \partial_{r} \phi - \text{div}_{(r,z)}\left( \phi \nabla_{(r,z)} v_{1} \right) - \text{div}_{(r,z)}\left( u_{1} \nabla_{(r,z)} \psi \right), \quad \psi(r,z) = (-\Delta_{\mathbb{R}^{3}})(\phi(r,z)).
\end{align*}
Next, we observe that
\begin{align*}
	- \text{div}_{(r,z)} \left( (u_{1}^{(2)} + \Phi) \nabla_{(r,z)} (-\Delta_{\mathbb{R}^{3}})^{-1}(u_{1}^{(2)} + \Phi) \right) =
	&-\text{div}_{(r,z)} \left( u_{1}^{(2)} \nabla_{(r,z)} (-\Delta_{\mathbb{R}^{3}})^{-1} u_{1}^{(2)} \right)  \\
    &-\text{div}_{(r,z)} \left( u_{1}^{(2)} \nabla_{(r,z} (-\Delta_{\mathbb{R}^{3}})^{-1} \Phi \right) \\
	& - \text{div}_{(r,z)} \left( \Phi \nabla_{(r,z)} (-\Delta_{\mathbb{R}^{3}})^{-1} \Phi \right).
\end{align*}
We may also write
\begin{align*}
\mathcal{L}_{u_{1}^{(1)}}\left[ u_{(2)}^{i} + \frac{1}{\lambda_{i}^{2}} \phi_{i}^{(2)} \chi_{i} + \varphi^{o}_{(2)} \right] 
= \left( \Delta_{(r,z)} + \frac{1}{r} \partial_{r} \right) \left( u_{(2)}^{i} + \frac{1}{\lambda_{i}^{2}} \phi_{i}^{(2)} \chi_{i} + \varphi^{o}_{(2)} \right) + \mathcal{E}_{j,i}.
\end{align*}
Hence, we obtain the following expression:
\begin{align*}
	S(u_{1} + \Phi) = 
	& - \partial_{t} \left( \frac{1}{\lambda^{2}_{1}} \phi^{i}_{(1)} \chi_{1} \right) - \partial_{t} \varphi^{o}_{(1)} + \mathcal{L}_{u_{1}^{(1)}}\left[ \frac{1}{\lambda^{2}_{1}} \phi^{i}_{(1)} \chi_{1} \right] + \mathcal{L}_{u_{1}^{(1)}}[ \varphi^{o}_{(1)} ] \\
	& - \partial_{t} \left( \frac{1}{\lambda^{2}_{2}} \phi^{i}_{(2)} \chi_{2} \right) + \partial_{t} \varphi^{o}_{(2)} + \mathcal{L}_{u_{1}^{(2)}}\left[ \frac{1}{\lambda^{2}_{2}} \phi^{i}_{(2)} \chi_{2} \right] + \mathcal{L}_{u_{1}^{(2)}}[ \varphi^{o}_{(2)} ] \\
	& + \left( S(u_{1}^{(1)}) + S(u_{2}^{(2)}) + \mathcal{E}_{1,2} + \mathcal{E}_{2,1} - \text{div}_{(r,z)} ( \Phi \nabla_{(r,z)} (-\Delta_{\mathbb{R}^{3}})^{-1} \Phi ) \right) \sum_{i=1,2} \chi_{i} \\
	& + \left( S(u_{1}^{(1)}) + S(u_{2}^{(2)}) + \mathcal{E}_{1,2} + \mathcal{E}_{2,1} - \text{div}_{(r,z)} ( \Phi \nabla_{(r,z)} (-\Delta_\mathbb{R}^{3})^{-1} \Phi ) \right) \left( 1 - \sum_{i=1,2} \chi_{i} \right).
\end{align*}
It is evident that the resulting system will consist of two inner equations for $\phi_{(j)}^{i}$ and two outer equations for $\varphi^{o}_{(j)}$. Observe that the contribution from the mixed terms is negligible, as these involve factors that are localized far from the region of interest and, within the relevant domain, they only generate lower-order terms.

\end{document}